\documentclass[a4paper]{cedram-smai-jcm}
\usepackage{etex}

\usepackage{amsrefs, amsmath, amsthm, amssymb}
\usepackage{booktabs}
\usepackage{wasysym}
\usepackage{lscape}

\usepackage{subfig}

\allowdisplaybreaks

\usepackage{graphicx}
\graphicspath{{./figs/}{./newfigs/}}

\def\supp{\mathrm{supp}}
\def\rmd{\mathrm{d}}

\def\rmT{\textup{T}}
\def\area{\text{area}}
\def\vol{\text{vol}}

\def\vc{\boldsymbol{c}}
\def\vl{\boldsymbol{l}}
\def\ve{\boldsymbol{e}}
\def\vg{\boldsymbol{g}}
\def\vh{\boldsymbol{h}}
\def\vp{\boldsymbol{p}}
\def\vq{\boldsymbol{q}}
\def\vs{\boldsymbol{s}}
\def\vu{{\boldsymbol{u}}}
\def\vv{{\boldsymbol{v}}}
\def\vw{{\boldsymbol{w}}}
\def\vx{{\boldsymbol{x}}}
\def\vy{{\boldsymbol{y}}}
\def\mA{\boldsymbol{A}}
\def\mH{\boldsymbol{H}}
\def\mP{\boldsymbol{P}}
\def\mR{\boldsymbol{R}}
\def\mT{\boldsymbol{T}}
\def\mU{\boldsymbol{U}}
\def\mM{\boldsymbol{M}}

\def\valpha{\boldsymbol{\alpha}}
\def\vbeta{\boldsymbol{\beta}}
\def\vpsi{\boldsymbol{\psi}}
\def\vlambda{\boldsymbol{\lambda}}

\def\cA{\mathcal{A}}
\def\cG{\mathcal{G}}
\def\cH{\mathcal{H}}
\def\cK{\mathcal{K}}
\def\cS{\mathcal{S}}
\def\cU{\mathcal{U}}

\def\PP{\mathbb{P}}
\def\SS{\mathbb{S}}
\def\RR{\mathbb{R}}

\def\bfc{{\boldsymbol{c}}}
\def\bfe{{\boldsymbol{e}}}
\def\bff{{\boldsymbol{f}}}
\def\bfi{{\boldsymbol{i}}}
\def\bfj{{\boldsymbol{j}}}
\def\bfk{{\boldsymbol{k}}}
\def\bfp{{\boldsymbol{p}}}
\def\bfq{{\boldsymbol{q}}}
\def\bfu{{{\boldsymbol{u}}}}

\def\bfx{{\boldsymbol{x}}}
\def\bfy{{\boldsymbol{y}}}
\def\bfone{{\boldsymbol{1}}}
\def\bfxi{{ {\boldsymbol{\xi}} }}
\def\bfbeta{{ {\boldsymbol{\beta}} }}
\def\bfalpha{{ {\boldsymbol{\alpha}} }}
\def\bfA{{\boldsymbol{A}}}

\usepackage{tikz}

\def\PS{
\begin{tikzpicture}[scale = 1.7]
\draw[line width = 0.5] (-0.1, 0.0) -- (0.1,0.0);
\draw[line width = 0.5] (-0.1, 0.0) -- (0.0,0.1732);
\draw[line width = 0.5] ( 0.1, 0.0) -- (0.0,0.1732);
\end{tikzpicture}
}

\def\PSsmall{
\begin{tikzpicture}[scale = 1.3]
\draw[line width = 0.5] (-0.1, 0.0) -- (0.1,0.0);
\draw[line width = 0.5] (-0.1, 0.0) -- (0.0,0.1732);
\draw[line width = 0.5] ( 0.1, 0.0) -- (0.0,0.1732);
\end{tikzpicture}
}

\def\PSB{
\begin{tikzpicture}[scale = 1.7]
\draw[line width = 0.5] (-0.1, 0.0) -- (0.1,0.0);
\draw[line width = 0.5] (-0.1, 0.0) -- (0.0,0.1732);
\draw[line width = 0.5] ( 0.1, 0.0) -- (0.0,0.1732);

\draw[very thin]  (-0.05, 0.5*0.1732) -- (0.05, 0.5*0.1732);
\draw[very thin]  (-0.05, 0.5*0.1732) -- (0.0, 0.0);
\draw[very thin]  (0.0, 0.0) -- (0.05, 0.5*0.1732);

\draw[very thin]  (0.0, 0.0) -- (0.0,0.1732);
\draw[very thin] (-0.1, 0.0) -- (0.05, 0.5*0.1732);
\draw[very thin] (-0.05, 0.5*0.1732) -- (0.1,0.0);
\end{tikzpicture}
}

\def\PSC{
\begin{tikzpicture}[scale = 1.7]
\draw[line width = 0.5] (-0.1, 0.0) -- (0.1,0.0);
\draw[line width = 0.5] (-0.1, 0.0) -- (0.0,0.1732);
\draw[line width = 0.5] ( 0.1, 0.0) -- (0.0,0.1732);

\draw[very thin]  (0.0, 0.0) -- (0.0,0.1732);
\draw[very thin] (-0.1, 0.0) -- (0.05, 0.5*0.1732);
\draw[very thin] (-0.05, 0.5*0.1732) -- (0.1,0.0);
\end{tikzpicture}
}

\def\PSH{
\begin{tikzpicture}[scale = 1.7]
\draw[line width = 0.5] (-0.1, 0.0) -- (0.1,0.0);
\draw[line width = 0.5] (-0.1, 0.0) -- (0.0,0.1732);
\draw[line width = 0.5] ( 0.1, 0.0) -- (0.0,0.1732);

\draw[very thin]  (0.0, 0.05773) -- (0.0,0.1732);
\draw[very thin] (-0.1, 0.0) -- (0.0, 0.05773);
\draw[very thin] (0.0, 0.05773) -- (0.1,0.0);
\end{tikzpicture}}

\newcommand{\SimS}[1]{\raisebox{-0.75em}{\includegraphics[scale=0.54, clip = true, trim = 2 0 1 0]{SimplexSplineSmall#1.pdf}}}
\newcommand{\SimSgeneric}{\raisebox{-0.75em}{\includegraphics[scale=0.54]{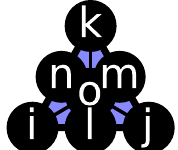}}}
\newcommand{\SimSconstant}[1]{\raisebox{-0.75em}{\includegraphics[scale=0.54, clip = true, trim = 7 0 7 7]{SimplexSplineSmallConstant#1.pdf}}}
\newcommand{\SmpS}[1]{\raisebox{-0.75em}{\includegraphics[scale=0.06]{ss-#1.pdf}}}

\title[]{B-spline-like bases for $C^2$ cubics on the Powell-Sabin 12-split}

\author[T. Lyche]{\firstname{Tom} \lastname{Lyche}}
\address{University of Oslo, Department of Mathematics, P.O. Box 1053, Blindern, NO-0316, Oslo, Norway}
\email{tom@math.uio.no}

\author[G. Muntingh]{\firstname{Georg} \lastname{Muntingh}}
\address{SINTEF Digital, Department of Mathematics and Cybernetics, P.O. Box 124 Blindern, NO-0314, Oslo, Norway}
\email{georg.muntingh@sintef.no}

  

\begin{document}
\maketitle

\begin{abstract}
For spaces of constant, linear, and quadratic splines of maximal smoothness on the Powell-Sabin 12-split of a triangle, the so-called S-bases were recently introduced. These are simplex spline bases with B-spline-like properties on the 12-split of a single triangle, which are tied together across triangles in a B\'ezier-like manner.

In this paper we give a formal definition of an S-basis in terms of certain basic properties. We proceed to investigate the existence of S-bases for the aforementioned spaces and additionally the cubic case, resulting in an exhaustive list. From their nature as simplex splines, we derive simple differentiation and recurrence formulas to other S-bases. We establish a Marsden identity that gives rise to various quasi-interpolants and domain points forming an intuitive control net, in terms of which conditions for $C^0$-, $C^1$-, and $C^2$-smoothness are derived.
\end{abstract}

\section{Introduction}
\subsection{Motivation}
Piecewise polynomials, or \emph{splines}, defined over triangulations have applications in many branches of science, ranging from scattered data fitting to finding numerical solutions to partial differential equations. See \cite{Lai.Schumaker07,Ciarlet.78} for comprehensive monographs. 

In applications like geometric modelling \cite{Cohen.Riensenfeld.Elber01} and solving PDEs by isogeometric methods \cite{Hughesbook} one often desires a low degree spline with higher smoothness. For a general triangulation, it was shown in Theorem~1.(ii) of \cite{Zenisek.74} that the minimal degrees of a triangular $C^1$ and $C^2$ element are $5$ and $9$, respectively. To obtain smooth splines of lower degree one can split each triangle in the triangulation into several subtriangles. Three such splits are the Clough-Tocher split $\PSH$ (CT), the Powell-Sabin 6-split $\PSC$ (PS6) and 12-split $\PSB$ (PS12) of a triangle~$\PS$, with 3, 6, and 12 subtriangles, respectively.
On these splits global $C^1$-smoothness can be obtained with degree 3 for CT, degree 2 and 3 for  PS6  and degree 2 for PS12 \cite{CT65,GS2017,Powell.Sabin77}. $C^2$-smoothness is achieved for PS6 and PS12 using degree $5$; see \cite{Lai.Schumaker03,Schumaker.Sorokina06,Speleers10} on a general (planar) triangulation. 

To compute with splines we need a basis for the spline space. In the univariate case B-splines have many advantages. They lead to banded matrices with good stability properties for low degrees and can be computed efficiently using stable recurrence relations. We would like similar bases for splines on triangulations.  In \cite{Cohen.Lyche.Riesenfeld13} a basis, called the \emph{S-basis}, was introduced for $C^1$ quadratics on the PS12-split. The S-basis consists of simplex splines \cite{Micchelli79, Prautzsch.Boehm.Paluszny02} and has all the usual properties of univariate B-splines, including a recurrence relation down to piecewise linear polynomials and a Marsden identity. Global $C^1$-smoothness is achieved by connecting neighboring triangles using classical B\'ezier techniques. This basis has been applied for swift assembly of the stiffness matrices in the finite element method \cite{Stangeby18}.

For a quintic B-spline-like basis on the PS12-split see \cite{Lyche.Muntingh14, Lyche.Muntingh16}. This basis has $C^3$ supersmoothness on each macro triangle and has global $C^2$-smoothness. Moreover, in addition to giving $C^1$- and $C^2$-smooth spaces on any triangulation, these spaces on the PS12-split are suitable for multiresolution analysis \cite{Davydov.Yeo13, DynLyche98, Lyche.Muntingh14, Oswald92}. A similar B-spline-like simplex basis has been constructed on the CT-split \cite{Lyche.Merrien18}, while for the PS6-split a B-spline basis has been developed for the $C^1$-smooth quadratics and cubics and $C^2$-smooth quintics \cite{D97,GS2017,Speleers10}. The latter bases have many of the nice B-spline properties, but have to be computed by conversion to Bernstein bases on each subtriangle.

In this paper we systematically enumerate the simplex splines and determine the possible S-bases for the spaces of $C^{d-1}$ splines of degree $d$ on the PS12-split for $d=0,1,2,3$. In the cubic case and for a general triangulation, we argue that these cannot be extended to globally smooth bases. Instead, we envision applications for local constructions, such as hybrid meshes and extra-ordinary points, which are important issues in isogeometric analysis.

\subsection{Main result}\label{sec:mainresult}
For $d = 0,1,2,3$, we consider the space $\SS_d(\PSB)$ of $C^{d-1}$-smooth splines of degree $d$ on the Powell-Sabin 12-split $\PSB$ of a triangle $\PS$ (see definition below). We consider bases $\vs$ of $\SS_d(\PSB)$ satisfying the following properties:
\begin{enumerate}
\item[P1] $\vs$ is invariant under the dihedral symmetry group $\cG$ of the equilateral triangle (cf. \S \ref{sec:symmetries}).
\item[P2] $\vs$ reduces to a shared B-spline basis on the boundary (cf. Remark \ref{rem:P2}).
\item[P3] $\vs$ forms a positive partition of unity and satisfies a Marsden identity, for which the dual polynomials only have real linear factors (cf. \S \ref{sec:Marsden}).
\item[P4] $\vs$ has all its domain points inside $\PS$, with precisely $d+2$ domain points (counting multiplicities) on each edge of $\PS$ (cf. Figure \ref{fig:domain_mesh}).
\item[P5] $\vs$ admits a stable recurrence relation (cf. \S \ref{sec:S-bases}).
\item[P6] $\vs$ admits a differentiation formula (cf. \S \ref{sec:differentiation}).
\item[P7] $\vs$ comes equipped with quasi-interpolants (cf. \S \ref{sec:quasi-interpolant}).
\item[P8] $\vs$ can be smoothly tied together across adjoining triangles using B\'ezier-type conditions (cf. \S \ref{sec:SmoothnessConditions}).
\end{enumerate}
In addition some of the bases $\vs$ satisfy:
\begin{enumerate}
\item[P9] $\vs$ has local linear independence.
\end{enumerate}

We call any basis for $\SS_d(\PSB)$ satisfying P1--P8 an \emph{S-basis}. This space has dimension $n_d$ as in \eqref{eq:nd} and simplex spline bases
\begin{equation}\label{eq:SS-bases}
\vs_d^\rmT = \big[S_{1,d},\ldots,S_{n_d,d}\big],\ d = 0,1,2,3,\qquad \tilde{\vs}_d^\rmT = \big[\tilde{S}_{1,d},\ldots,\tilde{S}_{n_d,d}\big],\ d = 2,3,
\end{equation}
listed in Table \ref{tab:dualpolynomials}. Generalizing a similar result for the bases $\vs_0, \vs_1, \vs_2$ in \cite{Cohen.Lyche.Riesenfeld13}, the main result of the paper is the following.

\begin{theorem}\label{thm:MainTheorem}
The sets $\vs = \vs_d^\rmT, \tilde{\vs}_d^\rmT$ as in \eqref{eq:SS-bases} are the only simplex spline bases for $\SS_d(\PSB)$ satisfying P1--P4. Moreover, these bases also satisfy P5--P8, while only $\vs_0,\vs_1,$ and $\vs_2$ satisfy P9.
\end{theorem}

Theorem \ref{thm:MainTheorem} is known \cite{Cohen.Lyche.Riesenfeld13} to hold for the constant, linear, and quadratic bases $\vs_0, \vs_1, \vs_2$. For the remaining bases $\tilde{\vs}_2, \vs_3, \tilde{\vs}_3$, Theorem~\ref{thm:MainTheorem} is established in this paper by showing in the coming sections that Properties P1--P8 hold for these bases only, and that Property P9 does not hold (Remarks \ref{rem:AlternativeQuadratic}, \ref{rem:s3P9}, \ref{rem:AlternativeCubic}). Property~P1 is imposed by only including entire $\cG$-equivalence classes of simplex splines. It ensures that basic properties of the basis are left invariant under affine transformations. Property P2 emerges from the analysis in Section~\ref{sec:simplexenumeration}. Properties P1 and P2 significantly reduce the number of cases to be considered. The Marsden identity of Property P3 is established in Section \ref{sec:Marsden}. It gives rise to Property P4 through its explicit form (Table~\ref{tab:dualpolynomials}) and Property P7 through an explicit quasi-interpolant (Theorem~\ref{thm:QI}). Properties P5 and P6 follow from basic properties of simplex splines. Properties P2--P4 allow to establish a B\'ezier-like control net, which, together with Property P6, yields the B\'ezier-type smoothness conditions of Property P8. Remarks \ref{rem:OtherLinear}--\ref{rem:OtherCubic} explain why there are no other bases with these properties. 

Supplementary computational results are presented in a Jupyter notebook \cite{WebsiteGeorg}.

\begin{figure}
\centerline{
\includegraphics[scale=0.8]{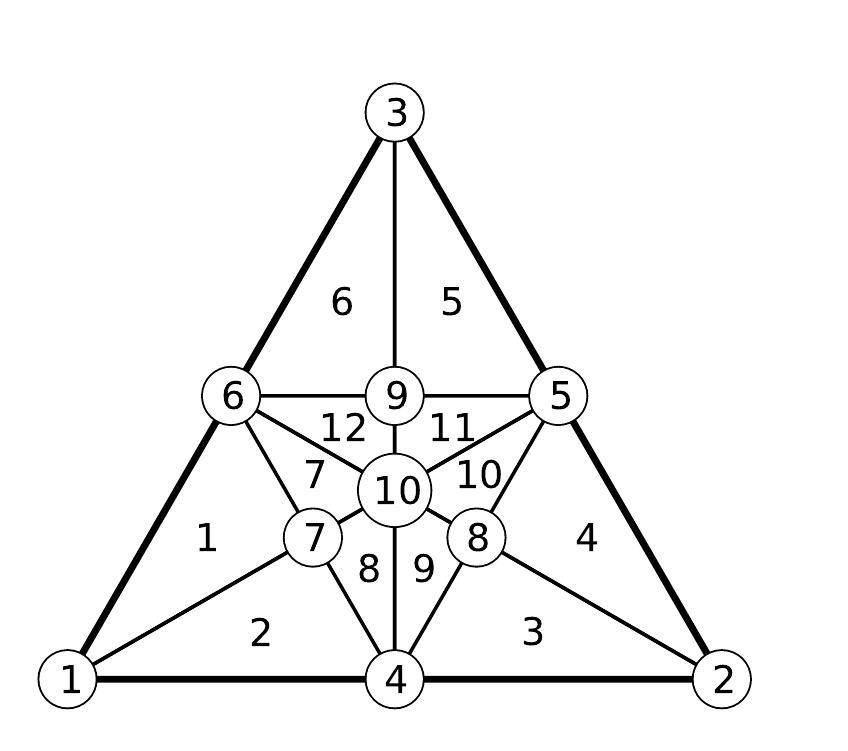}\qquad
\includegraphics[scale=0.8]{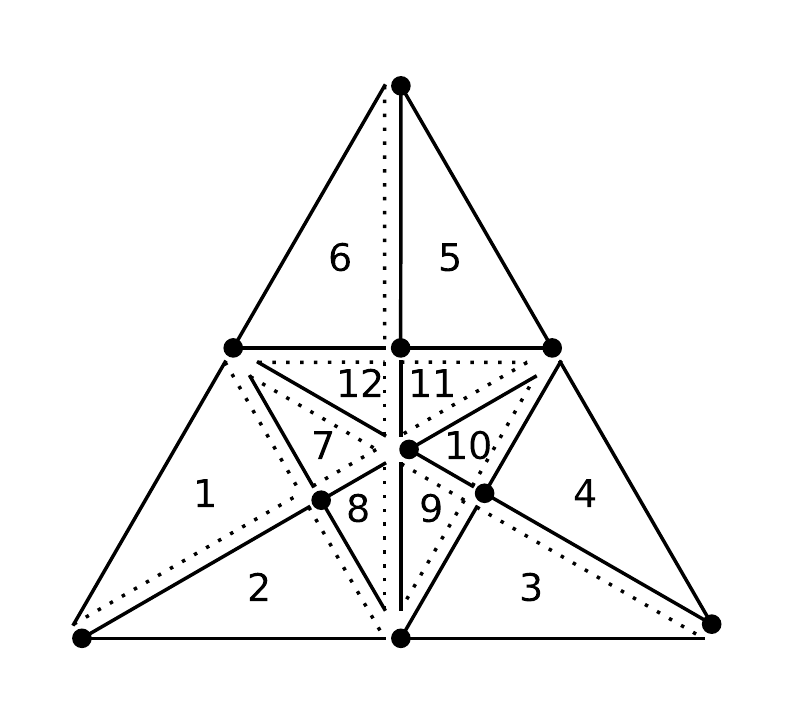}}
\caption{The Powell-Sabin 12-split with numbering of vertices and subtriangles (left), and a scheme assigning every point in the triangle to a unique subtriangle of the 12-split (right).}
\label{fig:ps12-numbering}
\end{figure}

\subsection{Basic tools}\label{sec:basictools}
We recall some basic tools used throughout the paper.

\subsubsection{Conventions}\label{sec:notation}
We use small Greek letters (e.g.~$\alpha, \beta$) for scalar values, small boldface letters (e.g.~$\vs$) to denote vectors, capital boldface letters (e.g. $\mR, \mT, \mU$) for matrices. Scalar-valued univariate functions are denoted by small letters, scalar-valued multivariate functions are denoted by capital letters (e.g.~$S, M, Q$), while vector-valued multivariate are, like matrices, denoted by capital boldface letters. Calligraphic fonts (e.g. $\cK$) are reserved for (multi)sets, expressed as
\[ \cK = \{\underbrace{\bfk_1,\ldots,\bfk_1}_{\mu_1},\ldots, \underbrace{\bfk_s,\ldots,\bfk_s}_{\mu_s}\} = \{\bfk_1^{\mu_1}\cdots \bfk_s^{\mu_s}\}, \]
with $\mu_i$ the \emph{multiplicity} of $\bfk_i$. The \emph{size} $|\cK|$ of $\cK$ is its number of elements counting multiplicities, i.e., $|\cK| = \mu_1 + \cdots + \mu_s$. Generalizing the notation for closed and half-open intervals, we write $[\cK]$ for the convex hull of $\cK$. Whenever $\cK$ consists of vertices of $\PSB$, we write $[\cK)$ for the half-open convex hull of $\cK$ obtained as union of the half-open subtriangles shown in Figure~\ref{fig:ps12-numbering} (right).

Blackboard bold (e.g. $\PP_d, \SS_d$) is used to denote function spaces. In particular, identifying matrices with linear maps, the symbol $\RR^{m,n}$ denotes the space of $m\times n$ real matrices. We identify $\RR^m$ with $\RR^{m,1}$ (column vectors), and denote the standard basis vectors in $\RR^m$ by $\ve_1,\ldots,\ve_m$.

For an $m\times n$ matrix $\mA$ and $\bfi = [i_1,\ldots,i_r]^\rmT$, $\bfj = [j_1,\ldots,j_s]^\rmT$ with $1\leq i_1 < \cdots < i_r \leq m, 1\leq j_1 \leq \cdots \leq j_s \leq n$, then $\mA(\bfi,\bfj)$ is the $r\times s$ matrix whose $(k,\ell)$ element is $a_{i_k,j_\ell}$. In particular, $\vc(\bfi)$ denotes the vector whose $j$th element is $c_{i_j}$. 

The support of a function $F$, denoted by $\supp(F)$, is the closure of the set of values in the domain of $F$ at which $F$ is nonzero. Empty products are assumed to be 1. For any set $\cA$, the indicator function on $\cA$ is denoted by $\bfone_\cA$.

\subsubsection{The Powell-Sabin 12-split}
Consider the triangle $\PS$ with vertices $\vp_1,$ $\vp_2,$ $\vp_3 \in \RR^2$ and midpoints
\begin{equation}\label{eq:p456}
\vp_4:= \frac{\vp_1+\vp_2}{2},\qquad \vp_5:=\frac{\vp_2+\vp_3}{2},\qquad \vp_6:=\frac{\vp_3+\vp_1}{2}.
\end{equation}
Taking the complete graph on these six points, one obtains additional points
\begin{equation}\label{eq:p78910}
\vp_7 := \frac{\vp_4+\vp_6}{2},\qquad \vp_8:= \frac{\vp_4+\vp_5}{2},\qquad \vp_9:= \frac{\vp_5+\vp_6}{2},\qquad \vp_{10}:= \frac{\vp_1+\vp_2+\vp_3}{3}
\end{equation}
and subtriangles $\PS_1,\ldots,\PS_{12}$ as in Figure~\ref{fig:ps12-numbering}. The resulting split is called the \emph{Powell-Sabin 12-split} $\PSB$ of $\PS$.

\subsubsection{Barycentric and directional coordinates}
The \emph{barycentric coordinates} $\vbeta = (\beta_1,\beta_2,\beta_3)$ of a point $\bfx\in \RR^2$ with respect to the triangle $\PS = [\bfp_1,\bfp_2,\bfp_3]$ is the unique solution to
\begin{equation}\label{eq:BarycentricCoordinates}
\vx = \beta_1\vp_1 + \beta_2\vp_2 + \beta_3\vp_3,\qquad 1 = \beta_1 + \beta_2 + \beta_3.
\end{equation}
Similarly, we write $\vbeta^{i,j,k} = (\beta_1^{i,j,k}, \beta_2^{i,j,k}, \beta_3^{i,j,k})$ for the barycentric coordinates of $\bfx$ with respect to the triangle $[\bfp_i, \bfp_j, \bfp_k] \subset\PS$. To save space in the recursion and differentiation matrices, we use the short-hands
\begin{equation}\label{eq:gammabetasigma}
\gamma_j:=2\beta_j-1,\quad \beta_{i,j}=\beta_i-\beta_j,
\quad \sigma_{i,j}=\beta_i+\beta_j,\qquad \text{ for $i,j=1,2,3$}.
\end{equation}
Note that
\begin{equation}
\begin{aligned}\label{eq:positivegamma}
\gamma_j & \geq 0\ \text{at}\ \PS_i,\qquad i = 2j-1, 2j, \\
\gamma_j & \leq 0\ \text{at}\ \PS_i,\qquad i \neq 2j-1, 2j.
\end{aligned}
\end{equation}

For any $\vu=[u_1,u_2]^\rmT\in\RR^2$, consider the corresponding directional derivative $D_\vu:=\vu\cdot\nabla=u_1\frac{\partial}{\partial x_1} + u_2\frac{\partial}{\partial x_2}$. The unique solution $\valpha:=[\alpha_1,\alpha_2,\alpha_3]^\rmT$ of
\begin{equation}\label{eq:dircoef}
\vu = \alpha_1\vp_1+\alpha_2\vp_2+\alpha_3\vp_3,\qquad 0 = \alpha_1+\alpha_2+\alpha_3,
\end{equation}
is called the \emph{directional coordinates} of $\vu$. If $\bfu=\bfq^1-\bfq^2$, with $\bfq^1, \bfq^2\in\RR^2$, then $\alpha_j:=\beta^1_j - \beta^2_j$, $j=1,2,3$, where $\vbeta^i:=[\beta^i_1,\beta^i_2,\beta^i_3]^\rmT$ is the vector of barycentric coordinates of $\vq^i$, $i=1,2$.

\subsubsection{Function spaces}
Let $\PP_d (\RR^2)$ denote the space of bivariate polynomials of total degree at most $d$ and with real coefficients, which has dimension $\nu_d := (d+1)(d+2)/2$.
On a triangle $\PS$ with barycentric coordinates $\beta_1,\beta_2,\beta_3$, a convenient basis of $\PP_d$ is formed by the \emph{Bernstein polynomials}
\begin{equation}\label{eq:Bernstein}
B^d_{i_1,i_2,i_3} := \frac{d!}{i_1!i_2!i_3!} \beta_1^{i_1} \beta_2^{i_2} \beta_3^{i_3}, \qquad i_1 + i_2 + i_3 = d.
\end{equation}

Analogously, on the 12-split $\PSB$ of a triangle $\PS$, we consider the spaces
\begin{equation}
\SS_d (\PSB) := \{f\in C^{d-1}(\PS) : f|_{\PSsmall_k} \in \PP_d,\ k = 1,\ldots,12 \}, \quad d=0,1,2,\ldots
\end{equation}
The dimension $n_d$ of this space is \cite[Theorem 3]{Lyche.Muntingh16}
\begin{equation}\label{eq:nd}
(n_0, n_1, n_2, n_3,\ldots) = (12, 10, 12, 16,\ldots),\qquad n_d = \frac12 d^2 + \frac32 d + 7,\qquad d\geq 2.
\end{equation}
For $d = 0,1,2$, we equip these spaces with the S-bases $\vs^\rmT_d = [S_{j,d}]_{j=1}^{n_d}$ presented in \cite{Cohen.Lyche.Riesenfeld13}. 

Each piecewise polynomial on $\PSB$ can be represented as an element of the $\PP_d$-module $\PP_d^{12}$, i.e., as a vector with components the polynomial pieces on the faces $\PS_k$ of $\PSB$.

\section{Simplex splines}\label{sec:SimplexSplines}
In this section we recall the definition and some basic properties of simplex splines, and determine a list of all simplex splines in $\SS_d(\PSB)$ for $d=0,1,2,3$.
\subsection{Definition and properties}\label{sec:SimplexSplinesDefProp}
First we provide the definition of simplex spline convenient for our purposes, and recall properties necessary for the remainder of the paper.

\subsubsection{Geometric construction}
Let $\bfk_1,\ldots,\bfk_{d+3} \subset \RR^2$ be a sequence of points in the plane, called \emph{knots}, defining a multiset $\cK$. Let $\sigma = \big[\overline\bfk_1,\ldots,\overline\bfk_{d+3}\big] \subset \RR^{d+2}$ be a simplex whose projection $\mP: \RR^{d+2} \longrightarrow \RR^2$ onto the first two coordinates satisfies $\mP(\overline\bfk_i) = \bfk_i$, for $i = 1, \ldots, d+3$. For any integer $k\geq 1$, let $\vol_k$ denote the $k$-dimensional volume. For $k=2$ we simply write $\area := \vol_2$, to be understood in the usual sense. We define the \emph{integral normalized simplex spline}
\[ M[\cK]: \RR^2\longrightarrow \RR,\qquad M[\cK](\bfx) := \frac{\vol_d \big(\sigma\cap \mP^{-1}(\bfx)\big)}{\vol_{d+2}(\sigma)}. \]
This is well defined, independently of the choice of the simplex $\sigma$ \cite[\S 18.3]{Prautzsch.Boehm.Paluszny02}.

We will restrict ourselves to simplex splines on the 12-split $\PSB$ of a triangle $\PS$, in which case $\cK = \{\bfp_1^{\mu_1} \cdots \bfp_{10}^{\mu_{10}}\}$.  While $M[\cK]$ is the simplex spline most commonly encountered in the literature, our discussion is simpler in terms of the \emph{(area normalized) simplex spline}, defined as 
\[ Q[\cK] := \area(\PS) \cdot {|\cK| - 1 \choose 2}^{-1} M[\cK]. \]
Whenever $\mu_7 = \mu_8 = \mu_9 = 0$, we use the graphical notation
\[ \SimSgeneric := Q[\bfp_1^i \bfp_2^j \bfp_3^k \bfp_4^l \bfp_5^m \bfp_6^n \bfp_{10}^o]. \]

\begin{figure}
\centering{
\includegraphics[scale=0.2]{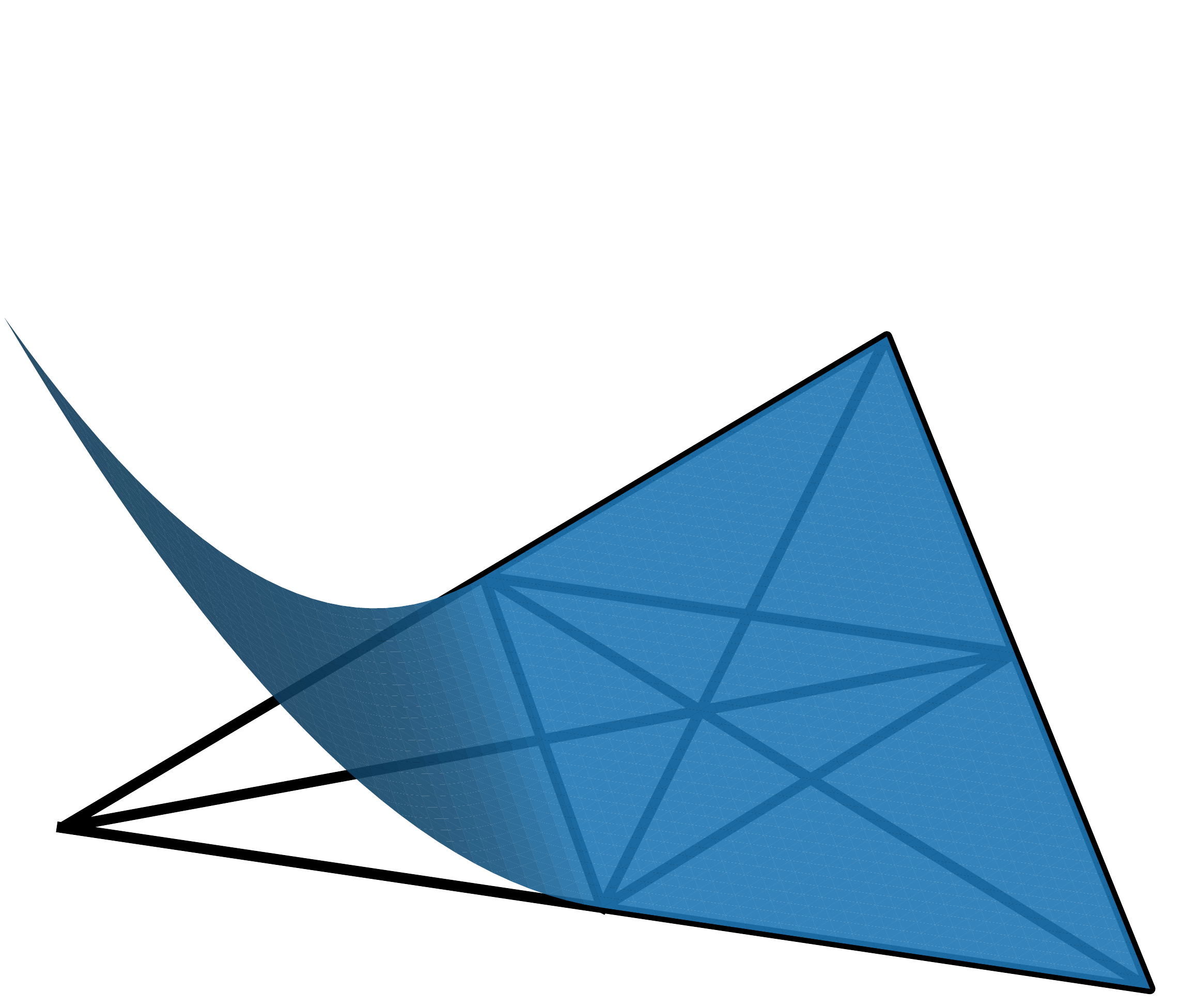}\quad
\includegraphics[scale=0.2]{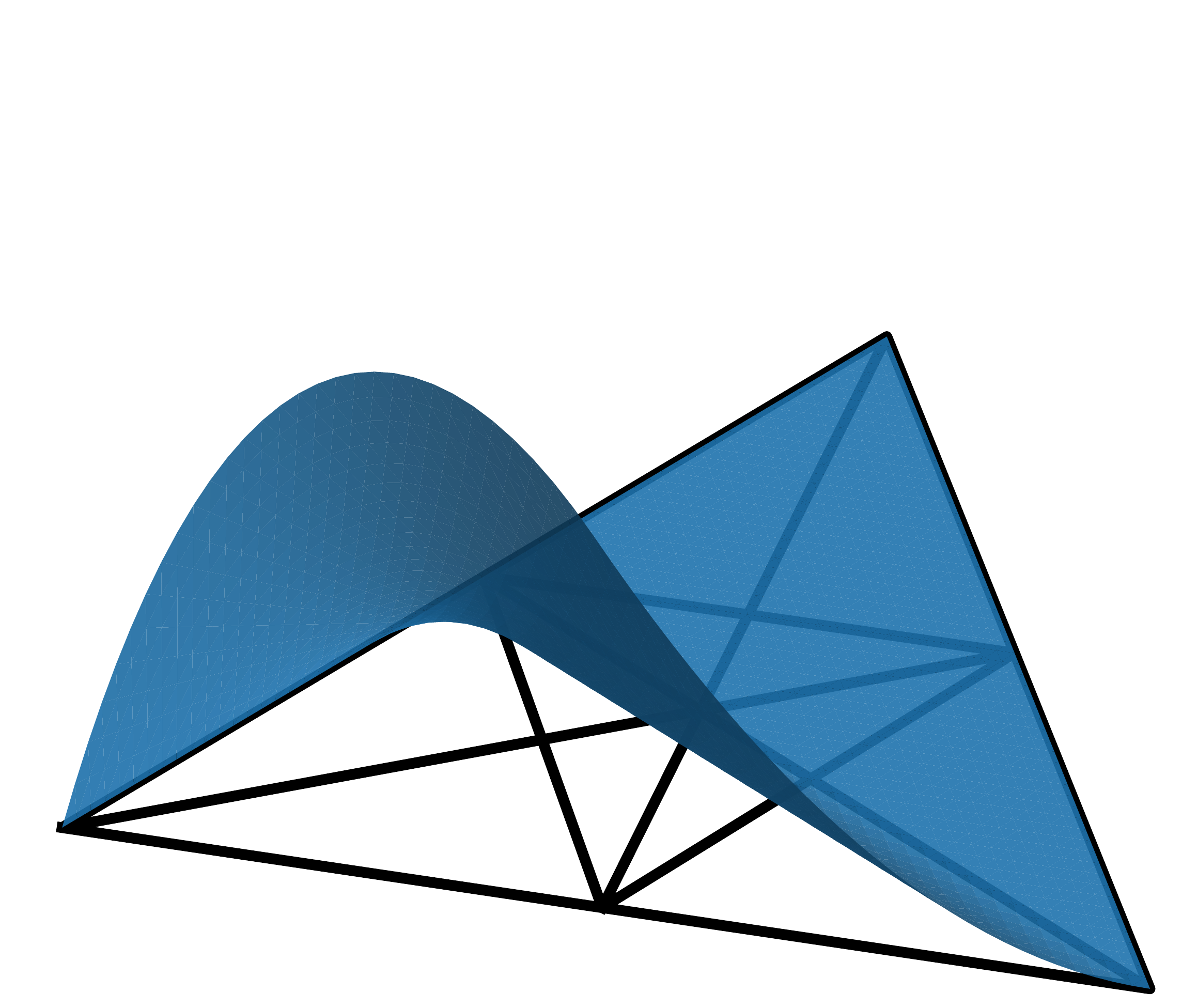}\quad
\includegraphics[scale=0.2]{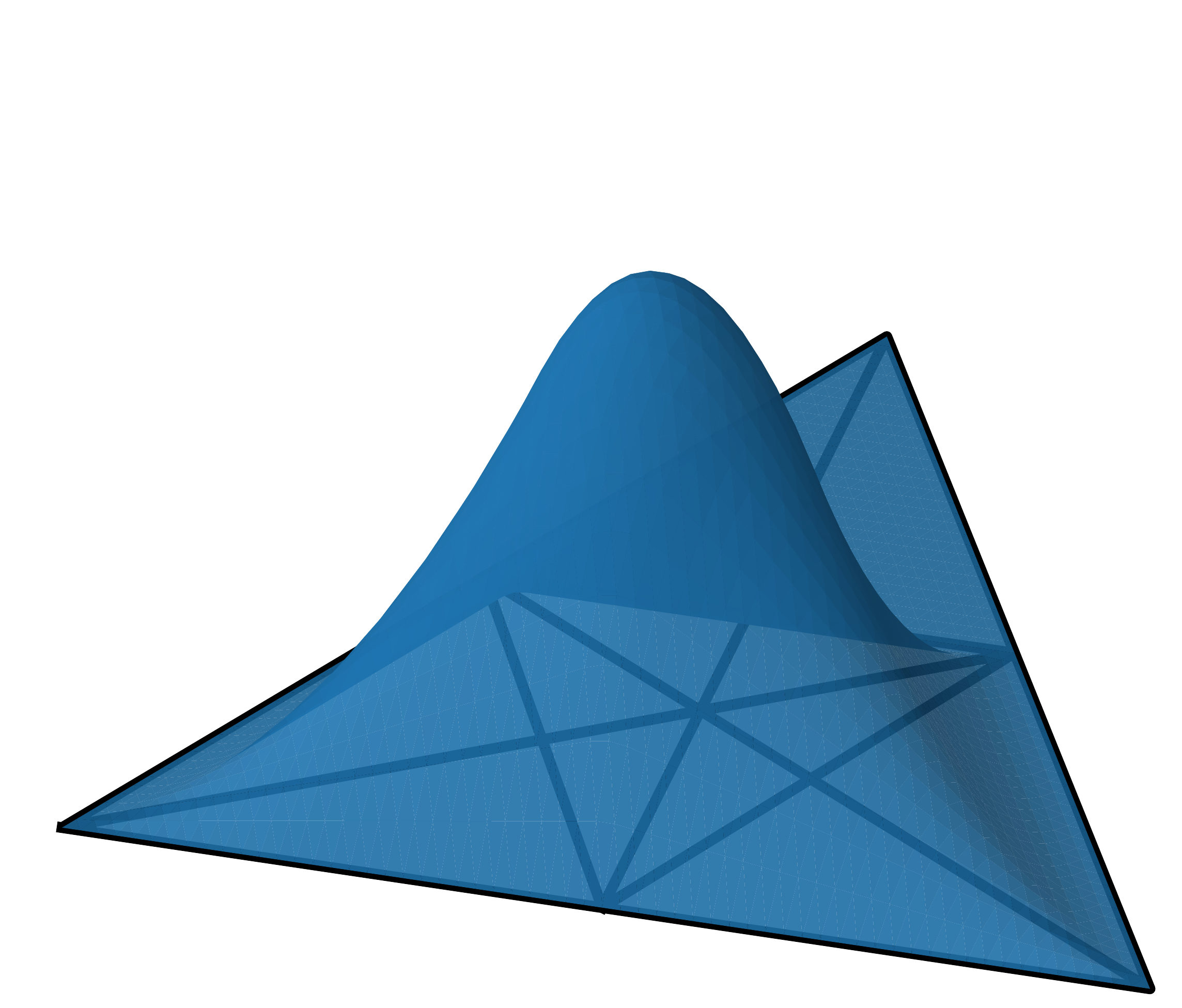}}
\caption{From left to right, the simplex splines $\frac14 Q[\bfp_1^3\bfp_4\bfp_6]$, $\frac12 Q[\bfp_1^2\bfp_2\bfp_4\bfp_6]$, and $\frac34 Q[\bfp_1\bfp_2\bfp_4\bfp_5\bfp_6]$ in $\vs_2$.}
\label{fig:SimplexSplineDef6}
\end{figure}

\subsubsection{Piecewise polynomial structure} $Q[\cK]$ is a piecewise polynomial on the convex hull $[\cK]$ of $\cK$, with knot lines formed by the complete graph of $\cK$ \cite[\S 18.5]{Prautzsch.Boehm.Paluszny02}; see Figure \ref{fig:SimplexSplineDef6}.
\subsubsection{Degree} Each polynomial piece of $Q[\cK]$ has total degree bounded as \cite[\S 18.5]{Prautzsch.Boehm.Paluszny02}
\begin{equation}\label{eq:MultiplicitySumGeneral}
\deg\,Q[\cK] \leq d := |\cK| - 3.
\end{equation}
\subsubsection{Smoothness} The smoothness across a knot line can be controlled locally. More precisely,
for any $\bfx\in \PS$, let $\mu$ be the maximum number of knots of $\cK$ (counting multiplicities), at least two of which are distinct, along any affine line containing $\bfx$. Then $Q[\cK]$ will have continuous derivatives up to order $d+1-\mu$ at $\bfx$ \cite[\S 18.6]{Prautzsch.Boehm.Paluszny02}, which we will express with the notation
\begin{equation}\label{eq:InteriorSmoothnessGeneral}
Q[\cK] \in C^{d+1-\mu}\ \text{at}\ \bfx.
\end{equation}
For example, if $Q[\cK]$ is a $C^{d-1}$-smooth simplex spline of degree $d$, then any line segment in $\PS$ can contain at most two distinct knots.

\subsubsection{Recursion} For $d\ge 1$, the simplex spline can be expressed in terms of simplex splines of lower degree \cite[\S 18.5]{Prautzsch.Boehm.Paluszny02},
\begin{equation}\label{eq:Qrec}
Q[\cK](\bfx)=\sum_{j=1}^{d+3} \beta_j Q[\cK\setminus \bfk_j](\bfx), \qquad \sum_j \beta_j = 1, \qquad \sum_j \beta_j\bfk_j=\bfx \in \RR^2.
\end{equation}
For simplex splines with knot multiset $\cK = \{\bfp_1^{\mu_1}\cdots \bfp_{10}^{\mu_{10}}\} \subset \RR^2$ composed of vertices of $\PSB$, we can therefore equivalently define $Q[\cK]$ recursively by
\begin{equation}\label{eq:Qrec2}
Q[\cK](\bfx) := \left\{
\begin{array}{cl}
0 & \text{if }\area([\cK]) = 0,\\
\bfone_{[\cK)} (\bfx) \frac{\area(\PSsmall)}{\area([\cK])} & \text{if } \area([\cK]) \neq 0 \text{ and } |\cK| = 3,\\
\sum_{j=1}^{10} \beta_j Q[\cK \backslash \bfk_j] (\bfx) & \text{if } \area([\cK]) \neq 0 \text{ and } |\cK| > 3,
\end{array} \right.
\end{equation}
with $\bfx = \beta_1 \bfp_1 + \cdots + \beta_{10}\bfp_{10}$, $\beta_1 + \cdots + \beta_{10} = 1$, and $\beta_i = 0$ whenever $\mu_i = 0$. For instance, with $\beta_1,\beta_2,\beta_3$ the barycentric coordinates of $\PS$,
\[ \SimS{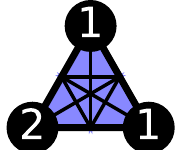} = \beta_1\cdot \bfone_{\PSsmall} + \beta_2 \cdot 0 + \beta_3 \cdot 0 = \beta_1 \cdot \bfone_{\PSsmall}. \]

\subsubsection{Differentiation}
When it is defined, the directional derivative of the simplex spline of degree $d$ can be expressed in terms of simplex splines of lower degree \cite[\S 18.6]{Prautzsch.Boehm.Paluszny02}, 
\begin{equation}\label{eq:Qdiff}
D_\bfu Q[\cK] = d\sum_{j=1}^{d+3} \alpha_j Q[\cK\setminus \bfk_j],\ \qquad \sum_j \alpha_j=0,\qquad \sum_j \alpha_j\bfk_j=\bfu \in \RR^2.
\end{equation}
For instance, with $\alpha_1,\alpha_2,\alpha_3$ directional coordinates of $\bfu$ with respect to the triangle $\PS$,
\[ \frac13 D_\bfu \SimS{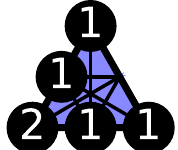} = \alpha_1 \SimS{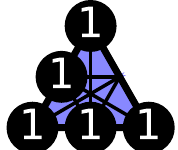} + \alpha_2 \SimS{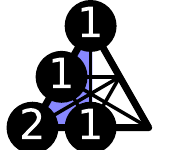} + \alpha_3 \SimS{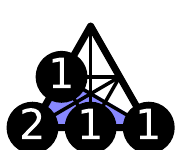}.
\]

\subsubsection{Knot insertion}
The simplex spline admits the knot insertion formula \cite[\S 18.4]{Prautzsch.Boehm.Paluszny02}
\begin{equation}\label{eq:Qins}
Q[\cK]=\sum_{j=1}^{d+3} c_j Q[(\cK\sqcup \bfy)\setminus \bfk_j],\qquad \sum_j c_j=1,\qquad \sum_j c_j\bfk_j=\bfy\in \RR^2.
\end{equation}
For instance, repeatedly applying knot insertion at the midpoints $\bfp_k = c_i \bfp_i + c_j \bfp_j$, $c_i = c_j = \frac12$, at the cost of the end points $\bfp_i, \bfp_j \in \{\bfp_1, \bfp_2, \bfp_3\}$,
\begin{align}\label{eq:222000to111111}
\SimS{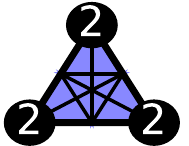} &= \frac12 \SimS{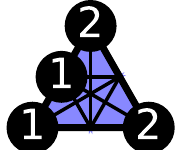} + \frac12 \SimS{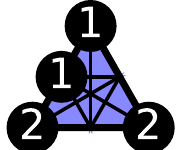} 
= \frac14 \SimS{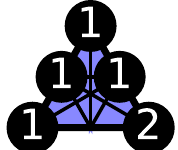} + \frac14 \SimS{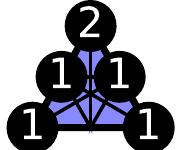} + \frac14 \SimS{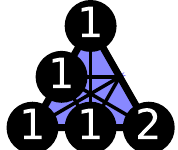} + \frac14 \SimS{211101} \\ \notag
& = \frac14 \SimS{211101} + \frac14 \SimS{112011}
+ \frac18 \SimS{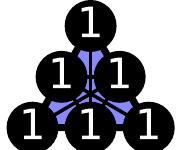} + \frac18 \SimS{111111} + \frac14 \left[ \frac12 \SimS{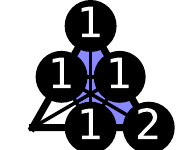} + \frac12 \SimS{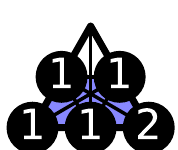}\right] \\ \notag
& = \frac14 \SimS{211101} + \frac14 \SimS{112011}
+ \frac14 \SimS{111111} + \frac14 \SimS{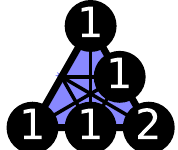}.
\end{align}

\subsubsection{Symmetries}\label{sec:symmetries}
\noindent The dihedral group $\cG$ of the equilateral triangle consists of the identity, two rotations and three reflections, i.e.,
\begin{center}
\includegraphics[scale=0.6]{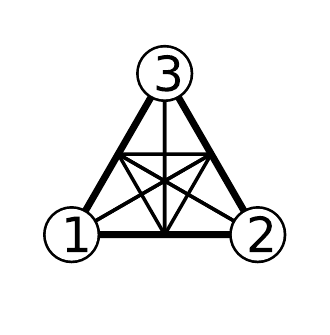}\quad
\includegraphics[scale=0.6]{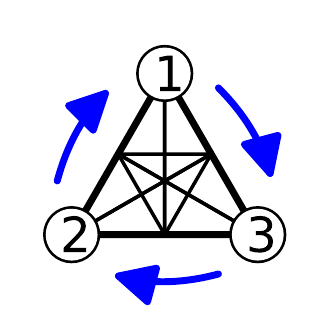}\quad
\includegraphics[scale=0.6]{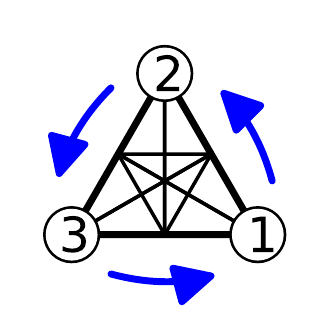}\quad
\includegraphics[scale=0.6]{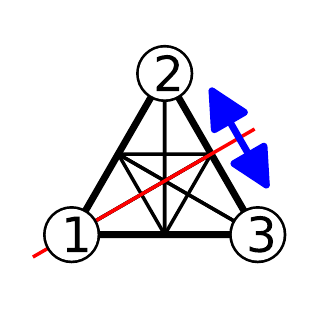}\quad
\includegraphics[scale=0.6]{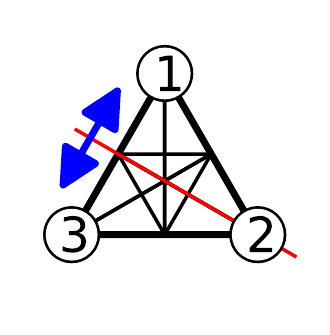}\quad
\includegraphics[scale=0.6]{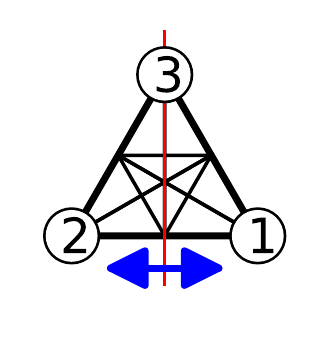}\quad
\end{center}
The affine bijection sending $\bfp_k$ to $(\cos\,2\pi k/3, \sin\,2\pi k/3)$, for $k = 1, 2, 3$, maps $\PSB$ to the 12-split of an equilateral triangle. Through this correspondence, the dihedral group permutes the vertices $\bfp_1,\ldots,\bfp_{10}$ of $\PSB$. Every such permutation $\sigma$ induces a bijection $Q[\bfp_1^{\mu_1}\cdots\bfp_{10}^{\mu_{10}}] \longmapsto Q[\sigma(\bfp_1)^{\mu_1}\cdots $ $\sigma(\bfp_{10})^{\mu_{10}}]$ on the set of all simplex splines on $\PSB$. For any set $\vs$ of simplex splines, we write
\[ [\vs]_\cG := \{Q[\sigma(\cK)]\,:\, Q[\cK]\in \vs,\ \sigma\in \cG\} \]
for the \emph{$\cG$-equivalence class of $\vs$}, i.e., the set of simplex splines related to $\vs$ by a symmetry in $\cG$. In particular, the bases in \eqref{eq:SS-bases} shown in Table \ref{tab:dualpolynomials} take the compact form
\begin{align*}
\vs_0 & = \left[\frac18\SimSconstant{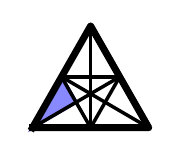}, \frac{1}{24}\SimSconstant{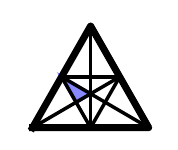}\right]_\cG,\\
\vs_1 & = \left[\frac14 \SimS{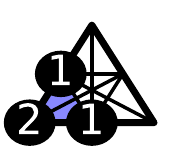}, \frac13\SimS{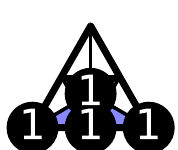}, \frac13\SimS{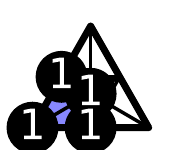}, \frac14\SimS{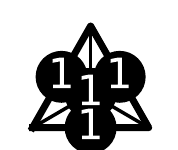} \right]_\cG,\\
\vs_2 & = \left[\frac14 \SimS{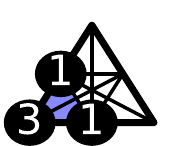}, \frac12 \SimS{210101}, \frac34 \SimS{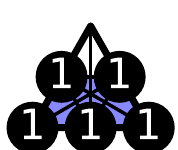}\right]_\cG,\\
\tilde{\vs}_2 &= \left[\frac14 \SimS{300101}, \frac12 \SimS{210101}, \frac34 \SimS{111101}\right]_\cG,\\
\vs_3 & = \left[\frac14 \SimS{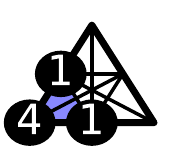}, \frac12 \SimS{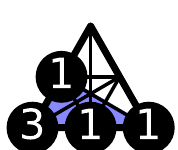}, \SimS{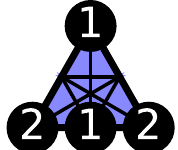}, \SimS{211101}, \frac14\SimS{111111}\right]_\cG,\\
\tilde{\vs}_3 &= \left[\frac14 \SimS{400101}, \frac12 \SimS{310101}, \SimS{221100}, \frac34 \SimS{211101}, \SimS{222000}\right]_\cG.
\end{align*}
We say that $\vs$ is \emph{$\cG$-invariant} whenever $[\vs]_\cG = \vs$ (Property P1). One sees immediately that this is the case for the bases in \eqref{eq:SS-bases}. 

\subsubsection{Restriction to an edge}\label{sec:restriction}
Let $e = [\bfp_i, \bfp_k]$ be an edge of $\PS$ with midpoint $\bfp_j$ and let $\varphi_{ik}(t) := (1-t)\bfp_i + t\bfp_k$. By induction on $|\cK|$,
\begin{equation}\label{eq:restriction}
Q[\cK] \circ \varphi_{ik}(t) =
\left\{ \begin{array}{cl}
                        0 & \text{if }\mu_i + \mu_j + \mu_k < |\cK| - 1,\\
\frac{\area(\PSsmall)}{\area([\cK])}B(t)  & \text{if }\mu_i + \mu_j + \mu_k = |\cK| - 1,
\end{array} \right.
\end{equation}
where $B$ is the univariate B-spline with knot multiset $\{0^{\mu_i}\,0.5^{\mu_j}\,1^{\mu_k}\}$.

We say that $Q[\cK]$ \emph{reduces to a B-spline on the boundary} when $B$ is one of the consecutive univariate B-splines $B_1^d, \ldots, B_{d+2}^d$ on the open knot multiset $\{0^{d+1}\, 0.5^1\, 1^{d+1}\}$, i.e.,
\begin{equation}\label{eq:UnivariateB-Splines}
B_1^d := B[0^{d+1}\, 0.5^1],\qquad B_2^d := B[0^d \, 0.5^1\, 1^1],\qquad \ldots, \qquad B_{d+2}^d := B[0.5^1\, 1^{d+1}].
\end{equation}
Similarly a basis $\{S_1, \ldots, S_{n_d}\}$ of $\SS_d(\PSB)$ \emph{reduces to a B-spline basis on the boundary} (Property P2) when, for $1\leq i < k\leq 3$, as multisets,
\[ \big\{S_1 \circ \varphi_{ik}, \ldots, S_{n_d} \circ \varphi_{ik}\big\} = \{ \big(B^d_1 \big)^1\, \cdots\, \big(B^d_{d+2}\big)^1\, 0^{n_d - d - 2}\}.\]
One sees that this is the case for the bases in \eqref{eq:SS-bases}. 

\begin{remark}\label{rem:P2}
Property P2 should be interpreted as follows: It requires that after restricting a bivariate basis to any edge of any triangle using the above reparametrization to $[0,1]$, one ends up with the same univariate basis. For instance, for the $C^2$ quintics on the 12-split \cite{Lyche.Muntingh16} this is the B-spline basis on the open knot multiset $\{0^6\, 0.5^2\, 1^6\}$, while for the $C^1$ cubics on the Clough-Tocher split \cite{Lyche.Merrien18} this is the cubic Bernstein basis (i.e., the B-splines on the open knot multiset $\{0^4\, 1^4\}$).
\end{remark}

\subsection{Enumeration on the 12-split}\label{sec:simplexenumeration}
Any simplex spline $Q[\cK]$ on $\PSB$ is specified by a multiset $\cK = \{\bfp_1^{\mu_1}\cdots\bfp_{10}^{\mu_{10}}\}$. Let us see how various properties of $Q[\cK]$ translate into conditions on the knot multiplicities.

Certain segments, like $[\bfp_1, \bfp_8]$, do not appear as edges in the 12-split, meaning that $C^{\infty}$-smoothness is required across these segments. Hence, by \eqref{eq:InteriorSmoothnessGeneral},
\begin{equation}\label{eq:WrongKnotLines}
\begin{aligned}
& \mu_1 \mu_8 = \mu_1 \mu_9 = \mu_8 \mu_9 = 0,\\
& \mu_2 \mu_7 = \mu_2 \mu_9 = \mu_7 \mu_9 = 0,\\
& \mu_3 \mu_7 = \mu_3 \mu_8 = \mu_7 \mu_8 = 0.
\end{aligned}
\end{equation}
If $Q[\cK]$ has degree $d$, then, by \eqref{eq:MultiplicitySumGeneral},
\begin{equation}\label{eq:MultiplicitySum0}
\mu_1 + \mu_2 + \cdots + \mu_{10} = d + 3.
\end{equation}
To achieve $C^r$-smoothness across the knotlines in $\PSB$, necessarily 
\begin{equation}
\begin{aligned}
 & \mu_1 + \mu_5 + \mu_7 + \mu_{10} \leq d + 1 - r,\qquad \mu_4 + \mu_6 + \mu_7 \leq d + 1 - r,\label{eq:InteriorSmoothness0}\\
 & \mu_2 + \mu_6 + \mu_8 + \mu_{10} \leq d + 1 - r,\qquad \mu_4 + \mu_5 + \mu_8 \leq d + 1 - r,\\
 & \mu_3 + \mu_4 + \mu_9 + \mu_{10} \leq d + 1 - r,\qquad \mu_5 + \mu_6 + \mu_9 \leq d + 1 - r,
\end{aligned}
\end{equation}
whenever two of the multiplicities are nonzero. 

\begin{lemma}\label{thm:KnotMultiplicities}
Suppose $Q[\bfp_1^{\mu_1}\cdots\bfp_{10}^{\mu_{10}}]$ is a $C^r$-smooth simplex spline on $\PSB$ of degree $d$. If $d \leq 2r + 1$, then
\begin{equation}\label{eq:m789}
\mu_7 = \mu_8 = \mu_9 = 0.
\end{equation}
If $d \leq \left\lceil\frac32 r\right\rceil$, then
\begin{equation}\label{eq:m10}
 \mu_{10} = 0, 
\end{equation}
\begin{equation}\label{eq:InteriorSmoothness}
\mu_1 + \mu_5,\ \mu_2 + \mu_6,\ \mu_3 + \mu_4,\ \mu_4 + \mu_6,\ \mu_4 + \mu_5,\ \mu_5 + \mu_6\leq d+1-r,
\end{equation}
(whenever both multiplicities are nonzero),
\begin{equation}\label{eq:m456}
\mu_4 + \mu_5 + \mu_6 \leq \frac32 (d + 1 - r),
\end{equation}
\begin{equation}\label{eq:m123}
\mu_1 + \mu_2 + \mu_3 \geq \frac12 (3r + 3 - d),
\end{equation}
\end{lemma}
\begin{proof}
Suppose $\mu_7\geq 1$. Then, by \eqref{eq:WrongKnotLines}, $\mu_2 = \mu_3 = \mu_8 = \mu_9 = 0$. Adding the first row in \eqref{eq:InteriorSmoothness0} and subtracting \eqref{eq:MultiplicitySum0}, yields $\mu_7 \leq d - 1 - 2r$. This is a contradiction whenever $d \leq 2r + 1$, establishing the first statement of the theorem.

Next assume $d \leq \left\lceil\frac32 r\right\rceil$. Adding the first column in \eqref{eq:InteriorSmoothness0} and using \eqref{eq:MultiplicitySum0}, yields $d + 3 + 2 \mu_{10} \leq 3d + 3 - 3r$. Solving for $\mu_{10}$, we obtain the second statement of the theorem.

The third statement follows immediately from the first two and \eqref{eq:InteriorSmoothness0}. Moreover, adding the inequalities in the second column in \eqref{eq:InteriorSmoothness0}, dividing by two, and using \eqref{eq:MultiplicitySum0}, one obtains the fourth statement. Finally, together with \eqref{eq:MultiplicitySum0}, we obtain the fifth statement.
\end{proof}

Next we determine the $C^{d - 1}$-smooth simplex splines of degree $d$ on $\PSB$ for $d = 0,1,2,3$. Selecting those that reduce to either zero or a B-spline on the boundary, we arrive at Table \ref{tab:AllSimplexSplines}.

\subsubsection{The case $d=0$ and $r=-1$}
The $C^{-1}$-smooth constant simplex splines $Q[\cK]$ on $\PSB$ have $|\cK| = 3$, corresponding to triples of knots not lying on a line.

\subsubsection{The case $d=1$ and $r=0$}
The $C^0$-smooth linear simplex splines $Q[\cK]$ on $\PSB$ have knot multiplicities satisfying $\mu_7 = \mu_8 = \mu_9 = 0$ by Lemma \ref{thm:KnotMultiplicities}, and therefore $\mu_1 + \cdots + \mu_6 + \mu_{10} = 4$ by \eqref{eq:MultiplicitySum0}. By \eqref{eq:InteriorSmoothness0}, $\mu_{10} \leq 1$. If $\mu_{10} = 1$, then \eqref{eq:InteriorSmoothness0} implies $\mu_1, \mu_2, \ldots, \mu_6\leq 1$. Up to symmetry, and systematically distinguishing cases by the number of corner knots, we obtain the simplex splines
\[ \SimS{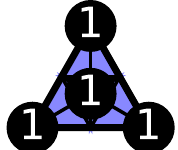},\ \SimS{1101000001},\ \SimS{1001010001},\ \SimS{0001110001}. \]
If $\mu_{10} = 0$, then, again distinguishing cases by the number of corner knots, we obtain the simplex splines 
\[ \SimS{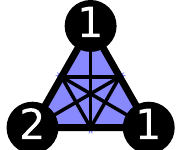},\ \SimS{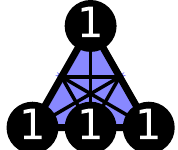},\ \SimS{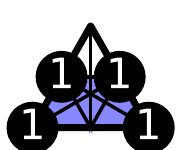},\ \SimS{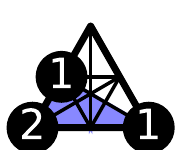},\ \SimS{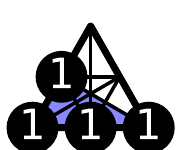},\ \SimS{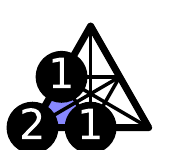},\ \SimS{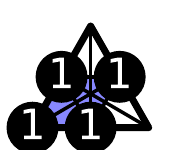}. \]

\subsubsection{The case $d=2$ and $r=1$}
The $C^1$-smooth quadratic simplex splines $Q[\cK]$ on $\PSB$ have knot multiplicities satisfying $\mu_7 = \mu_8 = \mu_9 = \mu_{10} = 0$ by Lemma~\ref{thm:KnotMultiplicities}, and
\[ \mu_1 + \cdots + \mu_6 = 5,\qquad \mu_1 + \mu_2 + \mu_3\geq 2,\qquad \mu_4 + \mu_5 + \mu_6 \leq 3. \]
Distinguishing cases by the number of corner knots, yields
\begin{equation}\label{eq:QuadraticSimplexSplines}
\SimS{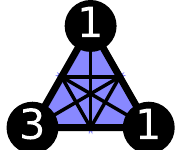},\ \SimS{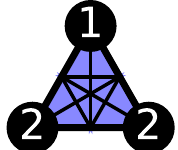},\ \SimS{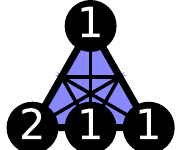},\ \SimS{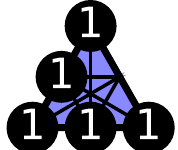},\ \SimS{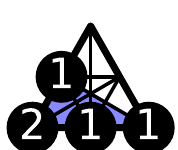},\ \SimS{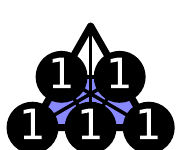},\ \SimS{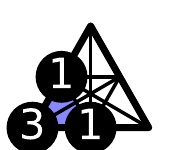}.
\end{equation}

\subsubsection{The case $d=3$ and $r=2$}
The $C^2$-smooth cubic simplex splines $Q[\cK]$ on $\PSB$ have, by Lemma~\ref{thm:KnotMultiplicities}, knot multiplicities satisfying $\mu_7 = \mu_8 = \mu_9 = \mu_{10} = 0$ 
\begin{equation}\label{eq:interiorexteriormultiplicities}
\mu_1 + \mu_2 + \mu_3\geq 3,\qquad \mu_4 + \mu_5 + \mu_6 \leq 3.
\end{equation}

Let $e = [\bfp_i, \bfp_k]$ be any edge of $\PS$ with midpoint $\bfp_j$. If $\mu_i + \mu_j + \mu_k < 5$, then $Q[\cK]|_e = 0$ by \eqref{eq:restriction}. In the remaining case $\mu_i + \mu_j + \mu_k = 5$ we demand that $Q[\cK]$ reduces to a B-spline on the boundary, yielding the conditions
\begin{equation}\label{eq:BoundaryBSpline}
\begin{aligned}
& \text{ not}(\mu_1 + \mu_4 + \mu_2 = 5 \text{ and } \mu_4 \geq 2),\\
& \text{ not}(\mu_1 + \mu_4 + \mu_2 = 5 \text{ and } \mu_1 \geq 1 \text{ and } \mu_2 \geq 1 \text{ and } \mu_4 \neq 1 ),\\
& \text{ not}(\mu_2 + \mu_5 + \mu_3 = 5 \text{ and } \mu_5 \geq 2),\\
& \text{ not}(\mu_2 + \mu_5 + \mu_3 = 5 \text{ and } \mu_2 \geq 1 \text{ and } \mu_3 \geq 1 \text{ and } \mu_5 \neq 1 ),\\
& \text{ not}(\mu_1 + \mu_6 + \mu_3 = 5 \text{ and } \mu_6 \geq 2),\\
& \text{ not}(\mu_1 + \mu_6 + \mu_3 = 5 \text{ and } \mu_1 \geq 1 \text{ and } \mu_3 \geq 1 \text{ and } \mu_6 \neq 1 ).
\end{aligned}
\end{equation}

\begin{theorem}\label{thm:simplexsplines12-split}
With one representative for each $\cG$-equivalence class, Table \ref{tab:AllSimplexSplines} presents an exhaustive list of the $C^2$ cubic simplex splines on $\PSB$ that reduce to either zero or a B-spline on the boundary.
\end{theorem}

\begin{table}
\begin{tabular*}{\columnwidth}{@{ }@{\extracolsep{\stretch{1}}}*{10}{c}@{ }}
\toprule
\raisebox{-1.3em}{$d = 1$} & \SmpS{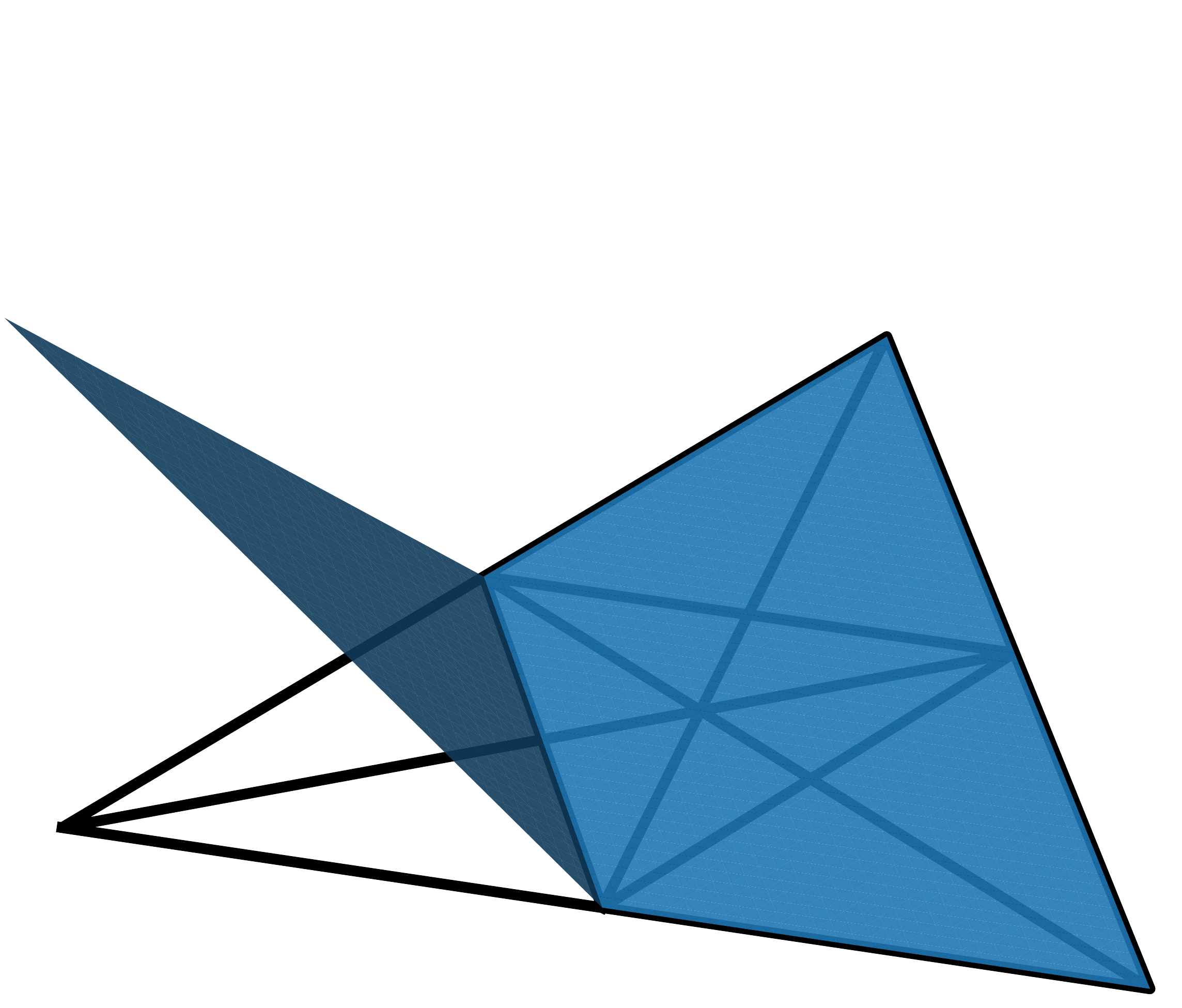} & \SmpS{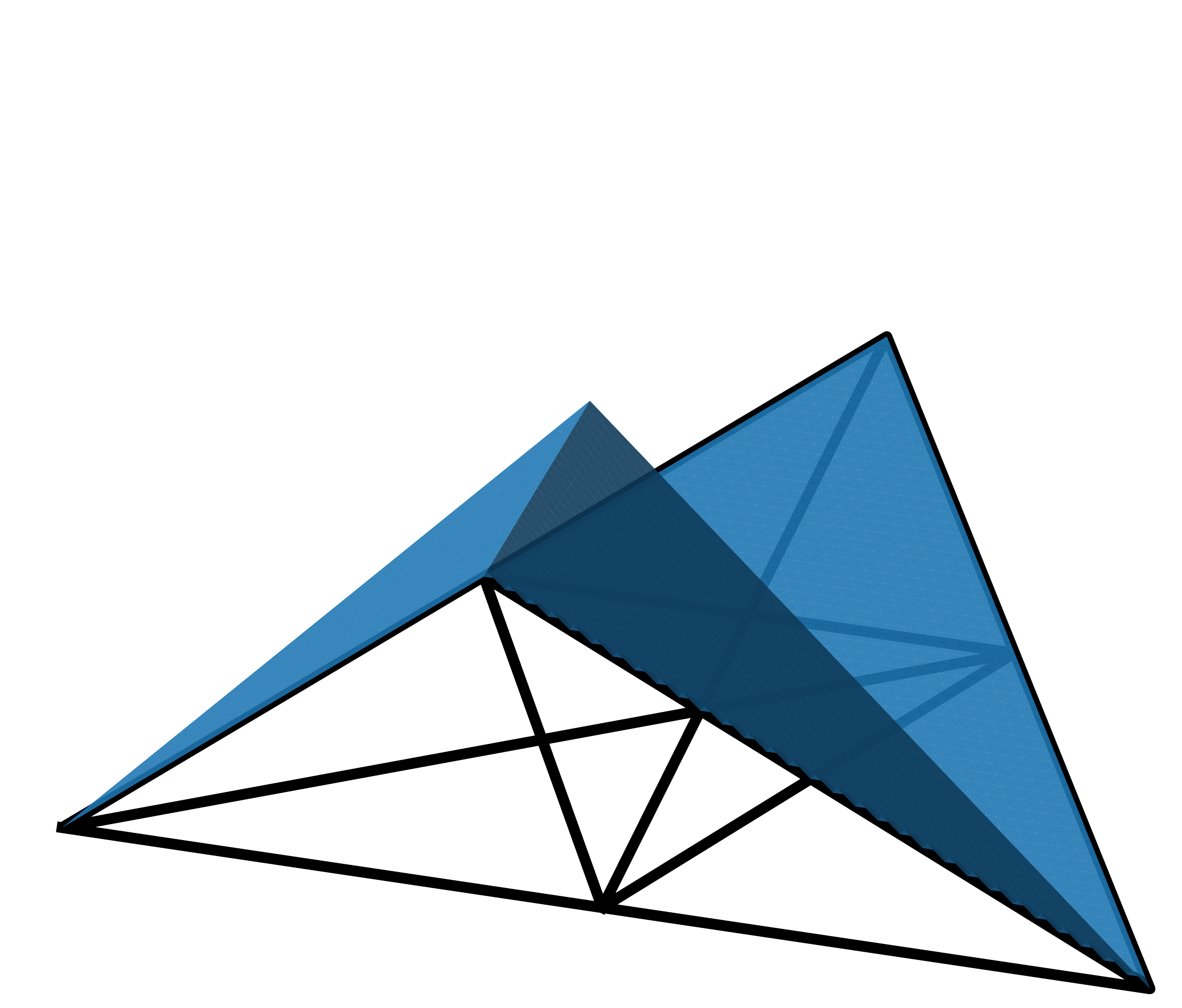} & \SmpS{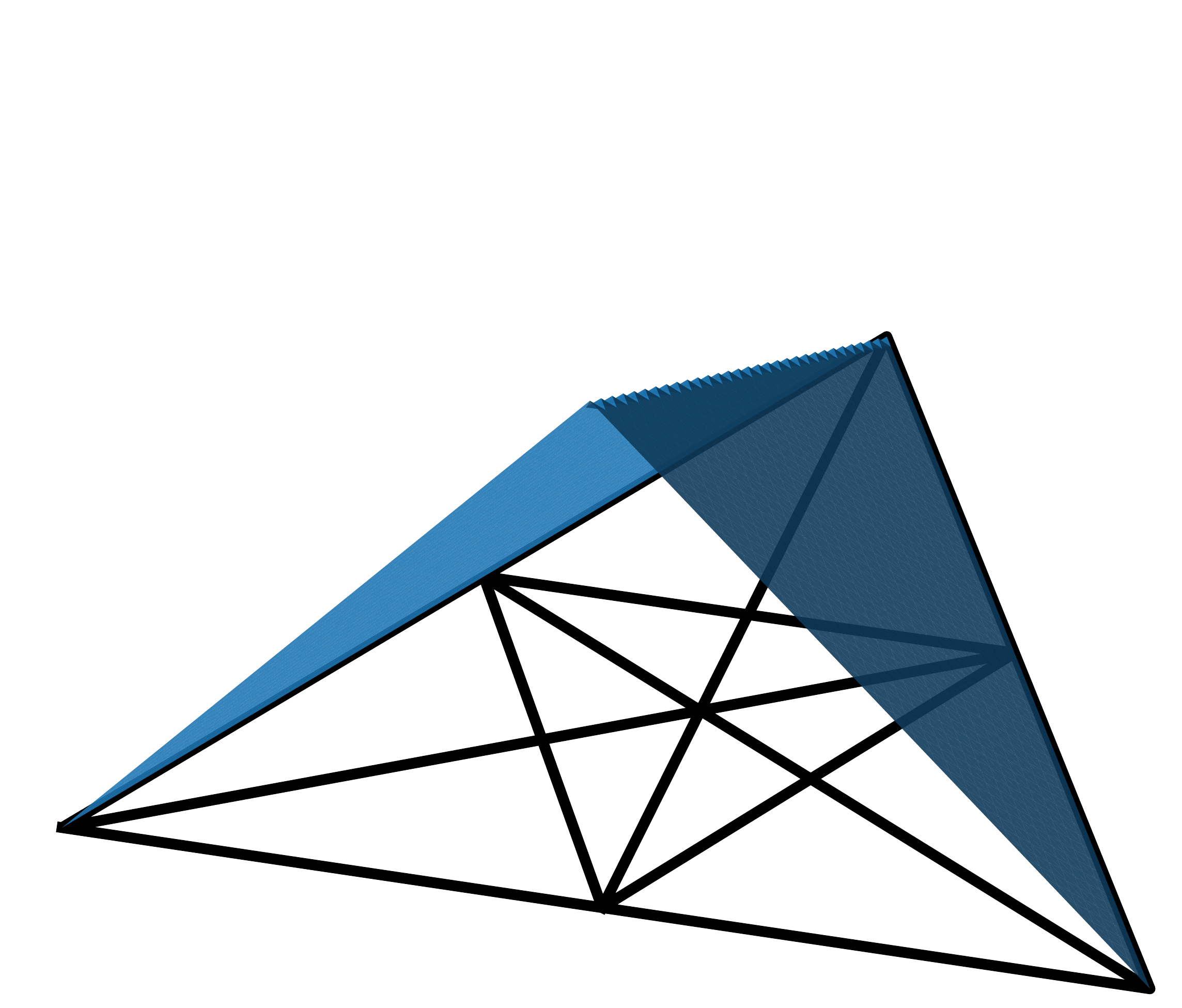} & \SmpS{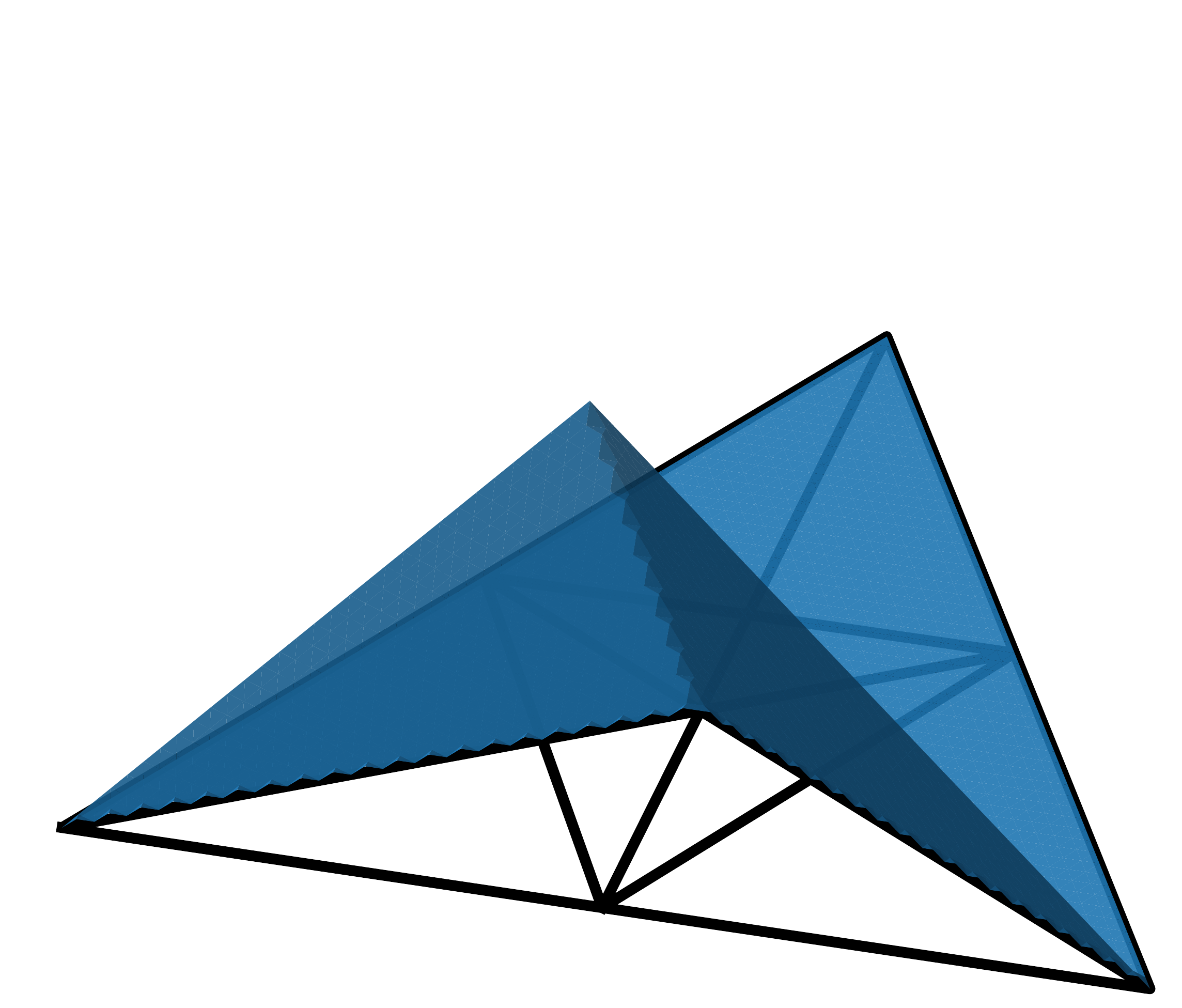} & \SmpS{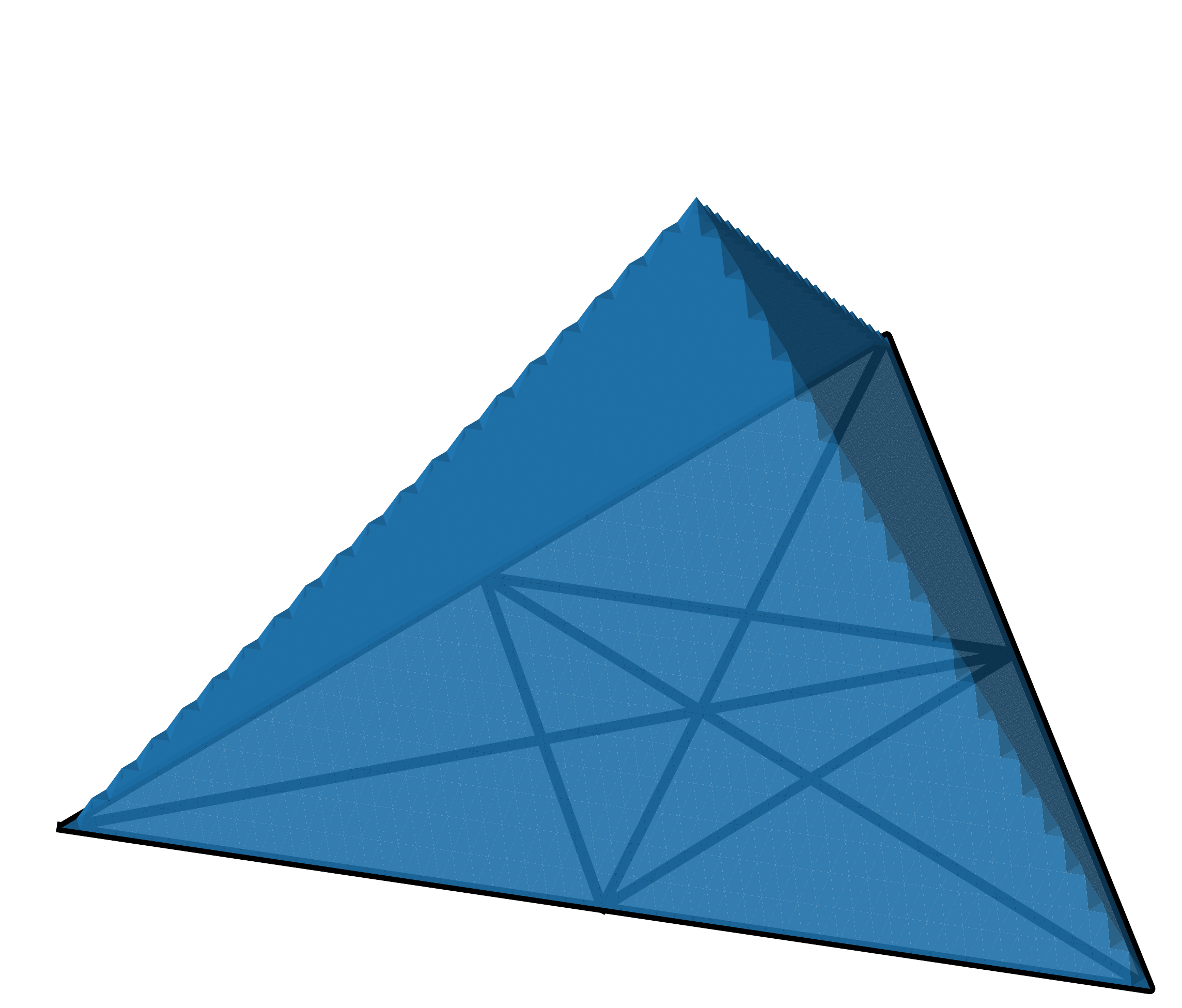} & \SmpS{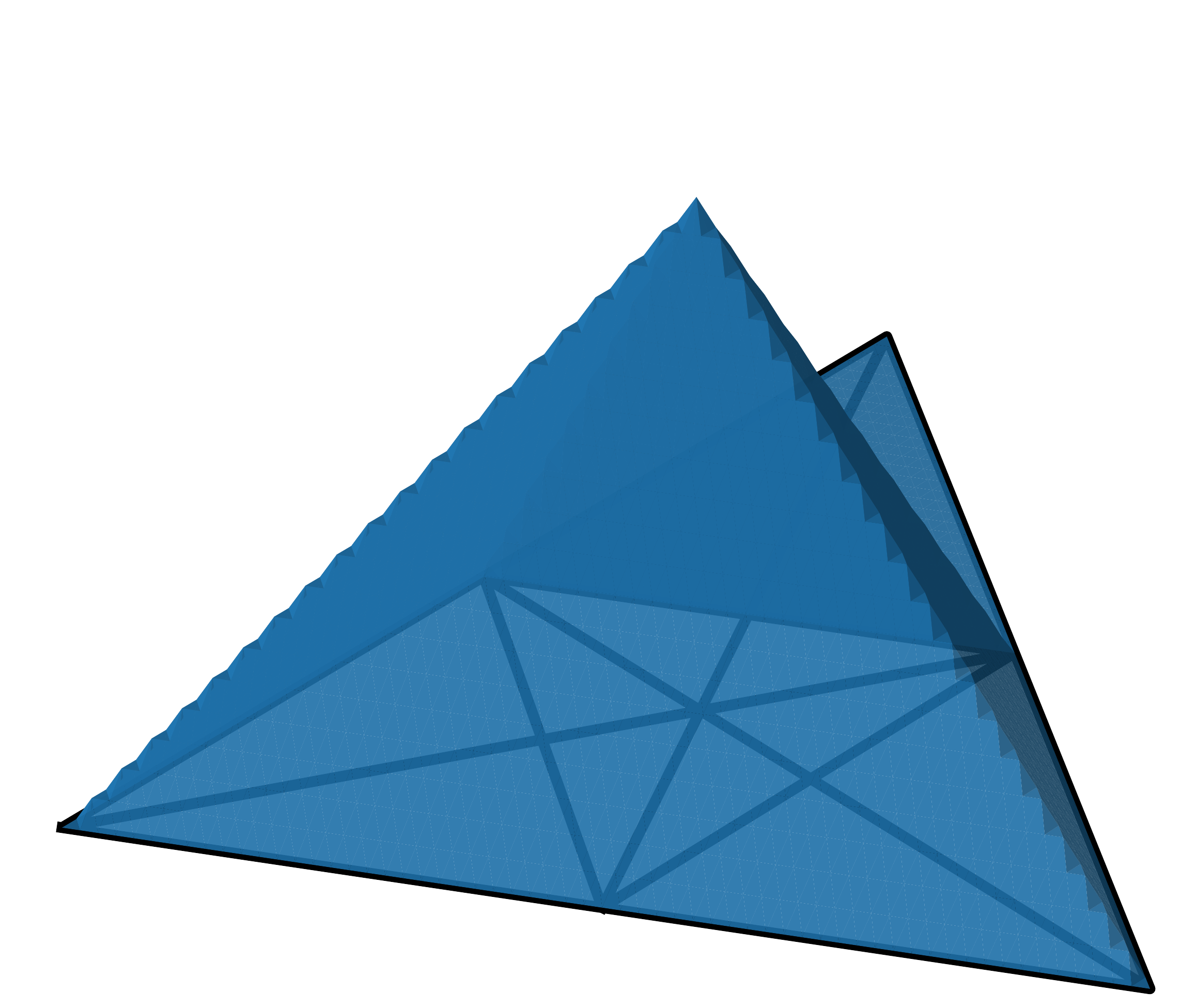} & \SmpS{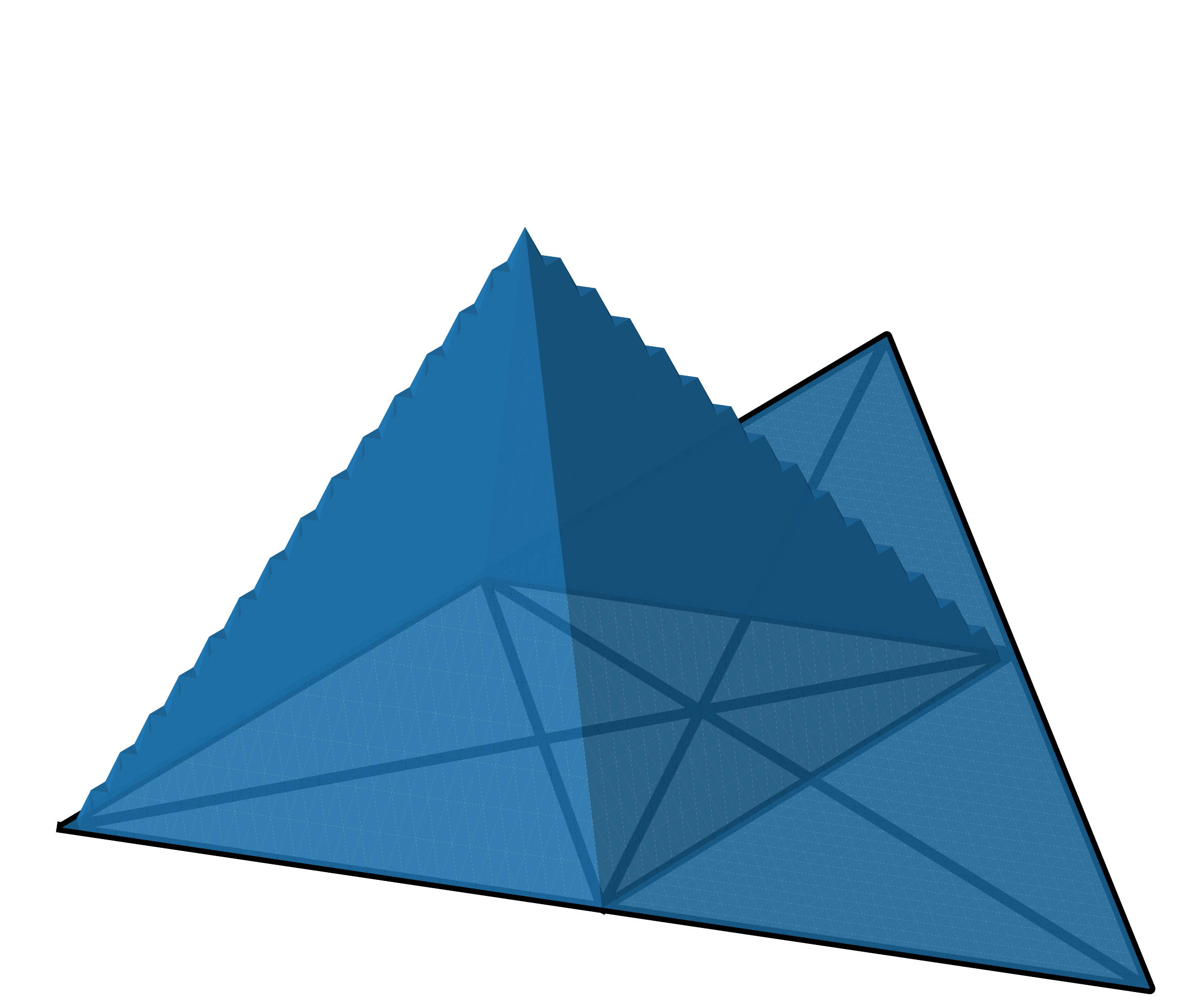} & \SmpS{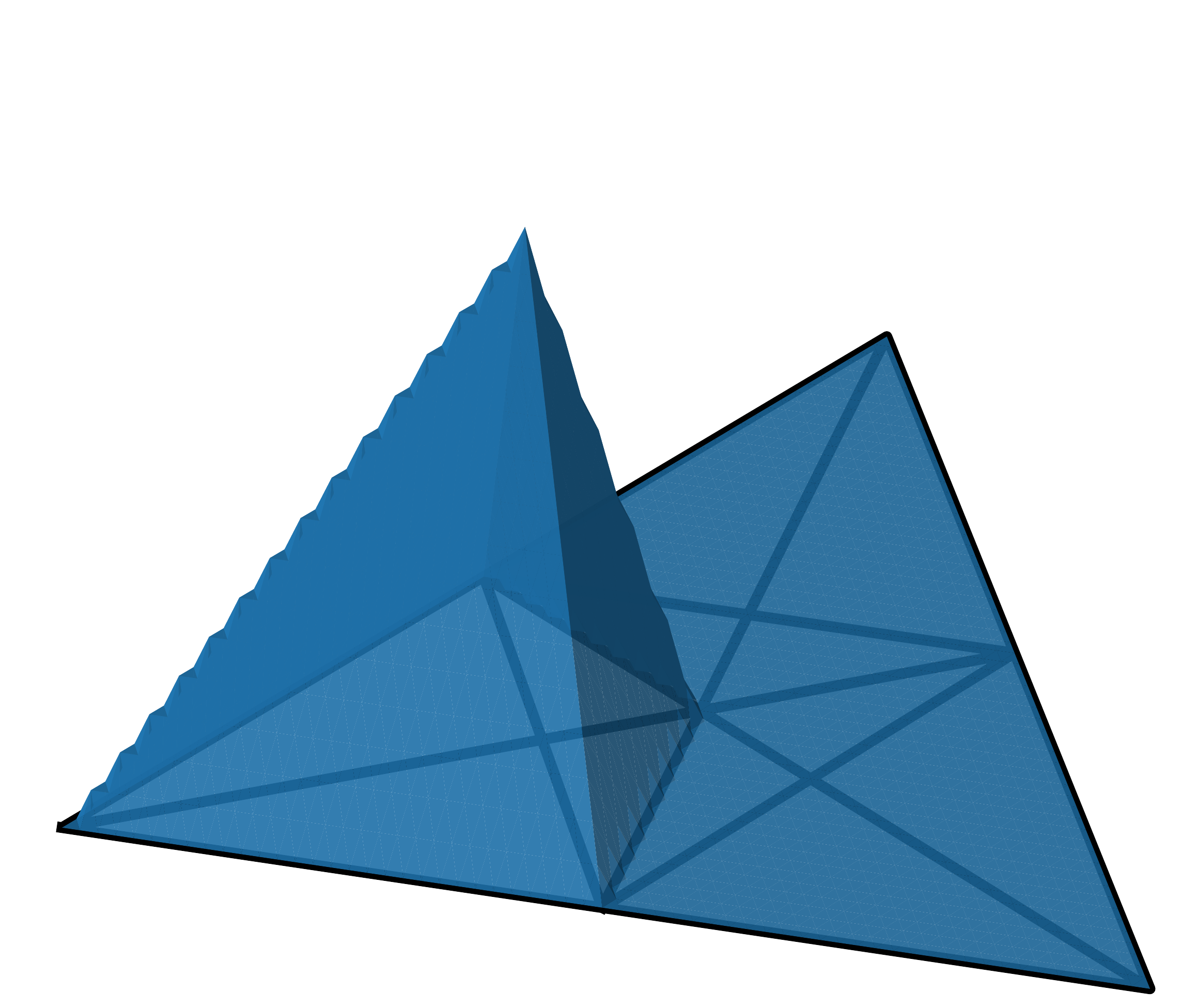} & \SmpS{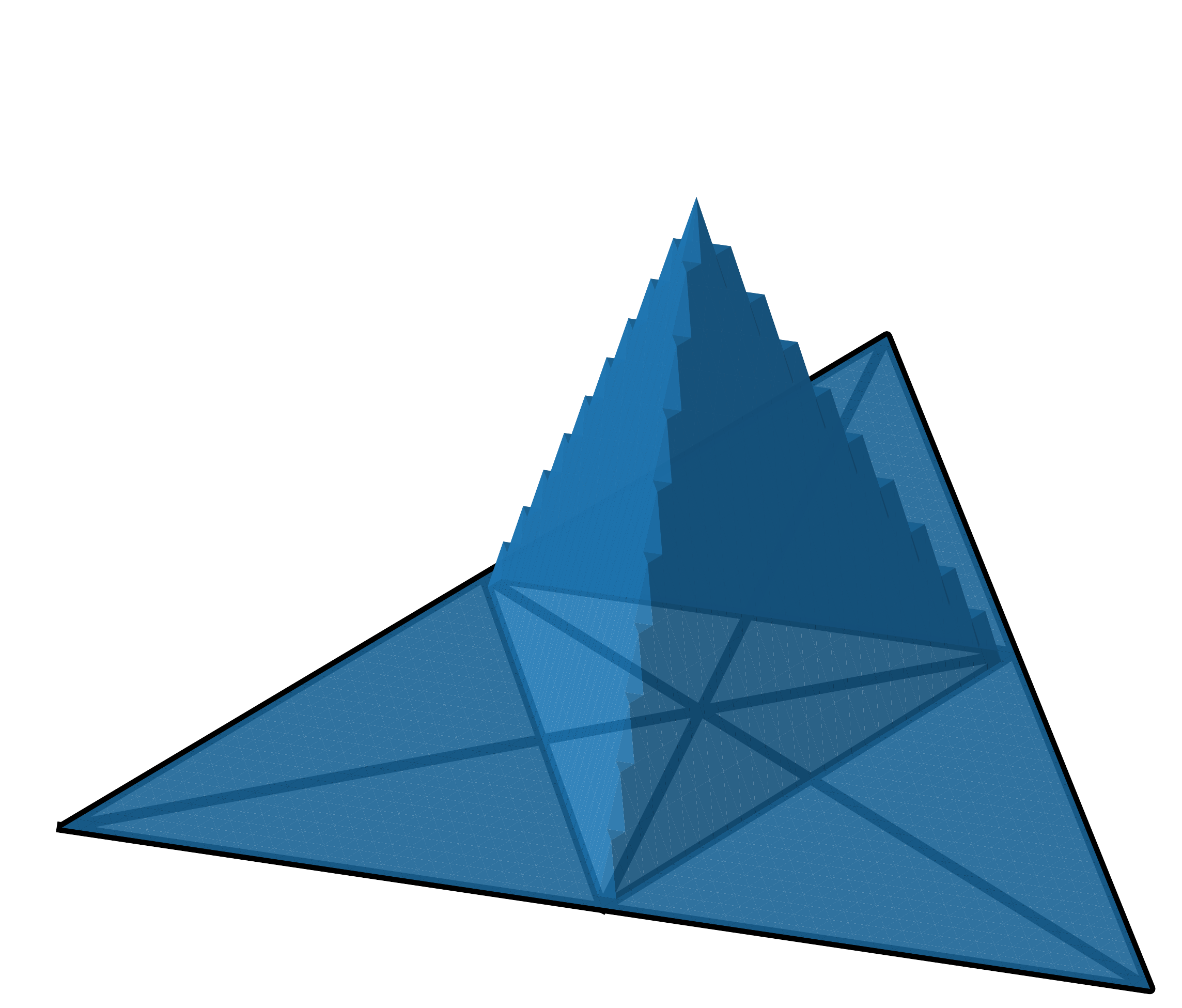}\\
        & \SimS{2001010000} & \SimS{1101010000} & \SimS{1111000000} & \SimS{1101000001} & \SimS{1110000001} & \SimS{1100110000} & \SimS{1001110000} & \SimS{1001010001} & \SimS{0001110001}\\
\midrule
\raisebox{-1.3em}{$d = 2$} & \SmpS{3001010000} & \SmpS{2101010000} & \SmpS{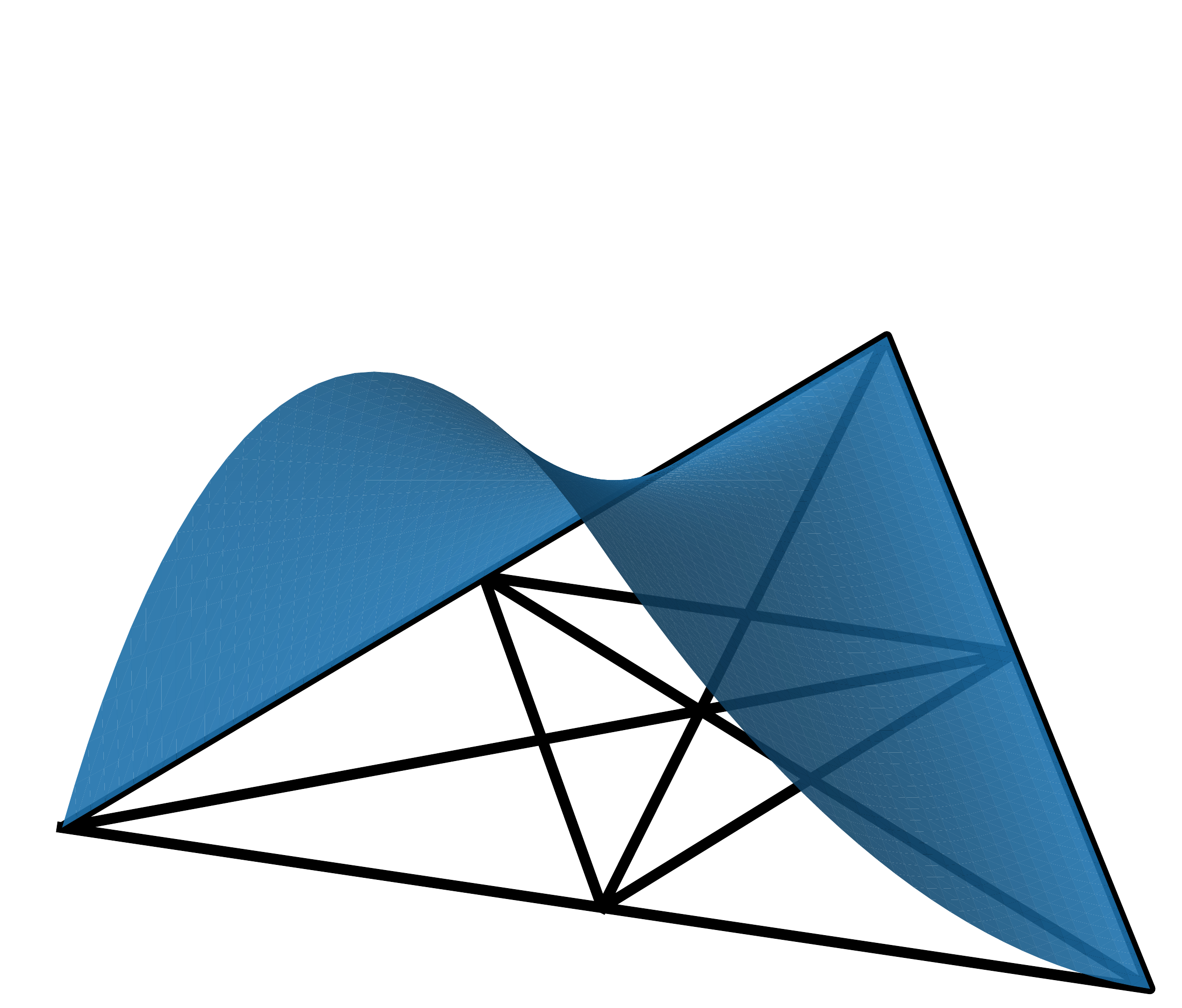} & \SmpS{1101110000}  & \SmpS{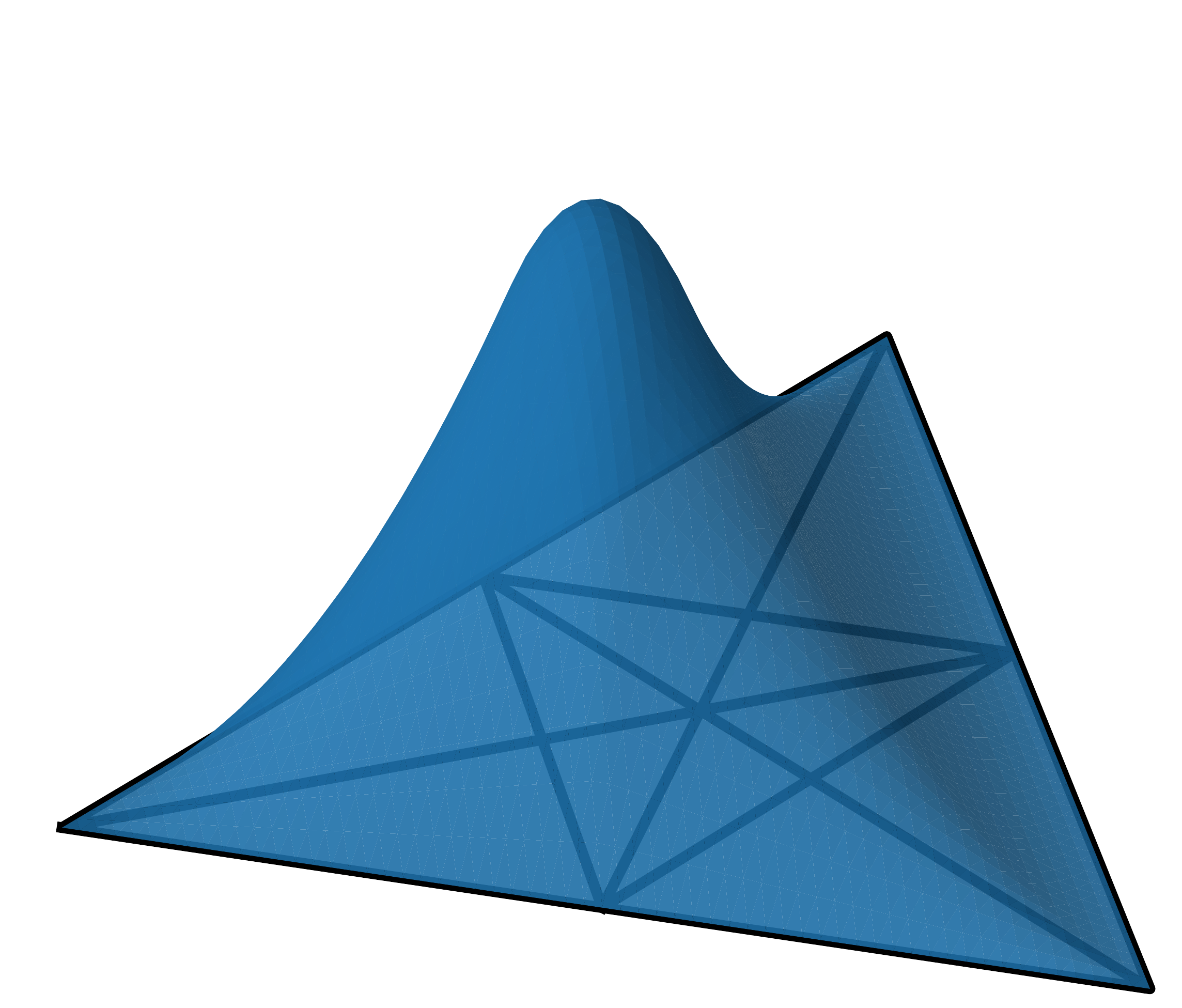} \\
        & \SimS{3001010000} & \SimS{2101010000} & \SimS{2111000000} & \SimS{1101110000}  & \SimS{1111010000} \\
\midrule
\raisebox{-1.3em}{$d = 3$} & \SmpS{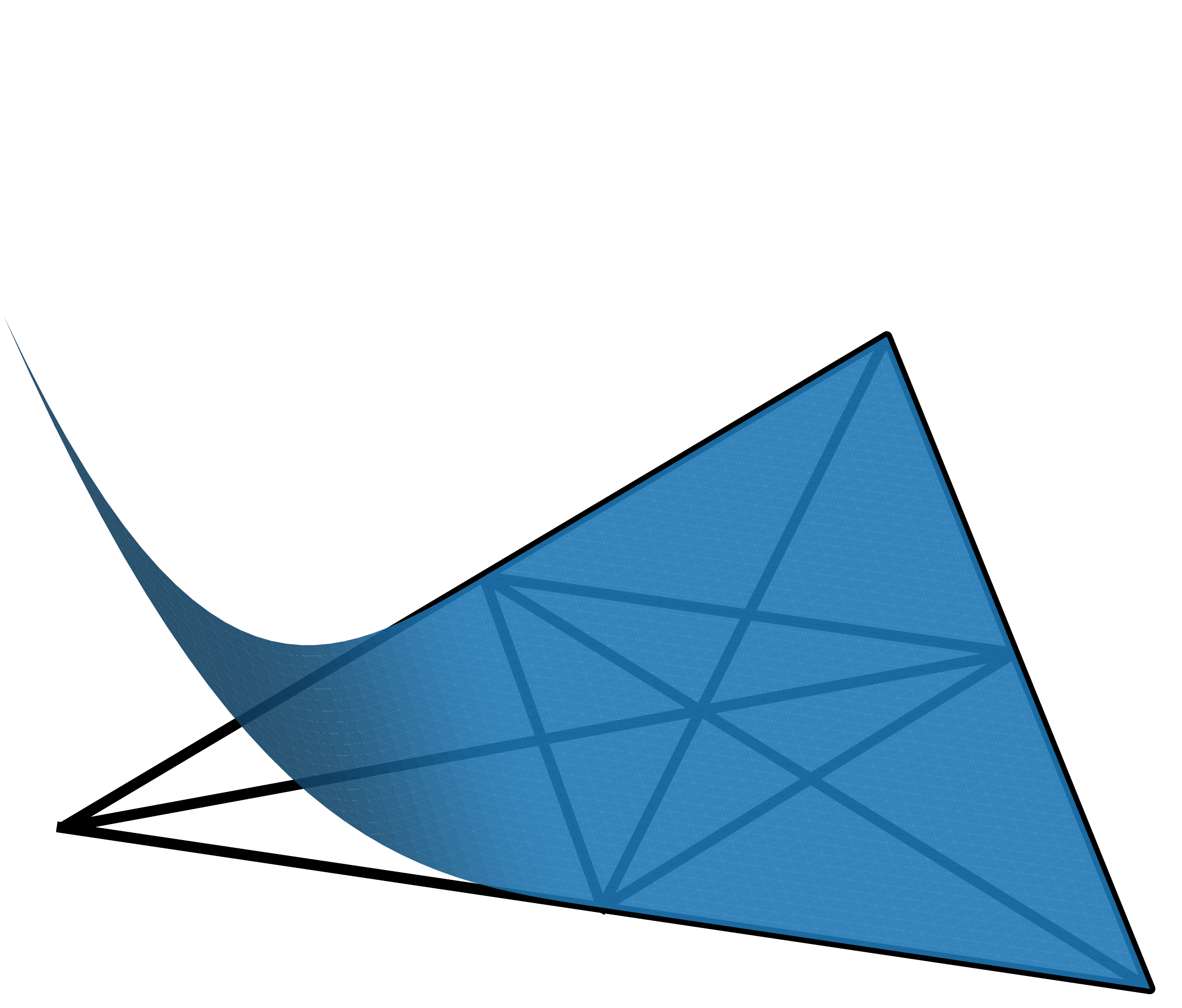} & \SmpS{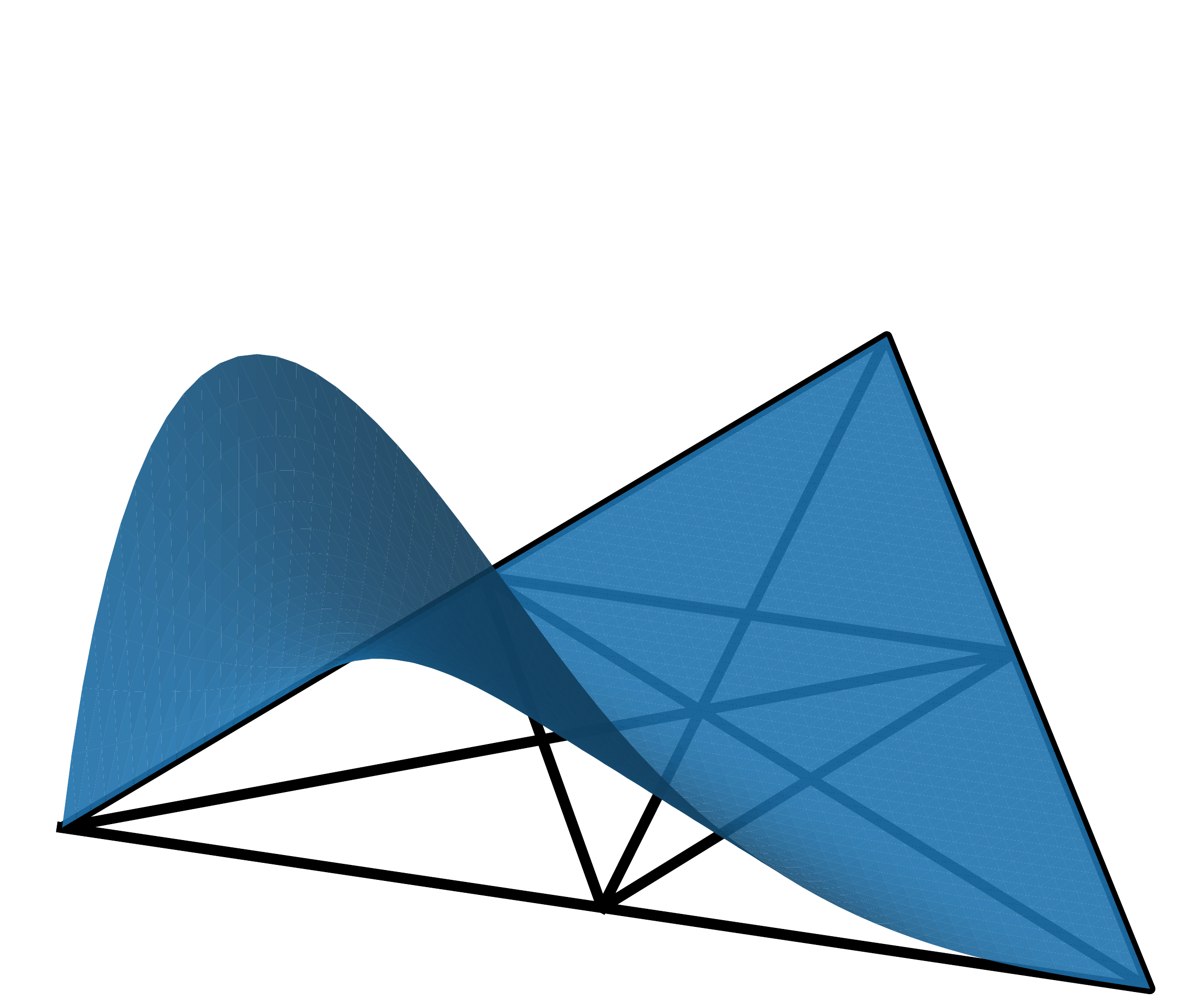} & \SmpS{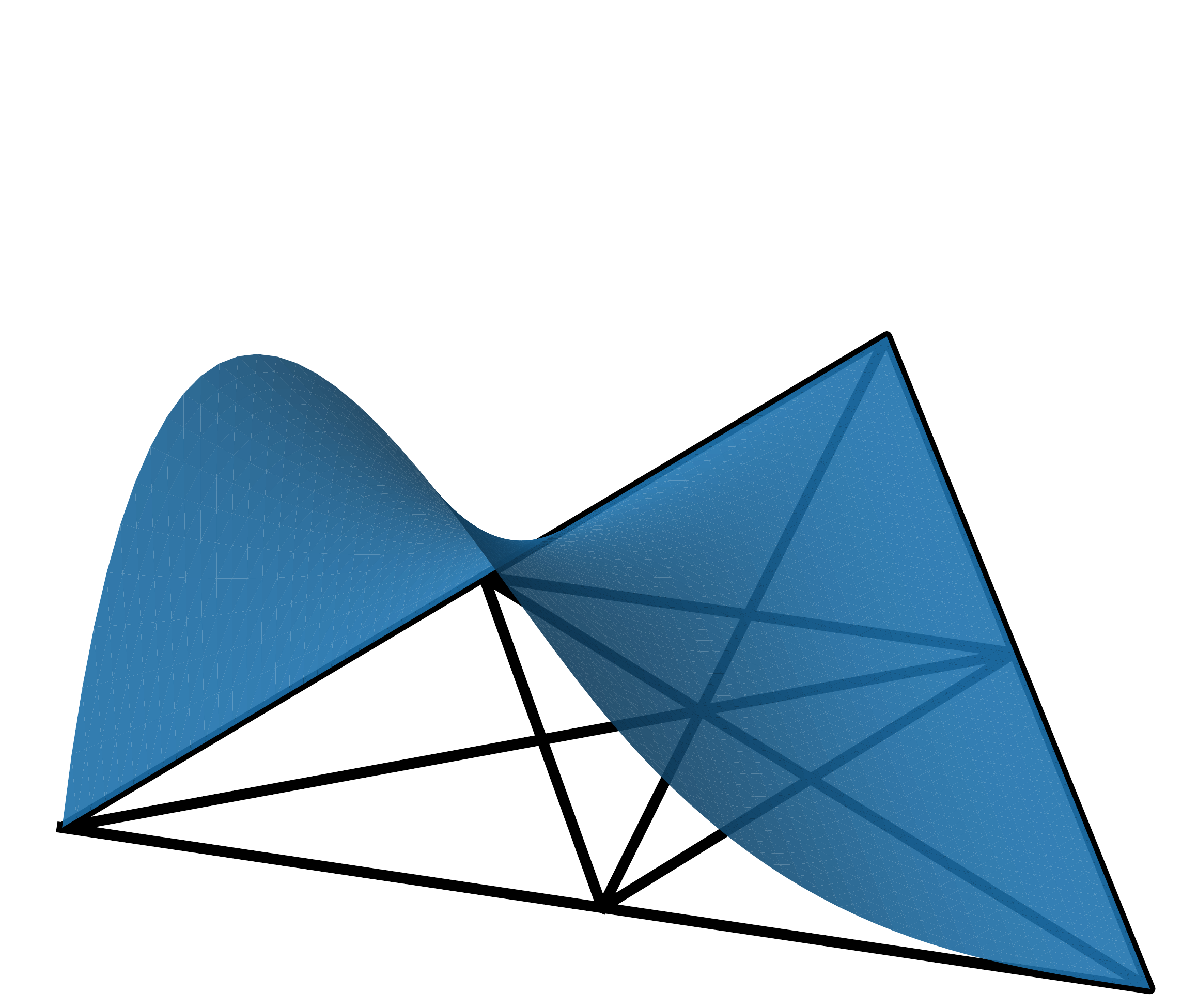} & \SmpS{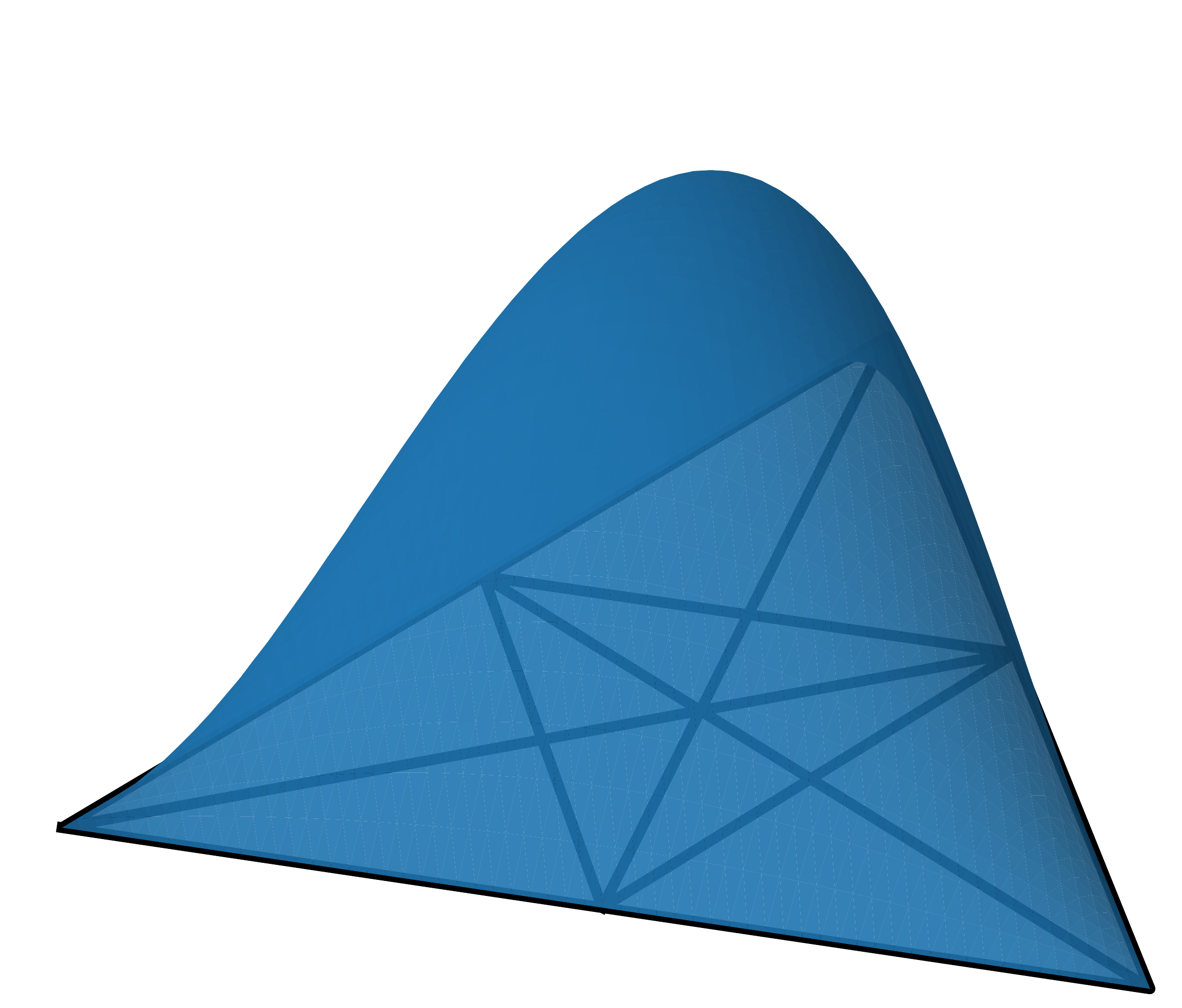} & \SmpS{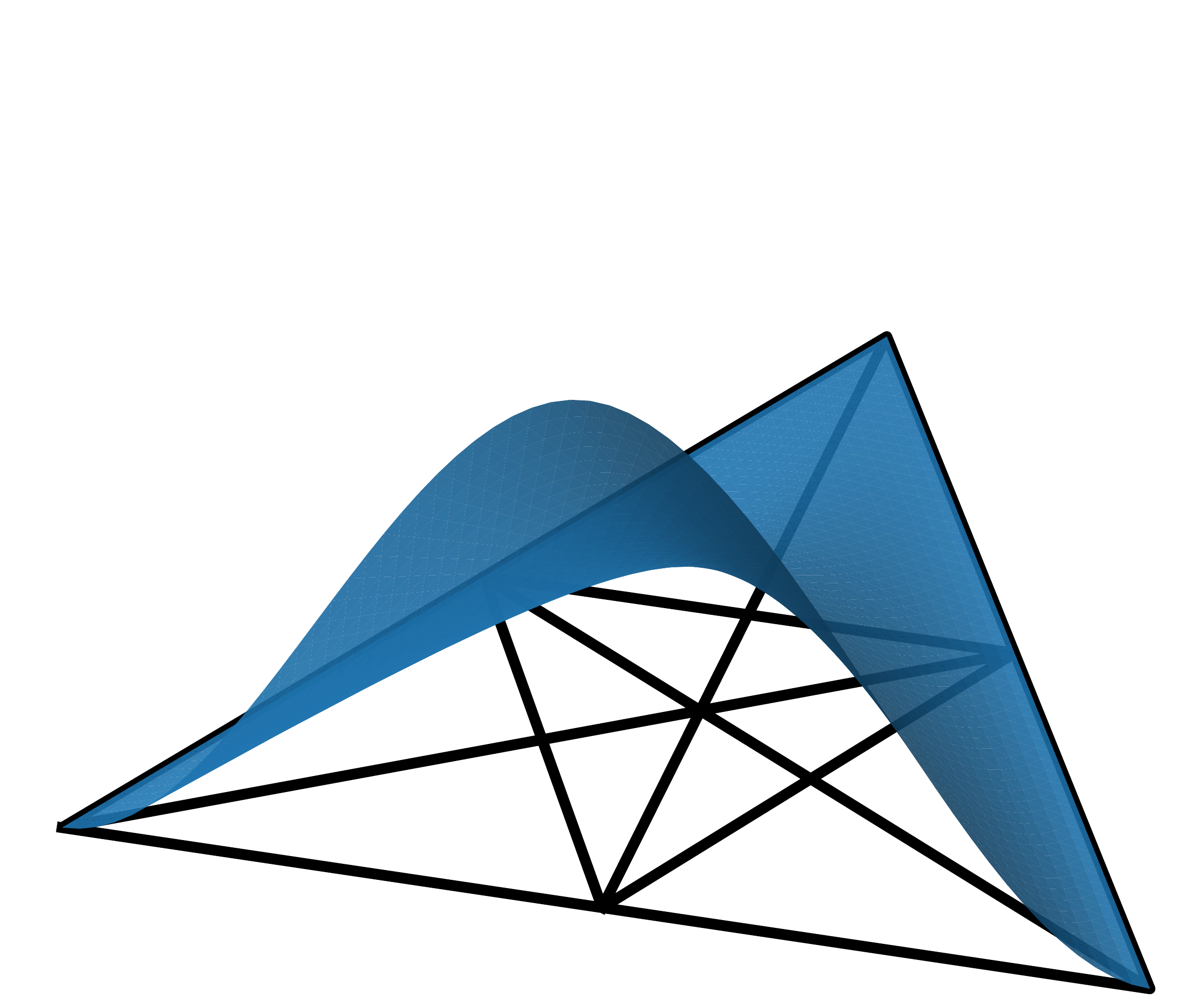} & \SmpS{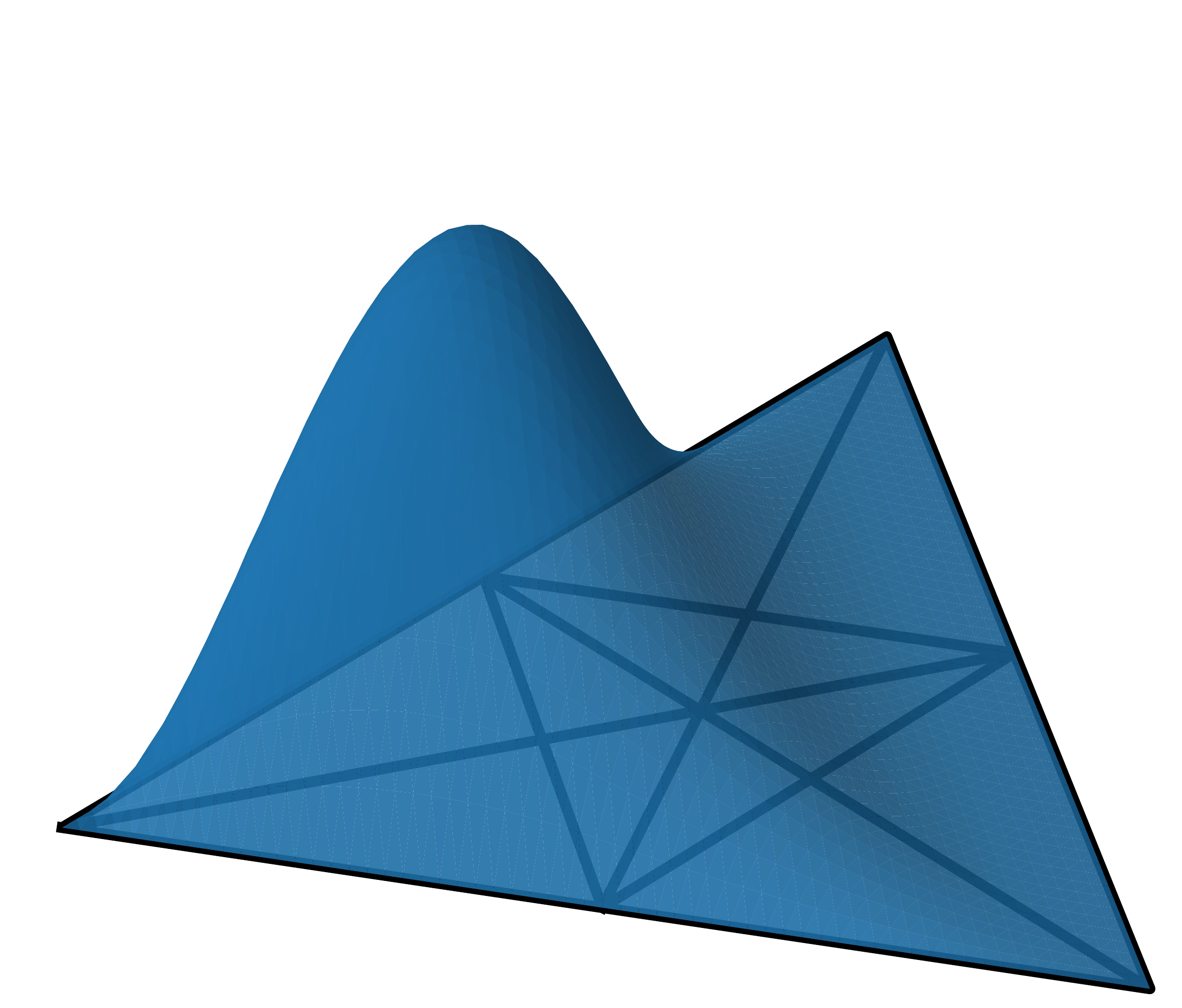} & \SmpS{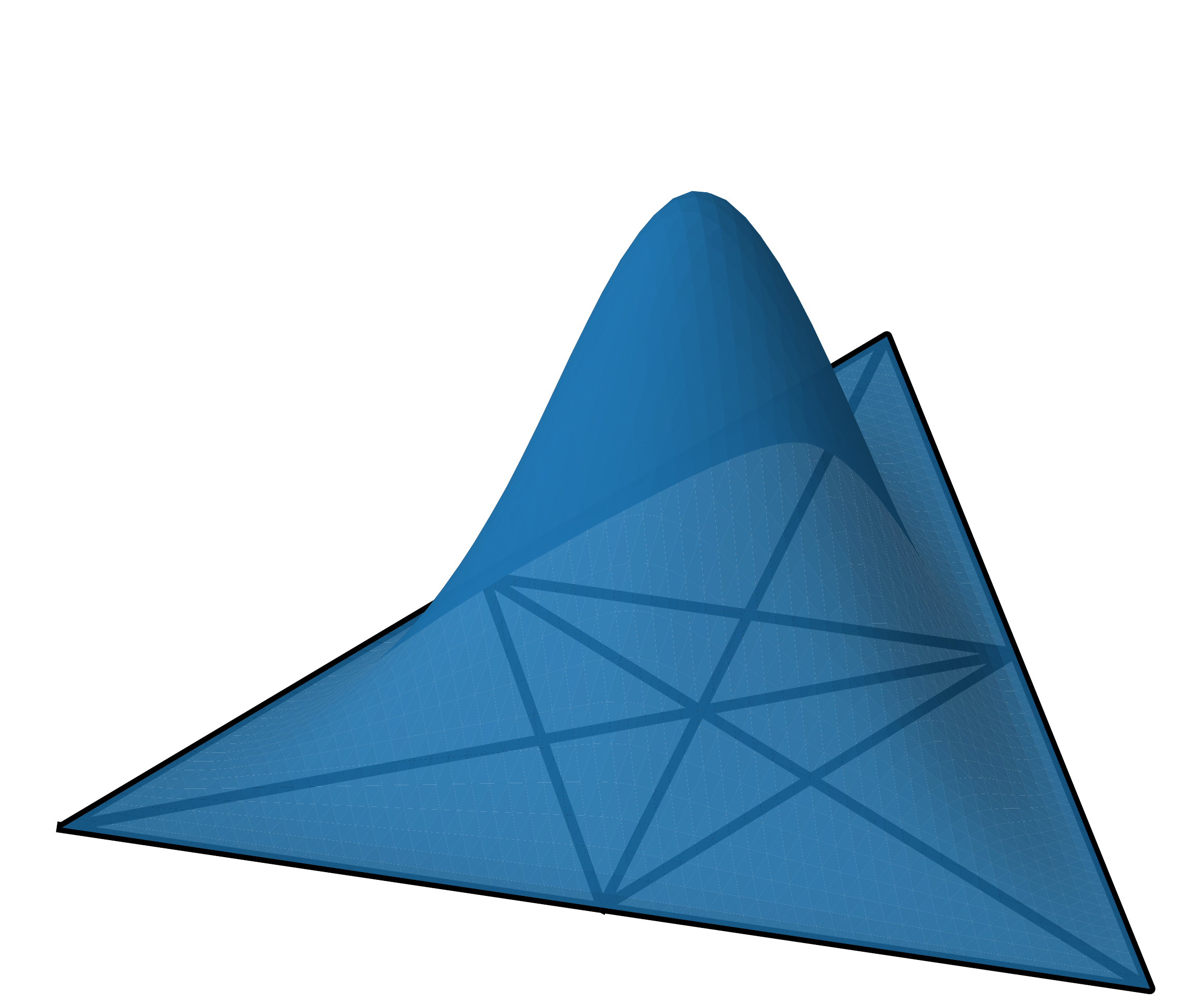}\\
        & \SimS{400101} & \SimS{310101} & \SimS{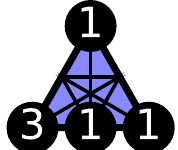} & \SimS{222000} & \SimS{221100} & \SimS{211101} & \SimS{111111}\\
\bottomrule
\end{tabular*}
\caption[]{For $d = 1, 2, 3$, the $C^{d-1}$-smooth simplex splines of degree $d$ on $\PSB$, one representative for each $\cG$-equivalence class, that reduce to a B-spline on the boundary.}\label{tab:AllSimplexSplines}
\end{table}

\begin{proof}
By \eqref{eq:WrongKnotLines}, it suffices to consider the following cases according to the support $[\cK]$ of $Q[\cK]$, up to a symmetry of $\cG$,

\emph{Case 0, no corner included, $[\cK] = [\bfp_4, \bfp_5, \bfp_6]$}: By \eqref{eq:MultiplicitySum0}, $\mu_4 + \mu_5 + \mu_6 = 6$, contradicting \eqref{eq:InteriorSmoothness}. Hence this case does not happen.

\emph{Case 1a, 1 corner included, $[\cK] = [\bfp_1, \bfp_4, \bfp_6]$}: For a positive support $\mu_1, \mu_4, \mu_6\geq 1$, and since $\mu_4 + \mu_6 \leq 2$ by \eqref{eq:InteriorSmoothness}, we obtain $\SimS{400101}$.

\emph{Case 1b, 1 corner included, $[\cK] = [\bfp_1, \bfp_4, \bfp_5, \bfp_6]$}: 
By \eqref{eq:MultiplicitySum0} and \eqref{eq:interiorexteriormultiplicities} one has $\mu_1 = 6 - \mu_4 - \mu_5 - \mu_6 \geq 3$, contradicting $\mu_1 + \mu_5 \leq 2$ from \eqref{eq:InteriorSmoothness}. Hence this case does not occur.

\emph{Case 2a, 2 corners included, $[\cK] = [\bfp_1, \bfp_2, \bfp_6]$}: For a positive support, $\mu_1, \mu_2, \mu_6 \geq 1$. Since $\mu_2 + \mu_6\leq 2$ by \eqref{eq:InteriorSmoothness}, it follows $\mu_2 = \mu_6 = 1$. Moreover, $\mu_4 = 1$ by \eqref{eq:BoundaryBSpline}, and we obtain $\SimS{310101}$.

\emph{Case 2b, 2 corners included, $[\cK] = [\bfp_1, \bfp_2, \bfp_5, \bfp_6]$}: Since $\mu_1 + \mu_5, \mu_2 + \mu_6 \leq 2$ by \eqref{eq:InteriorSmoothness}, implying $\mu_1 = \mu_2 = \mu_5 = \mu_6 = 1$. Then $\mu_4 = 2$ by \eqref{eq:MultiplicitySum0}, contradicting \eqref{eq:BoundaryBSpline}. Hence this case does not occur.

\emph{Case 3, 3 corners included, $[\cK] = [\bfp_1, \bfp_2, \bfp_3]$}: We distinguish cases for $(\mu_4, \mu_5, \mu_6)$, with $\mu_4\geq \mu_5\geq \mu_6$. By \eqref{eq:InteriorSmoothness}, $\mu_4, \mu_5, \mu_6\leq 1$. 
\begin{itemize}
\item[(0,0,0)] One has $\mu_1, \mu_2, \mu_3 = 2$ by \eqref{eq:BoundaryBSpline}, and we obtain $\SimS{222000}$.
\item[(1,0,0)] One has $\mu_3 + \mu_4 \leq 2$ by \eqref{eq:InteriorSmoothness} implying $\mu_3 = \mu_4 = 1$, yielding $\SimS{311100}$ and $\SimS{221100}$.
\item[(1,0,1)] One has $\mu_3 + \mu_4, \mu_2 + \mu_6 \leq 2$ by \eqref{eq:InteriorSmoothness}, implying $\mu_2 = \mu_3 = \mu_4 = \mu_6 = 1$. It follows from \eqref{eq:MultiplicitySum0} that $\mu_1 = 2$, yielding $\SimS{211101}$.
\item[(1,1,1)] From \eqref{eq:MultiplicitySum0} one immediately obtains $\SimS{111111}$. \qedhere
\end{itemize}
\end{proof}

\section{S-bases on the 12-split}\label{sec:S-bases}\label{sec:known}

For $d = 0,1,2,3$, consider the S-bases $\vs_d^\rmT = [S_{i, n_d}]_{i=1}^{n_d}$ listed in Table \ref{tab:dualpolynomials}. In this section, we relate these bases through a matrix recurrence relation (Property P5), generalizing Theorem~2.3 and Corollary~2.4 in \cite{Cohen.Lyche.Riesenfeld13} for $d\leq 2$.

\begin{theorem}\label{thm:recrel}
We have 
\begin{equation}\label{eq:bdR}
\vs_d^\rmT = \vs_{d-1}^\rmT \mR_d, \qquad d=1,2,3,
\end{equation}
where $\mR_1 \in\PP_1^{12,10}$ is given by \eqref{eq:R1}, $\mR_2 \in\PP_1^{10,12}$ by \eqref{eq:R2}, and $\mR_3 \in\PP_1^{12,16}$ by \eqref{eq:R3}. Moreover, $\mR_d(i,j)S_{i,d-1}(\vx)\ge 0$ for all $i,j$ and $\vx\in\PS$.
\end{theorem}

\begin{corollary}\label{cor:BRR}
Suppose $\vx\in\PS_k$ for some $1\le k\le 12$. Then
\begin{equation}\label{eq:sdr}
\vs_d^\rmT = \ve_k^\rmT \mR_1\cdots\mR_d, \qquad d = 0, 1, 2, 3.
\end{equation}
\end{corollary}

In the remainder of the section we build up this recurrence relation, starting from degree 0. We will make use of the short-hands \eqref{eq:gammabetasigma} involving the barycentric coordinates $\beta_1,\beta_2,\beta_3$ of $\vx$ with respect to the triangle $\PS$.

\begin{figure}
\subfloat[]{\includegraphics[scale=1.09, trim = 1 5 3 5]{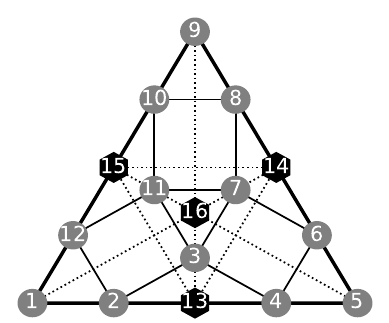} \label{fig:ControlNet-2}}
\subfloat[]{\includegraphics[scale=1.09, trim = 1 5 3 5]{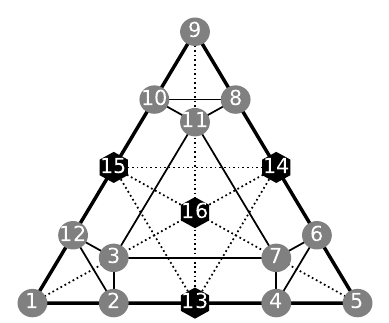}\label{fig:ControlNet-2t}}
\subfloat[]{\includegraphics[scale=1.09, trim = 1 5 3 5]{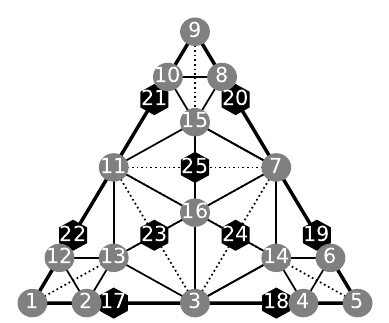} \label{fig:ControlNet-3}}
\subfloat[]{\includegraphics[scale=1.09, trim = 1 5 3 5]{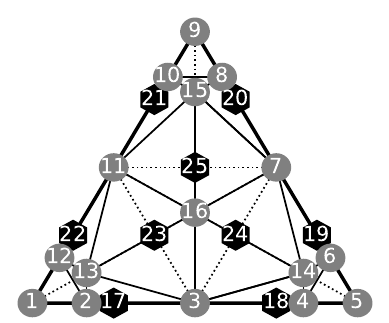}\label{fig:ControlNet-3t}}
\caption{
Domain meshes (solid) with numbering of the domain points (circles) and remaining dual point averages (hexagons), used in the quasi-interpolant \eqref{eq:quasi-interpolant}, for the bases (A) $\vs_2$, (B) $\tilde{\vs}_2$, (C) $\vs_3$,  (D) $\tilde{\vs}_3$ on the Powell-Sabin 12-split (dotted).
}\label{fig:domain_mesh}
\end{figure}

\subsection{Constant S-basis}
Since $\SS_0 (\PSB)$ has dimension $n_0 = 12$, it is easy to see that there is a unique S-basis $\vs_0 = [S_{1,0}, \ldots, S_{12,0}]$ forming a partition of unity. Explicitly, 
\begin{equation}\label{eq:S0}
S_{j,0}(\vx) = \mathbf{1}_{\PSsmall_j} (\bfx) := \begin{cases} 1,& \bfx\in\PS_j,\\ 0,&\text{otherwise},
\end{cases}\qquad j = 1,\ldots, 12,
\end{equation}
where the $\PS_j$ are the half-open subtriangles in Figure \ref{fig:ps12-numbering} (right), with disjoint union $\PS_1 \sqcup \cdots \sqcup \PS_{12} = \PS$. This implies that Corollary \ref{cor:BRR} follows immediately from Theorem \ref{thm:recrel}.

\subsection{Linear S-basis}
The basis $\vs_1^\rmT = \left[S_{1,1}, \ldots, S_{10,1}\right]$ of $\SS_1(\PSB)$ is the nodal basis dual to the point evaluations at the vertices of $\PSB$, i.e., $S_{j,1}(\vp_i)=\delta_{i,j}$, $i,j=1,\ldots,10$. Represented as elements of $\PP_1^{12}$, the basis functions $S_{1,1}, \ldots, S_{10,1}$ are precomputed and assembled as the columns of the matrix
\begin{equation}\label{eq:R1}
\mR_1 =
\begin{bmatrix}
\gamma_1& 0& 0& 0& 0&2\beta_{3,2} & 4 \beta_2& 0& 0& 0\\
 \gamma_1& 0& 0& 2\beta_{2,3}& 0& 0& 4 \beta_3& 0& 0& 0\\
 0& \gamma_2& 0&2\beta_{1,3}& 0& 0& 0& 4 \beta_3& 0& 0\\
 0& \gamma_2& 0& 0& 2\beta_{3,1}& 0& 0& 4 \beta_1& 0& 0\\
 0& 0&\gamma_3& 0&2\beta_{2,1}& 0& 0& 0& 4 \beta_1& 0\\
 0& 0& \gamma_3& 0& 0& 2\beta_{1,2}& 0& 0& 4 \beta_2& 0\\
 0& 0& 0& 0& 0& 2 \beta_{3,2}& 4 \beta_{1,3}& 0& 0& -3\gamma_1\\
 0& 0& 0& 2\beta_{2,3} & 0& 0& 4\beta_{1,2}& 0& 0& -3\gamma_1\\
 0& 0& 0& 2\beta_{1,3} & 0& 0& 0& 4\beta_{2,1}& 0& -3\gamma_2\\
 0& 0& 0& 0& 2\beta_{3,1}& 0& 0& 4\beta_{2,3} & 0& -3\gamma_2\\
 0& 0& 0& 0& 2\beta_{2,1} & 0& 0& 0& 4\beta_{3,2} & -3\gamma_3\\
 0& 0& 0& 0& 0& 2\beta_{1,2} & 0& 0& 4\beta_{3,1}& -3\gamma_3
\end{bmatrix}
\in \PP_1^{12,10}.
\end{equation}
The element $\mR_1(i,j)$ in row $i$ and column $j$ gives the value of $S_{j,1}(\vx)$ in subtriangle $\PS_i$, which can be seen to be nonnegative in $\PS_i$. For instance, for the last column this follows from \eqref{eq:positivegamma}.

\subsection{Quadratic S-basis}
Next we consider the quadratic S-basis $\vs^\rmT_2 = [S_{1,2},\ldots,S_{12,2}]$. This basis is precomputed using the recurrence relation \eqref{eq:Qrec2}. With appropriate choices of the coefficients in this relation, the result of the precomputation is the matrix
\begin{equation}\label{eq:R2}
\mR_2 =\left[
\begin{array}{cccccccccccc}
 \gamma_1 & 2 \beta_2 & 0 & 0 & 0 & 0 & 0 & 0 & 0 & 0 &
   0 & 2\beta_3 \\
 0 & 0 & 0 & 2 \beta_1 & \gamma_2 & 2\beta_3 & 0 & 0 & 0 & 0
   & 0 & 0 \\
 0 & 0 & 0 & 0 & 0 & 0 & 0 & 2\beta_2 & \gamma_3 & 2 \beta_1
   & 0 & 0 \\
 0 & \beta_{1,3} & 3\beta_3 & \beta_{2,3} & 0 & 0 & 0 & 0
   & 0 & 0 & 0 & 0 \\
 0 & 0 & 0 & 0 & 0 & \beta_{2,1} & 3\beta_1 & \beta_{3,1} &
   0& 0 & 0 & 0 \\
 0 & 0 & 0 & 0 & 0 & 0 & 0 & 0 & 0 & \beta_{3,2} & 3
  \beta_2 & \beta_{1,2} \\
 0 & \frac{\beta_{1,3}}{2} & \frac{3\beta_2}{2} & 0 & 0 &
   0 & 0 & 0 & 0 & 0 & \frac{3\beta_3}{2} &
   \frac{\beta_{1,2}}{2} \\
 0 & 0 & \frac{3\beta_1}{2} & \frac{\beta_{2,3}}{2} & 0 &
   \frac{\beta_{2,1}}{2} & \frac{3\beta_3}{2} & 0 & 0 & 0
   & 0 & 0 \\
 0 & 0 & 0 & 0 & 0 & 0 & \frac{3\beta_2}{2} &
   \frac{\beta_{3,1}}{2} & 0 & \frac{\beta_{3,2}}{2}
   & \frac{3\beta_1}{2} & 0 \\
 0 & 0 & -\gamma_3 & 0 & 0 & 0 & -\gamma_1 & 0 &
   0 & 0 & -\gamma_2 & 0
\end{array}
\right] \in \PP_1^{10,12}.
\end{equation}
The element in row $i$ and column $j$ of the matrix product $\mR_1(\vx)\mR_2(\vx)$ gives the value of $S_{j,2}(\vx)$ in triangle $\PS_i$. This computation only involves nonnegative combinations of nonnegative quantities. Thus the computation of the $S_{j,2}$ is fast and stable.

\begin{remark}[Alternative quadratic S-basis]\label{rem:AlternativeQuadratic}
The basis $\vs_2$ is the unique quadratic simplex spline basis with local linear independence, as changing out any of its elements with another spline in the second row of Table \ref{tab:AllSimplexSplines} will cause the outer subtriangles $\PS_1, \PS_2,\ldots, \PS_6$ to become overloaded.

Consider the basis $\tilde{\vs}_2$ as in Table \ref{tab:dualpolynomials}, which only differs from $\vs_2$ in the entries 3,7,11, satisfying the relation
\begin{equation}\label{eq:TransformationQuadraticAlternative}
\left[
\begin{array}{c}
\frac34 \SimS{111101} \\
\frac34 \SimS{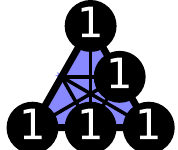} \\
\frac34 \SimS{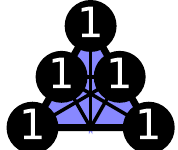}
\end{array}
\right]
=
\mT_2'^\rmT
\left[
\begin{array}{c}
\frac34 \SimS{110111} \\
\frac34 \SimS{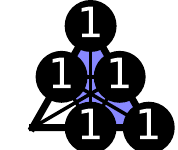} \\
\frac34 \SimS{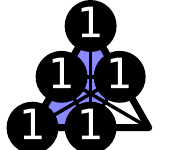}
\end{array}
\right],\qquad \mT_2'^\rmT = \begin{bmatrix}
\frac12 & 0 & \frac12\\
\frac12 & \frac12 & 0\\
0 & \frac12 & \frac12
\end{bmatrix},
\end{equation}
which follows from knot insertion \eqref{eq:Qins} at the midpoints in terms of the endpoints. Hence $\tilde{\vs}_2^\rmT = \vs_2^\rmT \mT_2 $, where $\mT_2 \in \RR^{12,12}$ is obtained from the identity matrix by replacing its principal $(3,7,11)$-submatrix by $\mT_2'$. Hence \eqref{eq:bdR}, \eqref{eq:sdr} hold, for $d = 2$, with $\vs_2$ replaced by $\tilde{\vs}_2$ and $\mR_2$ replaced by $\tilde{\mR}_2 := \mR_2 \mT_2$. 
\end{remark}

\subsection{Cubic S-basis}\label{sec:CubicBasis}
Finally we consider the cubic S-basis $\vs_3^\rmT = [S_{1,3},\ldots,S_{16,3}]$. This basis is precomputed using the recurrence relation \eqref{eq:Qrec2}. With appropriate choices of the coefficients in this relation, the result of the precomputation is the matrix
\begin{equation}\label{eq:R3}
\mR_3 = 
\left[\small
\begin{array}{cccccccccccccccc}
 \gamma_1 & 2 \beta _2 & 0 & 0 & 0 & 0 & 0 & 0 & 0 & 0 &
   0 & 2 \beta _3 & 0 & 0 & 0 & 0 \\
 0 & \beta _{1,3} & \beta _2 & 0 & 0 & 0 & 0 & 0 & 0 & 0
   & 0 & 0 & 2 \beta _3 & 0 & 0 & 0 \\
 0 & 0 & \frac{\sigma_{1,2}}{3} & 0 & 0 &
   0 & \frac{\beta _3}{3} & 0 & 0 & 0 & \frac{\beta _3}{3} &
   0 & \frac{2 \beta _1}{3} & \frac{2 \beta _2}{3} & 0 &
   \frac{\beta _3}{3} \\
 0 & 0 & \beta _1 & \beta _{2,3} & 0 & 0 & 0 & 0 & 0 & 0
   & 0 & 0 & 0 & 2 \beta _3 & 0 & 0 \\
 0 & 0 & 0 & 2 \beta _1 & \gamma_2 & 2 \beta _3 & 0 & 0 &
   0 & 0 & 0 & 0 & 0 & 0 & 0 & 0 \\
 0 & 0 & 0 & 0 & 0 & \beta _{2,1}& \beta _3 & 0 & 0 & 0
   & 0 & 0 & 0 & 2 \beta _1 & 0 & 0 \\
 0 & 0 & \frac{\beta _1}{3} & 0 & 0 & 0 & \frac{
   \sigma_{2,3}}{3} & 0 & 0 & 0 & \frac{\beta
   _1}{3} & 0 & 0 & \frac{2 \beta _2}{3} & \frac{2 \beta
   _3}{3} & \frac{\beta _1}{3} \\
 0 & 0 & 0 & 0 & 0 & 0 & \beta _2 & \beta _{3,1} & 0 & 0
   & 0 & 0 & 0 & 0 & 2 \beta _1 & 0 \\
 0 & 0 & 0 & 0 & 0 & 0 & 0 & 2 \beta _2 & \gamma_3 & 2
   \beta _1 & 0 & 0 & 0 & 0 & 0 & 0 \\
 0 & 0 & 0 & 0 & 0 & 0 & 0 & 0 & 0 & \beta _{3,2} &
   \beta _1 & 0 & 0 & 0 & 2 \beta _2 & 0 \\
 0 & 0 & \frac{\beta _2}{3} & 0 & 0 & 0 & \frac{\beta _2}{3}
   & 0 & 0 & 0 & \frac{\sigma_{1,3}}{3}&
   0 & \frac{2 \beta _1}{3} & 0 & \frac{2 \beta _3}{3} &
   \frac{\beta _2}{3} \\
 0 & 0 & 0 & 0 & 0 & 0 & 0 & 0 & 0 & 0 & \beta _3 & \beta
   _{1,2} & 2 \beta _2 & 0 & 0 & 0
\end{array} \right] \in \PP_1^{12,16}.
\end{equation}

\begin{proof}[Proof of Theorem \ref{thm:recrel}]
It remains to show the statement in the cubic case. Using the $\cG$-invariance of the basis, it suffices to show the recursion relations for the columns $j=1,2,3,13,16$ of $\mR_3(\vbeta)$. We find
\begin{align*}
S_{1,3} :=\ &\frac{1}{4}\SimS{400101} \overset{\eqref{eq:Qrec}}{=}\,
             \frac{1}{4}\beta_1^{1,4,6}\SimS{300101} = \gamma_1S_{1,2}, \\
S_{2,3} :=\ &\frac{1}{2}\SimS{310101} \overset{\eqref{eq:Qrec}}{=}\,
\frac{1}{2}\beta_1^{1,2,6}\SimS{210101} + \frac{1}{2}\beta_2^{1,2,6}\SimS{300101}
= \beta_{1,3} S_{2,2}+2\beta_2S_{1,2},\\ 
S_{3,3} :=\ &\phantom{\frac11} \SimS{221100}\overset{\eqref{eq:Qrec}}{=}\beta_1\SimS{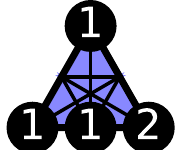}+\beta_2\SimS{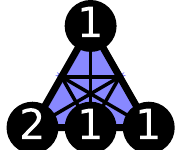}
\overset{\eqref{eq:Qins}}{=} \beta_1\left(\frac12\SimS{111110}+\frac12\SimS{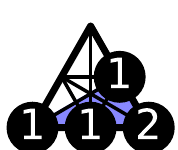}\right)+\beta_2\left(\frac12\SimS{111101}+\frac12\SimS{210101}\right)\\
\overset{\eqref{eq:Qins}}{=}&\beta_1\left(\frac14\SimS{011111}+\frac14\SimS{110111}+\frac12\SimS{120110}\right)+\beta_2\left(\frac14\SimS{101111}+\frac14\SimS{110111}+\frac12\SimS{210101}\right)\\
 =\ & \beta_1\left(\frac14\frac43S_{7,2}+\frac14\frac43S_{3,2}+S_{4,2}\right)+\beta_2\left(\frac14\frac43S_{11,2}+\frac14\frac43S_{3,2}+S_{2,2}\right)\\
 =\ &\beta_2S_{2,2}+\frac13(\beta_1+\beta_2)S_{3,2}
+\beta_1S_{4,2}+\frac13\beta_1S_{7,2}+\frac13\beta_2S_{11,2},\\
S_{13,3} :=\ &\SimS{211101}\overset{\eqref{eq:Qrec}}{=}\beta_1\SimS{111101}+\beta_2\SimS{201101}+\beta_3\SimS{210101}
\overset{\eqref{eq:Qins}}{=}\beta_1\left(\frac12\SimS{101111}+\frac12\SimS{110111}\right)+\beta_2\SimS{201101}+\beta_3\SimS{210101}\\
 =\ & \beta_1\left(\frac12\frac43S_{11,2}+\frac12\frac43S_{3,2}\right)+2\beta_2 S_{12,2}+2\beta_3 S_{2,2}
 = 2\beta_3 S_{2,2}+\frac23\beta_1S_{3,2}+\frac23\beta_1S_{11,2}+2\beta_2 S_{12,2},\\
S_{16,3} :=\ &\frac14\SimS{111111}
\overset{\eqref{eq:Qrec}}{=}\frac14\left(\beta_1\SimS{011111}+\beta_2\SimS{101111}+\beta_3\SimS{110111}\right)
= \frac14\frac43\left(\beta_1S_{7,2}+\beta_2S_{11,2}+\beta_3S_{3,2}\right).
\end{align*}

Clearly all coefficients in the recurrence relations for $S_{3,3}, S_{13,3}, S_{16,3}$ are nonnegative on $\PS$. The same holds for $S_{1,3}$ on the triangle $[\vp_1,\vp_4,\vp_6]$ and for $S_{2,3}$ on the triangle $[\vp_1,\vp_2,\vp_6]$. The remaining columns of $\mR_3$ can be found similarly, or using $\cG$-invariance. 

The final statement is easily checked by verifying that for each column $i$ with entry $\gamma_j$ the corresponding spline $S_{i,2}$ has support satisfying $\beta_j \geq \frac12$, and for each column $i$ with entry $\beta_{j,k}$ the corresponding spline $S_{i,2}$ has support satisfying $\beta_j \geq \beta_k$.
\end{proof}

\begin{remark}[No local linear independence]\label{rem:s3P9}
Since $\SS_3(\PSB)$ contains all 10 cubic polynomials on $\PS$, the basis $\vs_3$ has local linear independence (Property P9) if exactly 10 cubic S-splines in $\vs_3$ overlap each triangle $\PS_k$. Now the support of $S_{j,3}$, $j=3,7,11,13,14,15,16$, a total of 7 functions,  contains all the triangles. While the inner triangles $\PS_i$, $i=7,8,9,10,11,12$, contain the support of 10 cubic S-splines, the border triangles $\PS_i$, $i=1,2,3,4,5,6$, contain the support of 11 cubic S-splines. Hence the basis $\vs_3$ does not have local linear independence.
\end{remark}

\begin{remark}[Alternative cubic S-basis]\label{rem:AlternativeCubic}
Consider the alternative basis $\tilde{\vs}_3$, which only differs from $\vs_3$ in the entries $13,14,15,16$. From \eqref{eq:222000to111111} it follows that
\begin{equation}\label{eq:TransformationCubicAlternative}
\begin{bmatrix}
\frac34 \SimS{211101}\\ \frac34 \SimS{121110}\\ \frac34 \SimS{112011}\\ \phantom{+}\SimS{222000}
\end{bmatrix}
= \mT_3'^\rmT
\begin{bmatrix}
\phantom{+}\SimS{211101}\\ \phantom{+}\SimS{121110}\\ \phantom{+}\SimS{112011}\\ \frac14 \SimS{111111}
\end{bmatrix}
,\qquad \mT_3'^\rmT = 
\begin{bmatrix}
\frac34 & 0 & 0 & 0\\
0 & \frac34 & 0 & 0\\
0 & 0 & \frac34 & 0\\
\frac14 & \frac14 & \frac14 & 1
\end{bmatrix},
\end{equation}
and therefore $\tilde{\vs}_3^\rmT = \vs_3^\rmT \mT_3 $, where $\mT_3 \in \RR^{16,16}$ is obtained from the identity matrix by replacing its principal $(13,14,15,16)$-submatrix by $\mT_3'$. Hence \eqref{eq:bdR}, \eqref{eq:sdr} hold, for $d = 3$, with $\vs_3$ replaced by $\tilde{\vs}_3$ and $\mR_3$ replaced by $\tilde{\mR}_3 := \mR_3 \mT_3$ (but keeping $\vs_2$ and $\mR_2$ the same). The alternative basis $\tilde{\vs}_3$ does not satisfy Property P9, by Remark~\ref{rem:s3P9} and since the transformation \eqref{eq:TransformationCubicAlternative} does not change the support of the basis functions.
\end{remark}

\subsection{Fast evaluation}
Since the support of most splines in the bases $\vs_d$ only cover part of $\PS$, the evaluation procedure \eqref{eq:sdr} of $\vs_d$ (and similarly for its derivatives) for points on a given triangle can be efficiently implemented using multiplication of submatrices. For this purpose we define the index sets
\begin{equation}\label{eq:gid}
\begin{aligned}
\cG_d^k&:=\{j:\PS_k\subset\supp(S_{j,d}
)\},&\quad & k=1,\ldots,12, &\quad & d=0,1,2,3,\\
\cH_d^k&:=\{j:\mR_d(k,j)\ne 0\}, &\quad & k=1,\ldots,n_{d-1}, &\quad & d=1,2,3.
\end{aligned}
\end{equation}
Here $\cG_d^k$ encodes the splines in $\vs_d^\rmT$ that are nonzero over $\PS_k$, and $\cH_d^k$ encodes the splines in $\vs_{d-1}^\rmT$ that appear in the recurrence relation for $S_{k,d}$. In particular $\cG^k_0 = \{k\}$, and the remaining sets are listed explicitly in Table~\ref{tab:hg}. We use the symbols $\vg_d^k$ and $\vh_d^k$ for the vectors consisting of the elements in $\cG_d^k$ and $\cH_d^k$, respectively, arranged in increasing order.

\begin{table}
\begin{tabular*}{\columnwidth}{@{ }@{\extracolsep{\stretch{1}}}*{6}{c}@{ }}
\toprule
        &$\cH_1^k=\cG_1^k$&$\cH_2^k$&$\cH_3^k$&$\cG_2^k$&$\overline{\cG}_3^k$\\
\midrule
$k = 1$ & 1,6,7 & 1,2,12 & 1,2,12 & 1,2,3,10,11,12 & 1,2,10,12\\
\midrule
$k = 2$ &1,4,7&4,5,6&2,3,13&1,2,3,4,11,12&1,2,4,12\\
\midrule
$k = 3$ &2,4,8&8,9,10&3,7,11,13,14,16&2,3,4,5,6,7&2,4,5,6\\
\midrule
$k = 4$ &2,5,8&2,3,4&3,4,14&3,4,5,6,7,8&4,5,6,8\\
\midrule
$k = 5$ &3,5,9&6,7,8&4,5,6&6,7,8,9,10,11&6,8,9,10\\
\midrule
$k = 6$ &3,6,9&10,11,12&6,7,14&7,8,9,10,11,12&8,9,10,12\\
\midrule
$k = 7$ &6,7,10&2,3,11,12&3,7,11,14,15,16&2,3,7,10,11,12&2,10,12\\
\midrule
$k = 8$ &4,7,10&3,4,6,7&7,8,15&2,3,4,7,11,12&2,4,12\\
\midrule
$k = 9$ &4,8,10&7,8,10,11&8,9,10&2,3,4,6,7,11&2,4,6\\
\midrule
$k = 10$ &5,8,10&3,7,11&10,11,15&3,4,6,7,8,11&4,6,8\\
\midrule
$k = 11$ &5,9,10&&3,7,11,13,15,16&3,6,7,8,10,11&6,8,10\\
\midrule
$k = 12$ &6,9,10&&11,12,13&3,7,8,10,11,12&8,10,12\\
\bottomrule
\end{tabular*}
\caption{The sets $\cH_d^k$ and $\cG_d^k$ from \eqref{eq:gid}, with $\cG_3^k:=\overline{\cG}_3^k\cup\{3,7,11,13,14,15,16\}$}\label{tab:hg}
\end{table}

For $d=0,1,2$, it is easily verified that each $\cG_d^k$ contains $\nu_d = (d+1)(d+2)/2 = \dim \PP_d(\RR^2)$ elements. Hence $\vg_d^k=[i_1,\ldots,i_{\nu_d}]^\rmT$ with $i_1\le\cdots\le i_{\nu_d}$ (cf. Table \ref{tab:dualpolynomials}). Also note that
\begin{equation}\label{eq:gh}
\cG_1^k=\cH_1^k \qquad\text{and}\qquad \cG_2^k=\cH_2^{i_1}\cup\cH_2^{i_2}\cup\cH_2^{i_3},\quad [i_1,i_2,i_3]=\vg_1^k,\quad k=1,\ldots,12.
\end{equation}
For $d=0,1,2,3$, let $S_{j,d}^k$ be the polynomial representing $S_{j,d}$ on $\PS_k$, and let
\[ \vs_d^k=[S_{1,d}^k,\ldots,S_{n_d,d}^k]^\rmT,\qquad \overline{\vs}_d^k=\vs_d^k(\vg_d^k), \]
which represents the (ordered) vector whose elements form the set
$\cS_d^k:=\{S_{j,d}^k:j\in \cG_d^k\}$.
Next, for $1\le k\le 12$, define submatrices
$$\mR_1^k:=\mR_1(k,\vg_1^k)\in\RR^{1,3},\qquad
\mR_2^k:=\mR_2(\vg_1^k,\vg_2^k)\in\RR^{3,6},\qquad
\mR_3^k:=\mR_3(\vg_2^k,\vg_3^k)\in\RR^{6,\eta_k}$$
where $\vg_d^k$ is defined in \eqref{eq:gid}, $\eta_1 = \cdots = \eta_6 = 11$ and $\eta_7 = \cdots = \eta_{12} = 10$.

\begin{example}
Since $\vg_1^1=[1,6,7]$ and $\vg_2^1=[1,2,3,10,11,12]$ as in Table~\ref{tab:hg},
\[ \mR_1^1(\vx)\mR_2^1(\vx)=
\begin{bmatrix}\gamma_1&2\beta_{3,2}&4\beta_2\end{bmatrix}
\begin{bmatrix} \gamma_1 & 2\beta_2 & 0 & 0 & 0 & 2\beta_3\\
0 & 0 & 0 & \beta_{3,2} & 3\beta_2 & \beta_{1,2}\\
0 & \frac{\beta_{1,3}}{2} & \frac{3\beta_2}{2} & 0 & \frac{3\beta_3}{2} & \frac{\gamma_{1,2}}{2}
\end{bmatrix}.
\]
\end{example}

We are now ready to state the polynomial version of Corollary~\ref{cor:BRR}.

\begin{corollary}\label{cor:fpol}
For $d=0,1,2,3$, $k=1,\ldots,12$, coefficient vector $\vc\in \RR^{n_d}$ with subvector $\vc(\vg_d^k)$, $F_d = \vs_d^\rmT \vc \in \SS_d(\PSB)$, and $F_d^k = F_d|_{\PSsmall_k}\in \PP_d(\RR^2)$,
\begin{equation}\label{eq:b12pol}
\vs_d^{k\,\rmT} = \mR_1^k\cdots \mR_d^k,\qquad F_d^k = \mR_1^k\cdots \mR_d^k\vc(\vg_d^k).
\end{equation}
\end{corollary}

\begin{proof}
Clearly $\vs_0^k = 1$ and $F_0^k =\vc(k)$, showing the result for $d=0$.
By Corollary \ref{cor:BRR},
\[ S_{j,1}^k=\mR_1(k,j),\quad j=1,\ldots,10, \]
and  \eqref{eq:b12pol} follows for $d=1$.
Now
\begin{align*} S_{j,2}^k&=\ve_k^\rmT\mR_1\mR_2(:,j)=\sum_{i=1}^{10}\mR_1(k,i)\mR_2(i,j)\\
&=\sum_{i\in\cG_1^k}\mR_1(k,i)\mR_2(i,j)=\mR_1(k,\vg_1^k)\mR_2(\vg_1^k,j), & j=1,\ldots,12,\\
 S_{j,3}^k&=\ve_k^\rmT\mR_1\mR_2\mR_3(:,j)=\sum_{m=1}^{12}\sum_{l=1}^{10}\mR_1(k,l)\mR_2(l,m)\mR_3(m,j)\\
&=\mR_1(k,\vg_1^k)\mR_2(\vg_1^k,\vg_2^k)\mR_3(\vg_2^k,j), & j=1,\ldots,16.
\end{align*}
Hence \eqref{eq:b12pol} follows for $d=2,3$ as well.
\end{proof}

\begin{remark}
For the alternative bases $\tilde{\vs}_2$ (resp. $\tilde{\vs}_3$) the set $\cH^k_2$ (resp. $\cH^k_3$) needs to be recomputed from the modified recursion matrices $\tilde{\mR}_2$ (resp. $\tilde{\mR}_3$). The splines in $\tilde{\vs}_3$ and $\vs_3$ have identical support, so that they can be evaluated by slicing their recursion matrices using the same index vectors. In the quadratic case, $\tilde{S}_{3,2}, \tilde{S}_{7,2}, \tilde{S}_{11,2}$ have full support, as opposed to the trapezoidal support of $S_{3,2}, S_{7,2}, S_{11,2}$. Hence $\cG^1_2, \cG^2_2$ (resp. $\cG^3_2, \cG^4_2$, resp. $\cG^5_2, \cG^6_2$) need to be augmented by $\{7\}$ (resp. $\{11\}$, resp. $\{3\}$).
\end{remark}

\subsection{Derivatives}\label{sec:differentiation}
Analogous to the evaluation procedure for splines expressed in an S-basis, this section presents a formula and evaluation procedure for their (higher-order) directional derivatives (Property P6). This is achieved by applying the Leibniz rule to \eqref{eq:sdr}, and making use of special properties of the recursion matrices $\mR_i$, made precise in the following two lemmas. As before, we consider barycentric coordinates $\bfbeta$ and directional coordinates $\bfalpha$ with respect to a triangle $\PS = [\bfp_1,\bfp_2,\bfp_3]$.

\begin{lemma}\label{lem:HigherOrderDerivatives}
Let $m\geq 1$ and $f\in C^m(\cU)$, where $\cU\subset \RR^2$ is a region and $\bfx\in \cU$ with barycentric coordinates $\vbeta = (\beta_1,\beta_2,\beta_3)$. For $i = 1,\ldots, m$, consider vectors $\bfu_i\in\RR^2$ with directional coordinates $\valpha_i=(\alpha^i_1,\alpha^i_2,\alpha^i_3)$. Then
\begin{equation}\label{eq:HigherOrderDerivatives}
D_{\bfu_m}\cdots D_{\bfu_1} f = \sum_{i_1 = 1}^3 \cdots \sum_{i_m = 1}^3 \alpha^1_{i_1} \cdots \alpha^m_{i_m} \frac{\partial^m f}{\partial \beta_{i_1} \cdots \partial \beta_{i_m}}.
\end{equation}
Moreover, with $\bfalpha_1, \bfalpha_2$ directional coordinates of $\bfe_1, \bfe_2$, 
\begin{align}\label{eq:HigherOrderDerivativesPartial}
\frac{\partial^m f}{\partial x_1^{m_1} \partial x_2^{m_2}}
& = \sum_{|\bfi| = m} \sum_{\bfi_1 + \bfi_2 = \bfi} {|\bfi_1| \choose \bfi_1} {|\bfi_2| \choose \bfi_2} \bfalpha_1^{\bfi_1}  \bfalpha_2^{\bfi_2} \frac{\partial^{|\bfi|}}{\partial \bfbeta^{\bfi}} f,
\end{align}
where we used standard multi-index notation.
\end{lemma}

\begin{proof}
With $\bfx = \bfx(\beta_1, \beta_2, \beta_2) = \beta_1(\bfx)\bfp_1 + \beta_2(\bfx)\bfp_2 + \beta_3(\bfx)\bfp_3$, we can consider $f(\bfx) = f\big(\bfx(\beta_1, \beta_2, \beta_2)\big)$ as a function of $\beta_1,\beta_2,\beta_3$.

For any $t\in \RR$ and $j = 1,\ldots,m$, the barycentric coordinates of $\bfx + t\bfu_j$ are $\bfbeta + t\bfalpha_j$, implying
\begin{align}\label{eq:HigherOrderDerivatives1}
D_{\bfu_j} f(\vx)&=\left.\frac{\rmd}{\rmd t}f\big(
(\beta_1+t\alpha^j_1)\vp_1+
(\beta_2+t\alpha^j_2)\vp_2+
(\beta_3+t\alpha^j_3)\vp_3\big)\right|_{t=0}\\ \notag
&=\alpha^j_1\frac{\partial f}{\partial \beta_1}
 +\alpha^j_2\frac{\partial f}{\partial \beta_2}
 +\alpha^j_3\frac{\partial f}{\partial \beta_3}.
\end{align}
Hence the action of $D_{\bfu_j}$ on $f$ is that of the differential polynomial 
\begin{equation}\label{eq:DifferentialPolynomial}
  \alpha^j_1\frac{\partial }{\partial \beta_1}
 +\alpha^j_2\frac{\partial }{\partial \beta_2}
 +\alpha^j_3\frac{\partial }{\partial \beta_3},\qquad j = 1,\ldots, m.
\end{equation}
Since these differential polynomials commute, we can apply polynomial arithmetic to compute their product, and thus arrive at \eqref{eq:HigherOrderDerivatives}.

Next consider the standard basis vectors $\bfe_1$ and $\bfe_2$, with corresponding directional coordinates $\bfalpha_1 = (\alpha_1^1, \alpha_2^1, \alpha_3^1)$ and $\bfalpha_2 = (\alpha_1^2, \alpha_2^2, \alpha_3^2)$, and let $\{\bfu_1,\ldots, \bfu_m\} = \{\bfe_1^{m_1}\, \bfe_2^{m_2}\}$ as multisets. Then, taking the product of \eqref{eq:DifferentialPolynomial} in this case and applying the multinomial theorem twice,
\begin{align*}
\frac{\partial^m f}{\partial x_1^{m_1} \partial x_2^{m_2}}
& = \left(\sum_{|\bfi_1| = m_1} {|\bfi_1| \choose \bfi_1} \bfalpha_1^{\bfi_1} \frac{\partial^{|\bfi_1|} }{\partial \bfbeta^{\bfi_1}} \right) \left(\sum_{|\bfi_2| = m_2} {|\bfi_2| \choose \bfi_2} \bfalpha_2^{\bfi_2} \frac{\partial^{|\bfi_2|} }{\partial \bfbeta^{\bfi_2}} \right) f,
\end{align*}
from which \eqref{eq:HigherOrderDerivativesPartial} follows.
\end{proof}

\begin{lemma}\label{lem:differentiation}
For any $\vx,\vy,\vu\in\RR^2$ and $i=1,2$,
\begin{align}
\mR_i(\vx)\mR_{i+1}(\vy)   & = \mR_i(\vy)\mR_{i+1}(\vx),\label{eq:symxy}\\
(D_\vu\mR_i)\mR_{i+1}(\vx) & = \mR_i(\vx)(D_\vu\mR_{i+1}). \label{eq:symdx}
\end{align}
\end{lemma}

\begin{proof}
Fix $\vx,\vy\in\RR^2$ with barycentric coordinates $\vbeta^\vx$ and $\vbeta^\vy$, respectively. Equation \eqref{eq:symxy} will follow from 
\begin{equation}\label{eq:symto}
\mR_i(\vbeta^\vx)\mR_{i+1}(\vbeta^\vy)=
\mR_i(\vbeta^\vy)\mR_{i+1}(\vbeta^\vx),\qquad i=1,2.
\end{equation}
For $i=1$ this was proved in \cite{Cohen.Lyche.Riesenfeld13}. For $i=2$ it was checked symbolically in the Jupyter notebook. Taking the derivative with respect to $\vx$ on both sides of \eqref{eq:symxy} and setting $\vy=\vx$ we obtain \eqref{eq:symdx}.
\end{proof}

Note that, for fixed $\vu$, the matrices $D_\vu\mR_i$ and $D_\vu\mR_{i+1}$ are constant. In fact, with $\valpha=(\alpha_1,\alpha_2,\alpha_3)$ the directional coordinates of $\vu$, Lemma \ref{lem:HigherOrderDerivatives} implies
\begin{equation}\label{eq:cubicdiff3}
\mU_{d,\vu}:=D_\vu\mR_d(\vbeta)=\alpha_1\frac{\partial\mR_d(\vbeta)}{\partial\beta_1}
+\alpha_2\frac{\partial\mR_d(\vbeta)}{\partial\beta_2}
+\alpha_3\frac{\partial\mR_d(\vbeta)}{\partial\beta_3},\qquad d = 1,2,3.
\end{equation}
From the definition \eqref{eq:gammabetasigma} of $\gamma_j := 2\beta_j - 1$, $\beta_{i,j} := \beta_i - \beta_j$, and $\sigma_{i,j} := \beta_i + \beta_j$, it follows that 
\[
\frac{\partial \gamma_j}{\partial\beta_k} = 2\delta_{k,j}, \qquad
\frac{\partial \beta_{i,j}}{\partial\beta_k} = \delta_{k,i}-\delta_{k,j},\qquad
\frac{\partial \sigma_{i,j}}{\partial\beta_k} = \delta_{k,i}+\delta_{k,j}.
\]
Hence one obtains the matrix $\mU_{d,\vu}$ from $\mR_d$ by replacing
\[
\beta_j\mapsto \alpha_j, \qquad \beta_{i,j}\mapsto \alpha_{i,j} := \alpha_i - \alpha_j, \qquad \gamma_j \mapsto 2\alpha_j,\qquad \sigma_{i,j}\mapsto \tau_{i,j} := \alpha_i + \alpha_j.
\]
Analogous to the recursive evaluation \eqref{eq:sdr} of the value of $\vs_d$, there exist recursive formulas for its directional derivatives. 

\begin{theorem}\label{thm:cubicdiff}
For any point $\vx\in \RR^2$ and direction vectors $\vu, \vv, \vw\in \RR^2$,
\begin{align}
D_\vu\big(\mR_1(\vx)\mR_2(\vx)\mR_3(\vx)\big) & = 3\mR_1(\vx)\mR_2(\vx)\mU_{3,\vu},\label{eq:cubicdiff}\\
D_{\vv}D_{\vu}\big(\mR_1(\vx)\mR_2(\vx)\mR_3(\vx)\big) & = 6\mR_1(\vx)\mU_{2,\vv}\mU_{3,\vu},\label{eq:diff2}\\
D_{\vw}D_{\vv}D_{\vu}\big(\mR_1(\vx)\mR_2(\vx)\mR_3(\vx)\big) & = 6\mU_{1,\vw}\mU_{2,\vv}\mU_{3,\vu},\label{eq:diff1}
\end{align}
where in \eqref{eq:diff1} we assume that $\vx$ is not on a knot line of $\PSB$.
\end{theorem}

\begin{proof}
By the product rule
$$ D_{\vu}(\mR_1\mR_2\mR_3)=(D_\vu\mR_1)\mR_2\mR_3 +\mR_1(D_\vu\mR_2)\mR_3+\mR_1\mR_2(D_\vu\mR_3). $$
Using \eqref{eq:symdx} repeatedly we obtain \eqref{eq:cubicdiff}. Differentiating \eqref{eq:cubicdiff} using the product rule, applying \eqref{eq:symdx} and that $D_{\vv}\mU_{3,\vu}=0$, we obtain
\begin{align*}
D_{\vv}D_{\vu}\big(\mR_1\mR_2\mR_3\big)&=3\big((D_{\vv}\mR_1)\mR_2\mU_{3,\vu}+\mR_1(D_{\vv}\mR_2)\mU_{3,\vu}\big)\\ &=6\mR_1(\vx)(D_{\vv}\mR_2)\mU_{3,\vu},
\end{align*}
and \eqref{eq:diff2} follows. The proof of \eqref{eq:diff1} is similar.
\end{proof}

\begin{table}
\begin{tabular*}{\columnwidth}{@{ }@{\extracolsep{\stretch{1}}}*{8}{c}}
\toprule
$i$ & $\sigma_i$ & $S_{\sigma_i,3}$
  & $S_{\sigma_i,3}|_e$
  & $\frac13 D_{(\alpha_1,\alpha_3,\alpha_3)} S_{\sigma_i,3}|_e$
  & $\frac{1}{3\cdot 2} D^2_{(\alpha_1,\alpha_3,\alpha_3)} S_{\sigma_i,3}|_e$\\
\midrule
1 & 1 & $\frac14$ \SimS{400101}
  & $B_1^3$
  & $2\alpha_1 B_1^2$
  & $4\alpha_1^2 B_1^1$\\
2 & 2 & $\frac12$ \SimS{310101}
  & $B_2^3$
  & $2\alpha_2 B_1^2 + \alpha_{13} B_2^2$
  & $2\alpha_2 (4\alpha_1 + \alpha_2) B_1^1 + \alpha_{13}^2 B_2^1$\\
3 & 3 & \phantom{$\frac{1}{1}$}\SimS{221100}
  & $  B_3^3$
  & $\alpha_2 B_2^2 + \alpha_1 B_3^2$
  & $2\alpha_2^2 B_1^1 + 2\alpha_1\alpha_2 B_2^1 + 2\alpha_1^2 B_3^1$\\
4 & 4 & $\frac12$ \SimS{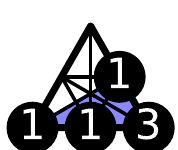}
  & $B_4^3$
  & $\alpha_{23} B_3^2 + 2\alpha_1 B_4^2$
  & $\alpha_{23}^2 B_2^1 + 2\alpha_1 (\alpha_1 + 4\alpha_2) B_3^1$  \\
5 & 5 & $\frac14$ \SimS{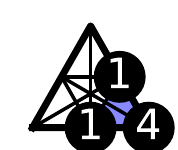}
  & $B_5^3$
  & $2\alpha_2 B_4^2$
  & $4\alpha_2^2 B_3^1$\\
\midrule
6 & 12 & $\frac12$ \SimS{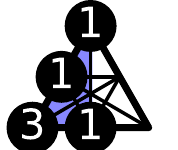}
  & 0
  & $2\alpha_3 B_1^2$
  & $2\alpha_3(4\alpha_1 + \alpha_3) B_1^1$\\
7 & 13 & \phantom{$\frac{1}{1}$} \SimS{211101}
  & 0
  & $2\alpha_3 B_2^2$
  & $8\alpha_2\alpha_3 B_1^1 + 2\alpha_3 (3\alpha_1 + \alpha_2) B_2^1$\\
8 & 14 & \phantom{$\frac{1}{1}$} \SimS{121110}
  & 0
  & $2\alpha_3 B_3^2$
  & $2\alpha_3(3\alpha_2 + \alpha_1)B_2^1 + 8\alpha_1\alpha_3 B_3^1$\\
9 & 6 & $\frac12$ \SimS{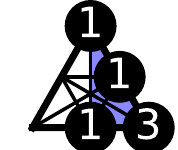}
  & 0
  & $2\alpha_3 B_4^2$
  & $2\alpha_3(4\alpha_2 + \alpha_3)B_3^1$\\
\midrule
10 & 11 & \phantom{$\frac{1}{1}$} \SimS{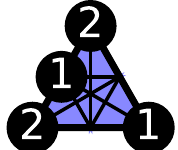}
   & 0
   & 0
   & $2\alpha_3^2 B_1^1 + \alpha_3^2 B_2^1$ \\
11 & 16 & $\frac14$ \SimS{111111}
   & 0
   & 0
   & $\alpha_3^2 B_2^1$ \\
12 & 7 & \phantom{$\frac{1}{1}$} \SimS{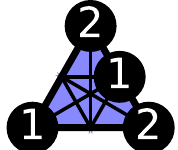}
   & 0
   & 0
   & $\alpha_3^2 B_2^1 + 2\alpha_3^2 B_3^1$ \\
\midrule
13 & 15 & \phantom{$\frac{1}{1}$} \SimS{112011}
   & 0
   & 0
   & 0\\
14 & 10 & $\frac12$ \SimS{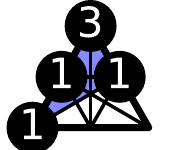}
   & 0
   & 0
   & 0\\
15 & 8  & $\frac12$ \SimS{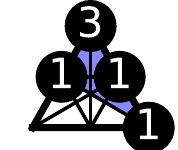}
   & 0
   & 0
   & 0\\
16 & 9  & $\frac14$ \SimS{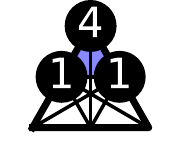}
   & 0
   & 0
   & 0\\
\bottomrule
\end{tabular*}
\caption{With $\sigma_i$ the reordering \eqref{eq:reordering}, restrictions of $S_{\sigma_i,3}$ and its directional derivatives to $e = [\bfp_1, \bfp_2]$ are expressed as linear combinations of the univariate B-splines $B^d_1,\ldots,B^d_{d+2}$.}\label{tab:derivative_restrictions}
\end{table}

Splines in $\SS_3(\PSB)$, and their directional derivatives of order $k$, restrict to univariate $C^{2-k}$-smooth splines of degree $3-k$ on each boundary edge, with a single knot at the midpoint. Hence they can, after a reparametrization \eqref{eq:restriction}, be expressed as linear combinations of the univariate B-splines $B^d_1,\ldots,B^d_{d+2}$ on the open knot multiset $\{0^{d+1}\, 0.5^1\, 1^{d+1}\}$; see Table \ref{tab:derivative_restrictions}. Here the directional derivatives $D_\vu$ are written in terms of the directional coordinates $\alpha_1,\alpha_2,\alpha_3$ of $\vu$ with respect to the triangle $\PS = [\bfp_1,\bfp_2,\bfp_3]$.

\begin{example}
The directional coordinates of $\vu$ with respect to the triangles $\PS = [\bfp_1, \bfp_2, \bfp_3]$ and $\PS' = [\bfp_1, \bfp_4, \bfp_6]$ are
\[ \vu = \alpha_1\bfp_1 + \alpha_2\bfp_2 + \alpha_3\bfp_3
       = 2\alpha_1\bfp_1 + 2\alpha_2 \bfp_4 + 2\alpha_3 \bfp_6, \quad \alpha_1 + \alpha_2 + \alpha_3 = 0.\]
Repeatedly applying the differentiation formula \eqref{eq:Qdiff} with respect to $\Delta'$,
\[
\frac13 D_\vu \SimS{400101} = 2\alpha_1 \SimS{300101}, \qquad \frac{1}{3\cdot 2} D^2_\vu \SimS{400101} = 4\alpha_1^2 \SimS{200101}.
\]
When applying \eqref{eq:restriction}, the weight of $S_{1,3}$ cancels the ratio $\frac{\area(\PSsmall)}{\area([\cK])}$, yielding the first row in Table \ref{tab:derivative_restrictions}.
\end{example}

\section{Marsden identity and ensuing properties}\label{sec:Marsden}
In this section we derive and apply Marsden identities for the bases in \eqref{eq:SS-bases}, establishing Property P3. These identities imply polynomial reproduction, i.e., $\PP_d(\PS) \subset \SS_d(\PSB)$, yield the construction of quasi-interpolants, imply stability of the bases in the $L_\infty$ norm, and yield a bound for the distance between spline values and corresponding control points.

\subsection{Derivation of the Marsden identity}
To the vertices $\bfp_j$ of $\PSB$ we associate linear polynomials
\begin{equation}\label{eq:lj}
c_j = c_j(\vy) := 1-\vp_j^\rmT\vy\in \PP_1 ,\qquad j=1,\ldots,10,
\end{equation}
which satisfy, by \eqref{eq:p456} and \eqref{eq:p78910}, 
\begin{equation}\label{eq:l456}
c_4=\frac{c_1+c_2}{2},\qquad c_5=\frac{c_2+c_3}{2},\qquad c_6=\frac{c_3+c_1}{2},\qquad c_{10}=\frac{c_1+c_2+c_3}{3}.
\end{equation}
Table~\ref{tab:dualpolynomials} introduces \emph{dual polynomials} $\Psi_{j,d}$ as $d$-fold products of these linear polynomials. Writing the $k$th factor of $\Psi_{j,d}$ as $1 - \vp_{j,d,k}^\rmT\vy$, with $\bfp_{j,d,k} \in \{\vp_1,\ldots,\vp_{10}\}$ the \emph{dual points of degree $d$}, one obtains the explicit form
\begin{equation}\label{eq:dualp}
\Psi_{j,d}(\vy):=\prod_{k=1}^d(1-\bfp^{\rmT}_{j,d,k}\vy)\in \PP_d,\qquad j=1,\ldots,n_d,\qquad d=0,1,2,3,
\end{equation}
with $n_d$ the dimension in \eqref{eq:nd}. For each basis, the dual polynomials are assembled in a vector
\begin{equation}\label{eq:dalvectors}
\vpsi_d:=[\Psi_{1,d},\ldots,\Psi_{n_d,d}]^\rmT, \qquad d=0,1,2,3.
\end{equation}

Corresponding to the dual polynomials $\Psi_{j,d}$ (or basis functions $S_{j,d}$), we define the \emph{domain points}
\[ \bfxi_{j,d} := \frac{\bfp_{j,d,1} + \cdots + \bfp_{j,d,d}}{d},\qquad j = 1,\ldots,n_d, \]
as the averages of the corresponding dual points. Figure \ref{fig:domain_mesh} shows each set of domain points, symmetrically connected in a domain mesh. Note that each set of domain points satisfies Property P4.

\begin{theorem}\label{thm:dualrec}
For $\vx,\vy\in\RR^2$ we have
\begin{equation}\label{eq:dualrec}
\mR_d(\vx)\vpsi_d(\vy)=(1-\vx^\rmT\vy)\vpsi_{d-1}(\vy),\qquad d=1,2,3,
\end{equation}
where $\mR_d(\vx)$ is given by \eqref{eq:R1},\eqref{eq:R2},\eqref{eq:R3}.
\end{theorem}

\begin{proof}
This holds for $d=1,2$ by Theorem~3.4 in \cite{Cohen.Lyche.Riesenfeld13}. Consider $d=3$. Let $\vx,\vy\in\RR^2$. Let $(\beta_1,\beta_2,\beta_3)$ be the barycentric coordinates of $\bfx$ with respect to $\PS$. From \eqref{eq:BarycentricCoordinates}, it follows
\[ \beta_1c_1(\vy) + \beta_2c_2(\vy) + \beta_3c_3(\vy) = 1 - \vx^\rmT\vy. \]
Thus, it is enough to show that 
\begin{equation}\label{eq:dualreci}
(\mR_3\vpsi_3)_i=
(\beta_1c_1+\beta_2c_2+\beta_3c_3)\Psi_{i,2},\qquad i=1\ldots,12.
\end{equation}
We verify this statement for $i=1,2,3$, by taking the product of the $i$th row of $\mR_3$ as in \eqref{eq:R3} with $\vpsi_3$ as in Table~\ref{tab:dualpolynomials}, which gives
\begin{align*}
(\mR_3\vpsi_3)_1 =\ & \gamma_1\Psi_{1,3}+2\beta_2\Psi_{2,3} + 2\beta_3\Psi_{12,3}
 =(2\beta_1-1)c_1^3+2\beta_2c_1^2 c_4 + 2\beta_3c_1^2c_6\\
\overset{\eqref{eq:l456}}{=} & \Big((\beta_1-\beta_2-\beta_3)c_1+\beta_2(c_1 +c_2) +\beta_3(c_1 +c_3)\Big)c_1^2
 =\big(\beta_1c_1+\beta_2c_2+\beta_3c_3\big)\Psi_{1,2},\\
(\mR_3\vpsi_3)_2 =\ & \beta_{1,3}\Psi_{2,3}+\beta_2\Psi_{3,3} +2\beta_3\Psi_{13,3}
 =(\beta_1-\beta_3)c_1^2c_4+\beta_2c_1 c_4c_2+2\beta_3c_1c_4c_6\\
\overset{\eqref{eq:l456}}{=}& \Big((\beta_1-\beta_3)c_1+\beta_2c_2+\beta_3(c_1 +c_3)\Big)c_1c_4
 =\big(\beta_1c_1+\beta_2c_2+\beta_3c_3\big)\Psi_{2,2},\\
3(\mR_3\vpsi_3)_3 =\ & \sigma_{1,2}\Psi_{3,3}+\beta_3\Psi_{7,3} +\beta_3\Psi_{11,3} + 2\beta_1\Psi_{13,3} + 2\beta_2\Psi_{14,3} + \beta_3\Psi_{16,3}\\
=\ & \beta_1\big(c_1c_4c_2+2c_1c_4c_6\big) + \beta_2\big(c_1c_4c_2+2c_2c_4c_5\big) + \beta_3\big(c_2c_5c_3+c_3c_6c_1+c_1c_2c_3\big)\\
\overset{\eqref{eq:l456}}{=}&3\beta_1c_1c_4c_{10} + 3\beta_2c_2c_4c_{10} + \frac12\beta_3c_3\big(c_2(c_2+c_3)+c_1c_2+ c_1(c_1+c_3)+c_1c_2\big)\\
\overset{\eqref{eq:l456}}{=}&
3\big(\beta_1c_1+\beta_2c_2+\beta_3c_3\big)\Psi_{3,2}.
\end{align*}
The remaining components are found similarly, or using $\cG$-invariance.
\end{proof}

From Theorem~\ref{thm:dualrec} we immediately obtain the following Marsden identity, generalizing Theorem~3.1 in \cite{Cohen.Lyche.Riesenfeld13}. 

\begin{corollary}[Marsden identity]\label{cor:marsdenlike}
With $S_{j,d}$ and $\Psi_{j,d}$ as in Table \ref{tab:dualpolynomials},
\begin{equation}\label{eq:marsdenlike}
(1-\vx^\rmT\vy)^d=\sum_{j=1}^{n_d}S_{j,d}(\vx)\Psi_{j,d}(\vy)=\vs_d(\vx)^\rmT\vpsi_d(\vy),\qquad d = 0,1,2,3.
\end{equation}
\end{corollary}

As was shown in \cite[Theorem 5]{Lyche.Muntingh16}, the Marsden identity can be brought into the following barycentric form, which is independent of the vertices of the triangle.

\begin{corollary}[Barycentric Marsden identity]
Let $\beta_j = \beta_j(\bfx)$, $j = 1,2,3$, be the barycentric coordinates of $\bfx\in \RR^2$ with respect to $\PS = [\bfp_1, \bfp_2, \bfp_3]$. Then \eqref{eq:marsdenlike} is equivalent to
\begin{equation}\label{eq:barycentricmarsdenlike}
(\beta_1c_1 + \beta_2c_2 + \beta_3c_3)^d = \sum_{j=1}^{n_d} S_{j,d}(\beta_1\bfp_1 + \beta_2\bfp_2 + \beta_3 \bfp_3) \Psi_j(c_1,c_2,c_3),
\end{equation}
where $\bfx\in \PS$, $c_1,c_2,c_3\in \RR$, and, for $j = 1,\ldots, n_d$, 
\[ \Psi_j(c_1, c_2, c_3) = \prod_{k = 1}^d \big( \beta_1(\bfp_{j,d,k})c_1 + \beta_2(\bfp_{j,d,k})c_2 + \beta_3(\bfp_{j,d,k})c_3 \big). \]
\end{corollary}

\begin{example}
The barycentric Marsden identity for the cubic S-basis $\vs_3$ is
\begin{align*}
& (c_1\beta_1 + c_2\beta_2 + c_3\beta_3)^3 = \left( c_1 \SimS{211000} + c_2 \SimS{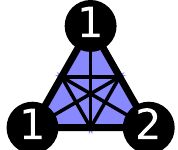} + c_3 \SimS{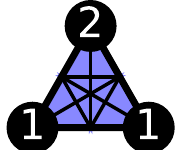}  \right)^3 = \\
\frac14 c_1c_2c_3 \SimS{111111}
& + \frac14 c_1^3 \SimS{400101} + \frac12 c_1^2c_4 \SimS{310101} + \frac12 c_2^2c_4 \SimS{130110} + c_1c_2c_4 \SimS{221100} + c_1c_4c_6 \SimS{211101} \\
& + \frac14 c_2^3 \SimS{040110} + \frac12 c_2^2c_5 \SimS{031110} + \frac12 c_3^2c_5 \SimS{013011} + c_2c_3c_5 \SimS{122010} + c_2c_4c_5 \SimS{121110} \\
& + \frac14 c_3^3 \SimS{004011} + \frac12 c_3^2c_6 \SimS{103011} + \frac12 c_1^2c_6 \SimS{301101} + c_1c_3c_6 \SimS{212001} + c_3c_5c_6 \SimS{112011}
\end{align*}
\end{example}

\begin{remark}
Substituting \eqref{eq:TransformationQuadraticAlternative} and \eqref{eq:TransformationCubicAlternative}, we also obtain Marsden identities for the alternative bases $\tilde{\vs}_d = [\tilde{S}_{j,d}]_j$, for $d=2,3$, with corresponding dual polynomials $\tilde{\vpsi}_d = [\tilde{\Psi}_{j,d}]_j$, shown in Table \ref{tab:dualpolynomials}.
\end{remark}

\begin{remark}[Are there other linear S-bases on the 12-split?]\label{rem:OtherLinear}
In order to not violate Properties P1 and P2, we can only alter $\vs_1$ by replacing
\begin{subequations}
\begin{align}
\left[\SimS{1101000001}\right]_\cG & \,\text{by}\, \left[\SimS{1111000000}\right]_\cG, \label{eq:Replacement1a}\\
\left[\SimS{1001010001}\right]_\cG & \,\text{by}\, \left[\SimS{1100110000}\right]_\cG & \text{or} & & \left[\SimS{1001010001}\right]_\cG \,\text{by}\, \left[\SimS{1001110000}\right]_\cG, \label{eq:Replacement1b}\\
\SimS{0001110001} & \,\text{by}\, \SimS{1110000001}. \label{eq:Replacement1c}
\end{align}
\end{subequations}

Regarding \eqref{eq:Replacement1c}, it can be shown from the piecewise polynomial representation (obtained through recursion) that
\begin{equation}\label{eq:Replacement1c2}
S_{10,1} := \frac14 \SimS{0001110001} = \SimS{1110000001} - \frac34\left(\frac13 \SimS{1001010001} + \frac13\SimS{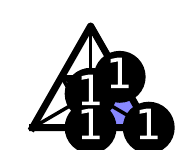} + \frac13 \SimS{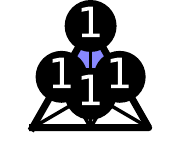}\right) = \SimS{1110000001} - \frac34(S_{7,1} + S_{8,1} + S_{9,1}).
\end{equation}
Substituting this into the Marsden identity for $\vs_1$ yields the terms
$\left(c_7 - \frac34 c_{10}\right) S_{7,1} = \frac{1}{4} c_1 S_{7,1}$ and $c_1 S_{1,1}$ with coinciding domain points, violating Property P4 when counting multiplicities.

Regarding \eqref{eq:Replacement1b}, knot insertion at the barycenter $\vp_{10}$ in terms of the midpoint $\vp_5$ and opposing corner $\vp_1$ gives
\begin{equation}\label{eq:Replacement1b1}
S_{7,1} := \frac13\SimS{1001010001} = \frac12 \SimS{1001110000} - \frac23 \frac14 \SimS{0001110001} = \frac12 \SimS{1001110000} - \frac23 S_{10,1}.
\end{equation}
Substituting this into the Marsden identity for $\vs_1$ yields the term
$\left(c_{10} - \frac23 (c_7 + c_8 + c_9) \right) S_{10,1} = - c_{10}S_{10,1}$, whose negative weight violates Property P3. Moreover, since knot insertion at the midpoints in terms of the corners gives
\begin{equation}\label{eq:Replacement1b2}
\begin{bmatrix} \SimS{1100110000}\\ \SimS{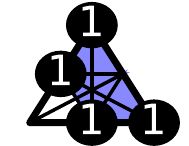}\\ \SimS{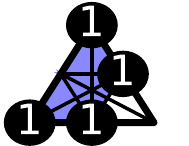} \end{bmatrix} =
\begin{bmatrix} \frac12 & \frac12 & 0\\ 0 & \frac12 & \frac12\\ \frac12 & 0 & \frac12 \end{bmatrix}
\begin{bmatrix} \SimS{1001110000}\\ \SimS{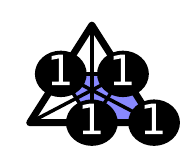}\\ \SimS{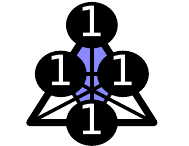} \end{bmatrix},
\end{equation}
replacing instead $\left[\SimS{1001010001}\right]_\cG$ by $\left[\SimS{1100110000}\right]_\cG$ still leaves the negative weight of $S_{10,1}$. Moreover, combining either of these replacements with \eqref{eq:Replacement1c2} instead yields a negative weight for $\SimS{1110000001}$.

Regarding \eqref{eq:Replacement1a}, knot insertion at the barycenter $\vp_{10}$ in terms of the midpoint $\vp_4$ and opposing corner $\vp_3$ gives
\begin{equation}
S_{4,1} := \frac13 \SimS{1101000001} = \SimS{1111000000} - \frac23 \SimS{1110000001}. 
\end{equation}
Substituting this into the Marsden identity for $\vs_1$, and eliminating $\SimS{1110000001}$ using \eqref{eq:Replacement1c2}, yields the term
$ \left(c_{10} - \frac23 \left(c_4 + c_5 + c_6\right) \right) S_{10,1}
= -c_{10}S_{10,1}, $
whose negative weight violates Property P3. Substituting \eqref{eq:Replacement1b1} makes this weight even smaller and using \eqref{eq:Replacement1b2} does not change this weight, while substituting \eqref{eq:Replacement1c2} instead yields a negative weight for $\SimS{1110000001}$. 
\end{remark}

\begin{remark}[Are there other quadratic S-bases on the 12-split?]\label{rem:OtherQuadratic}
In order to not violate Properties P1 and P2, Table \ref{tab:AllSimplexSplines} shows that our only option is to replace
\begin{equation}\label{eq:Replacement2}
\left[\SimS{210101}\right]_\cG \qquad \text{by}\qquad \left[\SimS{211100}\right]_\cG,
\end{equation}
in either the basis $\vs_2$ or $\tilde{\vs}_2$. Let us consider the latter case. Knot insertion at the midpoints in terms of the endpoints gives
\begin{equation}
S_{2,2} = \tilde{S}_{2,2} := \frac12 \SimS{210101} = \SimS{211100} - \frac12 \SimS{111101} = \SimS{211100} - \frac23 \tilde{S}_{3,2}.
\end{equation}
Inserting this equation in the Marsden identity for $\tilde{\vs}_2$ yields the Marsden identity for the new basis,
\begin{align*}
(c_1 \beta_1 + c_2\beta_2 + c_3\beta_3)^2 = 
& + \frac14 \left[ c_1^2 \SimS{300101} + c_2^2 \SimS{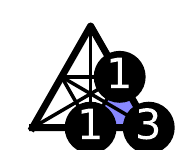} + c_3^2 \SimS{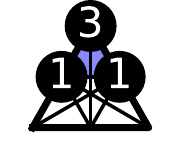} \right]
-\frac14 \left[ c_1^2 \SimS{111101} + c_2^2 \SimS{111110} + c_3^2 \SimS{111011} \right] \\
& + \left[ c_1c_4 \SimS{211100} + c_1c_6 \SimS{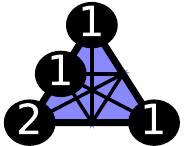} + c_2c_4 \SimS{121100} + c_2c_5\SimS{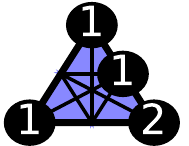} + c_3c_5 \SimS{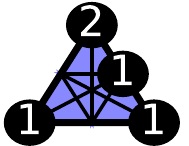} + c_3c_6 \SimS{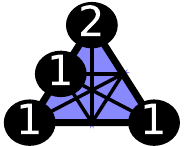} \right].
\end{align*}
Hence the new basis does not form a positive partition of unity. Moreover, since the matrix in \eqref{eq:TransformationQuadraticAlternative} is nonnegative, the same holds for the basis obtained by making the replacement \eqref{eq:Replacement2} in $\vs_2$.
\end{remark}

\begin{remark}[Are there other cubic S-bases on the 12-split?]\label{rem:OtherCubic}
Similar to the quadratic case, our only option is to replace
\begin{equation}\label{eq:Replacement3}
\left[\SimS{310101}\right]_\cG \qquad \text{by}\qquad \left[\SimS{311100}\right]_\cG
\end{equation}
in either the basis $\vs_3$ or $\tilde{\vs}_3$. Analogous to Remark \ref{rem:OtherQuadratic}, the latter case yields a basis without positive partition of unity. For the basis obtained by making the replacement \eqref{eq:Replacement3} in $\vs_3$, a similar calculation yields dual polynomials without linear factors, violating Property P3.
\end{remark}

\begin{table}[th!]
\begin{tabular*}{\columnwidth}{@{ }@{\extracolsep{\stretch{1}}}*{9}{c}@{ }}
\toprule
 & $j=1$ & $j=2$ & $j=3$      & $j=4$      & $j=5$   & $j=6$   & $j=7$ & $j=8$\\
\midrule 
$S_{j,1}$ & $\frac14$\SimS{2001010000}\! & $\frac14$\SimS{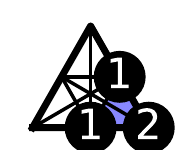}\! & $\frac14$\SimS{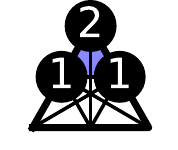}\! & $\frac13$\SimS{1101000001}\! & $\frac13$\SimS{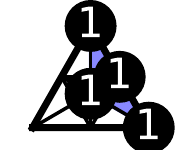}\! & $\frac13$\SimS{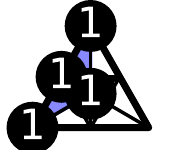}\! & $\frac13$\SimS{1001010001}\! & $\frac13$\SimS{0101100001}\!\\
$\Psi_{j,1}$ & $c_1$   & $c_2$      & $c_3$      & $c_4$      & $c_5$   & $c_6$   &$c_7$&$c_8$\\
$\bfbeta(\bfxi_{j,1})$ & $(1\,0\,0)$ & $(0\,1\,0)$ & $(0\,0\,1)$ & $(\frac12\,\frac12\,0)$ & $(0\,\frac12\,\frac12)$ & $(\frac12\,0\,\frac12)$ & $(\frac12\,\frac14\,\frac14)$ & $(\frac14\,\frac12\,\frac14)$\\
\midrule
$S_{j,2}$ & $\frac14$\SimS{300101}\! & $\frac12$\SimS{210101}\! & $\frac34$\SimS{110111}\! & 
$\frac12$\SimS{120110}\! & $\frac14$\SimS{030110}\! & $\frac12$\SimS{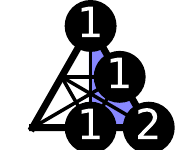} & $\frac34$\SimS{011111}\! & $\frac12$\SimS{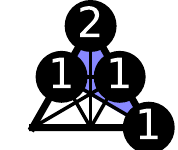}\!\\
$\Psi_{j,2}$ & $c_1^2$ & $c_1c_4$   & $c_4c_{10}$ & $c_2c_4$   & $c_2^2$ & $c_2c_5$&$c_5c_{10}$&$c_3c_5$\\
$\bfbeta(\bfxi_{j,2})$ & $(1\,0\,0)$ & $(\frac34\,\frac14\,0)$ & $(\frac{5}{12}\, \frac{5}{12}\, \frac16)$ & $(\frac14\,\frac34\,0)$ & $(0\,1\,0)$ & $(0\,\frac34\,\frac14)$ & $(\frac16\,\frac{5}{12}\,\frac{5}{12})$ & $(0\,\frac14\,\frac34)$\\
\midrule
$\tilde{S}_{j,2}$ & $\frac14$\SimS{300101}\! & $\frac12$\SimS{210101}\! & $\frac34$\SimS{111101}\! & 
$\frac12$\SimS{120110}\! & $\frac14$\SimS{030110}\! & $\frac12$\SimS{021110} & $\frac34$\SimS{111110}\! & $\frac12$\SimS{012011}\!\\
$\tilde{\Psi}_{j,2}$ & $c_1^2$ & $c_1c_4$   & $c_1c_{10}$& $c_2c_4$   & $c_2^2$ & $c_2c_5$&$c_2c_{10}$&$c_3c_5$\\
$\bfbeta(\tilde{\bfxi}_{j,2})$ & $(1\,0\,0)$ & $(\frac34\,\frac14\,0)$ & $(\frac46\, \frac16\, \frac16)$ & $(\frac14\,\frac34\,0)$ & $(0\,1\,0)$ & $(0\,\frac34\,\frac14)$ & $(\frac16\,\frac46\,\frac16)$ & $(0\,\frac14\,\frac34)$\\\midrule
$S_{j,3}$ & $\frac14$\SimS{400101}\! & $\frac12$\SimS{310101}\! & \SimS{221100}\! & $\frac12$\SimS{130110}\! & $\frac14$\SimS{040110}\! & $\frac12$\SimS{031110}\! & \SimS{122010}\! & $\frac12$\SimS{013011}\!\\
$\Psi_{j,3}$ & $c_1^3$ & $c_1^2c_4$ & $c_1c_2c_4$& $c_2^2c_4$ & $c_2^3$ & $c_2^2c_5$&$c_2c_3c_5$&$c_3^2c_5$\\
$\bfbeta(\bfxi_{j,3})$ & $(1\,0\,0)$ & $(\frac56\,\frac16\,0)$ & $(\frac12\,\frac12\,0)$ & $(\frac16\,\frac56\,0)$ & $(0\,1\,0)$ & $(0\,\frac56\,\frac16)$ & $(0\,\frac12\,\frac12)$ & $(0\,\frac16\,\frac56)$ \\
\midrule
$\tilde{S}_{j,3}$ &
$\frac14$\SimS{400101}\! & $\frac12$\SimS{310101}\! & \SimS{221100}\! & $\frac12$\SimS{130110}\! & $\frac14$\SimS{040110}\! & $\frac12$\SimS{031110}\! & \SimS{122010}\! & $\frac12$\SimS{013011}\\
$\tilde{\Psi}_{j,3}$ & $c_1^3$ & $c_1^2c_4$ & $c_1c_2c_4$& $c_2^2c_4$ & $c_2^3$ & $c_2^2c_5$&$c_2c_3c_5$&$c_3^2c_5$ \\
$\bfbeta(\tilde{\bfxi}_{j,3})$ & $(1\,0\,0)$ & $(\frac56\,\frac16\,0)$ & $(\frac12\,\frac12\,0)$ & $(\frac16\,\frac56\,0)$ & $(0\,1\,0)$ & $(0\,\frac56\,\frac16)$ & $(0\,\frac12\,\frac12)$ & $(0\,\frac16\,\frac56)$ \\
\midrule
 &$j=9$&$j=10$&$j=11$&$j=12$&$j=13$&$j=14$&$j=15$&$j=16$\\
\midrule
$S_{j,1}$ & $\frac13$\SimS{0010110001}\! & $\frac14$\SimS{0001110001}\\
$\Psi_{j,1}$ & $c_9$  & $c_{10}$\\
$\bfbeta(\bfxi_{j,1})$ & $(\frac14\,\frac14\,\frac12)$ & $(\frac13\,\frac13\,\frac13)$ \\
\midrule
$S_{j,2}$ & $\frac14$\SimS{003011}\! & $\frac12$\SimS{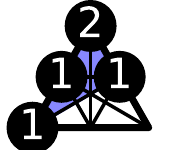}\! & $\frac34$\SimS{101111} & $\frac12$\SimS{201101}\!\\
$\Psi_{j,2}$ & $c_3^2$&$c_3c_6$  &$c_6c_{10}$ & $c_1c_6$\\
$\bfbeta(\bfxi_{j,2})$ & $(0\,0\,1)$ & $(\frac14\,0\,\frac34)$ & $(\frac{5}{12}\,\frac{5}{12}\,\frac16)$ & $(\frac34\,0\,\frac14)$\\
\midrule
$\tilde{S}_{j,2}$ & $\frac14$\SimS{003011}\! & $\frac12$\SimS{102011}\! & $\frac34$\SimS{111011} & $\frac12$\SimS{201101}\!\\
$\tilde{\Psi}_{j,2}$ &$c_3^2$&$c_3c_6$  &$c_3c_{10}$&$c_1c_6$\\
$\bfbeta(\tilde{\bfxi}_{j,2})$ & $(0\,0\,1)$ & $(\frac14\,0\,\frac34)$ & $(\frac16\,\frac16\,\frac46)$ & $(\frac34\,0\,\frac14)$\\
\midrule
$S_{j,3}$ & $\frac14$\SimS{004011}\! & $\frac12$\SimS{103011}\! & \SimS{212001}\! & $\frac12$\SimS{301101}\! & \SimS{211101}\! & \SimS{121110}\! & \SimS{112011}\! & $\frac14$\SimS{111111}\\
$\Psi_{j,3}$ &$c_3^3$&$c_3^2c_6$&$c_1c_3c_6$&$c_1^2c_6$&$c_1c_4c_6$&$c_2c_4c_5$&$c_3c_5c_6$&$c_1c_2c_3$\\
$\bfbeta(\bfxi_{j,3})$ & $(0\,0\,1)$ & $(\frac16\,0\,\frac56)$ & $(\frac12\,0\,\frac12)$ & $(\frac56\,0\,\frac16)$ & $(\frac23\,\frac16\,\frac16)$ & $(\frac16\,\frac23\,\frac16)$ & $(\frac16\,\frac16\,\frac23)$ & $(\frac13\,\frac13\,\frac13)$\\
\midrule
$\tilde{S}_{j,3}$ & $\frac14$\SimS{004011}\! & $\frac12$\SimS{103011}\! & \SimS{212001}\! & $\frac12$\SimS{301101}\! & $\frac34$\SimS{211101}\! & $\frac34$\SimS{121110}\! & $\frac34$\SimS{112011}\! & \SimS{222000}\\
$\tilde{\Psi}_{j,3}$ &$c_3^3$&$c_3^2c_6$&$c_1c_3c_6$&$c_1^2c_6$& $c_1^2c_{10}$ & $c_2^2c_{10}$ & $ c_3^2 c_{10}$ & $c_1c_2c_3$\\
$\bfbeta(\tilde{\bfxi}_{j,3})$ & $(0\,0\,1)$ & $(\frac16\,0\,\frac56)$ & $(\frac12\,0\,\frac12)$ & $(\frac56\,0\,\frac16)$ & $(\frac79\,\frac19\,\frac19)$ & $(\frac19\,\frac79\,\frac19)$ & $(\frac19\,\frac19\,\frac79)$ & $(\frac13\,\frac13\,\frac13)$\\
\bottomrule
\end{tabular*}
\caption{Basis functions $S_{j,d}, \tilde{S}_{j,d}$, domain points $\bfxi_{j,d}, \tilde{\bfxi}_{j,d}$, and corresponding dual polynomials $\Psi_{j,d}, \tilde{\Psi}_{j,d}$, factored into linear polynomials $c_j$ as in \eqref{eq:lj}.}\label{tab:dualpolynomials}
\vspace{-2.3em}
\end{table}

\subsection{Polynomial reproduction}
The barycentric Marsden identity can directly be applied to express Bernstein polynomials on $\PS$ in terms of the S-basis. In particular, applying the multinomial theorem to the left hand side of \eqref{eq:barycentricmarsdenlike}, one notices that the Bernstein polynomial $B^d_{i_1,i_2,i_3}$ appears as the coefficient of $c_1^{i_1} c_2^{i_2} c_3^{i_3}$. Hence, defining the ``coefficient of'' operator \cite{Knuth94}
\[ [c_1^{i_1} c_2^{i_2} c_3^{i_3}] F := \frac{1}{i_1!i_2!i_3!} \frac{\partial^{i_1 + i_2 + i_3} F}{\partial c_1^{i_1} \partial c_2^{i_2} \partial c_3^{i_3}}(0,0,0) \]
for any formal power series $F(c_1,c_2,c_3)$ and nonnegative integers with sum $i_1 + i_2 + i_3 = d$,
\[ B^d_{i_1,i_2,i_3} = \sum_{i=1}^{n_d} S_{j,d}(\beta_1\bfp_1 + \beta_2\bfp_2 + \beta_3 \bfp_3) [c_1^{i_1} c_2^{i_2} c_3^{i_3}] \Psi_j (c_1,c_2,c_3). \]
Thus one immediately sees from the monomials in the dual polynomials which simplex splines appear in the above linear combination. For instance, substituting the dual polynomials from Table \ref{tab:dualpolynomials} and the short-hands \eqref{eq:l456}, one obtains
\begin{align}
B^3_{300} & = \SimS{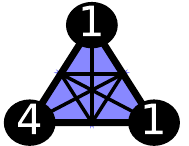} = \frac14 \SimS{400101} + \frac14 \SimS{310101} + \frac14 \SimS{301101} + \frac14 \SimS{211101}, \\
B^3_{210} & = \SimS{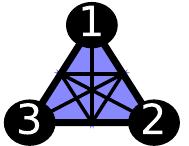} = \frac14 \SimS{310101} + \frac12 \SimS{221100} + \frac14 \SimS{211101}, \\
B^3_{111} & = \SimS{222000} = \frac14 \SimS{211101} + \frac14 \SimS{121110} + \frac14 \SimS{112011} + \frac14 \SimS{111111},\label{eq:reproductionB111}
\end{align}
consistent with the result achieved by repeatedly applying knot insertion, see for instance \eqref{eq:222000to111111}.

\subsection{Quasi-interpolation}\label{sec:quasi-interpolant}
Based on a standard construction, the Marsden identity gives rise to quasi-interpolants in terms of the de Boor-Fix functionals. In this section, we present a different quasi-interpolant that solely involves point evaluations at averages of dual points (Property P7).

Let $(\beta_1,\beta_2,\beta_3)$ be the barycentric coordinates with respect to the triangle $\PS$. As explained in \cite[\S 6.1]{Lyche.Muntingh16}, the Bernstein polynomial $B^d_{i_1,i_2,i_3}$ can be expressed in terms of the simplex spline basis by replacing each dual polynomial in \eqref{eq:marsdenlike}, after substituting \eqref{eq:l456}, by its coefficient of $c_1^{i_1} c_2^{i_2} c_3^{i_3}$.

\begin{theorem}\label{thm:QI}
For $d = 1,2,3$ and each basis in \eqref{eq:SS-bases} with dual points $\bfp_{j,d,k}$, consider the map
\begin{equation}\label{eq:quasi-interpolant}
Q_d: C^0(\PS)\longrightarrow \SS_d(\PSB),\qquad Q_d(F) = \sum_{j=1}^{n_d} l_{j,d}(F) S_{j,d},
\end{equation}
where the functionals $l_{j,d}: C^0(\PS)\longrightarrow \RR$ are given by
\[ l_{j,d}(F) := \sum_{m = 1}^d \frac{m^d}{d!} (-1)^{d-m} \sum_{1\leq k_1 < \cdots < k_m\leq d} F\left( \frac{\bfp_{j,d,k_1} + \cdots + \bfp_{j,d,k_m}}{m} \right). \]
Then $Q_d$ is a quasi-interpolant reproducing polynomials up to degree $d$.
\end{theorem}

\begin{proof}
For $d = 1,2$ the statement is shown in \cite[\S 6.1]{Cohen.Lyche.Riesenfeld13}. We give an explicit proof for the remaining case $d = 3$, analogous to the proof provided in \cite{Lyche.Merrien18}. In that case
$[l_{1,3}(f),\ldots, l_{16,3}(f)] = [\frac16, \frac16, \frac16, -\frac43, -\frac43, -\frac43, \frac92] f(\mM)$, applied component-wise, where
\begin{align*}
\mM
& =
\begin{bmatrix}
\bfp_{1,3,1} & \cdots & \bfp_{16,3,1}\\
\bfp_{1,3,2} & \cdots & \bfp_{16,3,2}\\
\bfp_{1,3,3} & \cdots & \bfp_{16,3,3}\\ 
\frac{\bfp_{1,3,1} + \bfp_{1,3,2}}{2} & \cdots & \frac{\bfp_{16,3,1} + \bfp_{16,3,2}}{2}\\
\frac{\bfp_{1,3,1} + \bfp_{1,3,3}}{2} & \cdots & \frac{\bfp_{16,3,1} + \bfp_{16,3,3}}{2}\\
\frac{\bfp_{1,3,2} + \bfp_{1,3,3}}{2} & \cdots & \frac{\bfp_{16,3,2} + \bfp_{16,3,3}}{2}\\
\frac{\bfp_{1,3,1} + \bfp_{1,3,2} + \bfp_{1,3,3}}{3} & \cdots & \frac{\bfp_{16,3,1} + \bfp_{16,3,2} + \bfp_{16,3,3}}{3}
\end{bmatrix} =
\end{align*}
$$ \left[
\begin{array}{cccccccccccccccc}
\vl_1&    \vl_1&    \vl_1&    \vl_5& \vl_5&    \vl_5&    \vl_5&    \vl_9& \vl_9&  \vl_{9}&    \vl_9&    \vl_1&    \vl_1&    \vl_5&    \vl_9&    \vl_1 \\
\vl_1&    \vl_1&    \vl_5&    \vl_5& \vl_5&    \vl_5&    \vl_9&    \vl_9& \vl_9&  \vl_{9}&    \vl_1&    \vl_1&    \vl_3&    \vl_7& \vl_{11}&    \vl_5 \\
\vl_1&    \vl_3&    \vl_3&    \vl_3& \vl_5&    \vl_7&    \vl_7&    \vl_7& \vl_9& \vl_{11}& \vl_{11}& \vl_{11}& \vl_{11}&    \vl_3&    \vl_7&    \vl_9 \\
\vl_1&    \vl_1&    \vl_3&    \vl_5& \vl_5&    \vl_5&    \vl_7&    \vl_9& \vl_9&  \vl_{9}& \vl_{11}&    \vl_1& \vl_{17}& \vl_{19}& \vl_{21}&    \vl_3 \\
\vl_1& \vl_{17}& \vl_{17}& \vl_{18}& \vl_5& \vl_{19}& \vl_{19}& \vl_{20}& \vl_9& \vl_{21}& \vl_{21}& \vl_{22}& \vl_{22}& \vl_{18}& \vl_{20}& \vl_{11} \\
\vl_1& \vl_{17}& \vl_{18}& \vl_{18}& \vl_5& \vl_{19}& \vl_{20}& \vl_{20}& \vl_9& \vl_{21}& \vl_{22}& \vl_{22}& \vl_{23}& \vl_{24}& \vl_{25}&    \vl_7 \\
\vl_1&    \vl_2&    \vl_3&    \vl_4& \vl_5&    \vl_6&    \vl_7&    \vl_8& \vl_9& \vl_{10}& \vl_{11}& \vl_{12}& \vl_{13}& \vl_{14}& \vl_{15}& \vl_{16} \\
\end{array}
\right]
$$
with $\vl_j = \bfxi_{j,3}$ for $j = 1, \ldots, 16$ the domain points, and with the quarterpoints
\begin{align*}
\vl_{17} = \frac{\vp_1 + \vp_4}{2},\qquad 
\vl_{18} = \frac{\vp_2 + \vp_4}{2},\qquad
\vl_{19} = \frac{\vp_2 + \vp_5}{2},\\
\vl_{20} = \frac{\vp_3 + \vp_5}{2},\qquad
\vl_{21} = \frac{\vp_3 + \vp_6}{2},\qquad
\vl_{22} = \frac{\vp_1 + \vp_6}{2},\\
\vl_{23} = \frac{\vp_4 + \vp_6}{2},\qquad
\vl_{24} = \frac{\vp_4 + \vp_5}{2},\qquad
\vl_{25} = \frac{\vp_5 + \vp_6}{2},
\end{align*}
as in Figure \ref{fig:ControlNet-3}. To prove that $Q_3$ reproduces polynomials up to degree 3, i.e., $Q_d(B^3_{ijk}) = B^3_{ijk}$ whenever $i + j + k = 3$, it suffices to show this for $B^3_{300}, B^3_{210}, B^3_{111}$ using the symmetries. Evaluating these polynomials at the dual point averages yields Table \ref{tab:evaluations_at_dual_point_averages}.

\begin{table}
\begin{tabular*}{\columnwidth}{@{ }@{\extracolsep{\stretch{1}}}*{14}{c}}
\toprule
           & $\vl_{1}$ & $\vl_{2}$ & $\vl_{3}$ & $\vl_{4}$ & $\vl_{5}$ & $\vl_{6}$ & $\vl_{7}$ & $\vl_{8}$ & $\vl_{9}$ & $\vl_{10}$ & $\vl_{11}$ & $\vl_{12}$ & $\vl_{13}$\\
\SimS{411000} & 1 & $\frac{125}{216}$ & $\frac{1}{8}$ & $\frac{1}{216}$ & 0 & 0 & 0 & 0 & 0 & $\frac{1}{216}$ & $\frac{1}{8}$ & $\frac{125}{216}$ & $\frac{8}{27}$\\
\SimS{321000} & 0 & $\frac{25}{72}$ & $\frac{3}{8}$ & $\frac{5}{72}$ & 0 & 0 & 0 & 0 & 0 & 0 & 0 & 0 & $\frac{2}{9}$\\
\SimS{222000} & 0 & 0 & 0 & 0 & 0 & 0 & 0 & 0 & 0 & 0 & 0 & 0 & $\frac{1}{9}$\\
\midrule
           & $\vl_{14}$ & $\vl_{15}$ & $\vl_{16}$ & $\vl_{17}$ & $\vl_{18}$ & $\vl_{19}$ & $\vl_{20}$ & $\vl_{21}$ & $\vl_{22}$ & $\vl_{23}$ & $\vl_{24}$ & $\vl_{25}$\\
\SimS{411000} & $\frac{1}{216}$ & $\frac{1}{216}$ & $\frac{1}{27}$ & $\frac{27}{64}$ & $\frac{1}{64}$ & 0 & 0 & $\frac{1}{64}$ & $\frac{27}{64}$ & $\frac{1}{8}$ & $\frac{1}{64}$ & $\frac{1}{64}$\\
\SimS{321000} & $\frac{1}{18}$ & $\frac{1}{72}$ & $\frac{1}{9}$ & $\frac{27}{64}$ & $\frac{9}{64}$ & 0 & 0 & 0 & 0 & $\frac{3}{16}$ & $\frac{3}{32}$ & $\frac{3}{64}$\\
\SimS{222000} & $\frac{1}{9}$ & $\frac{1}{9}$ & $\frac{2}{9}$ & 0 & 0 & 0 & 0 & 0 & 0 & $\frac{3}{16}$ & $\frac{3}{16}$ & $\frac{3}{16}$\\
\bottomrule
\end{tabular*}
\caption{The values of the Bernstein polynomials $B^3_{300}, B^3_{210}, B^3_{111}$ at the dual point averages $\vl_1,\ldots,\vl_{25}$.}\label{tab:evaluations_at_dual_point_averages}
\end{table}

For instance, the Bernstein polynomial $B^3_{111}$ is only nonzero for $\vl_{13},\vl_{14},\vl_{15},$ $\vl_{16}$ and $\vl_{23}, \vl_{24}, \vl_{25}$. Hence only the entries in the last 4 columns and last 2 rows in $f(\mM)$ can be nonzero, yielding the coefficients
\begin{align*}
\begin{bmatrix} l_{13,3} & \cdots & l_{16,3} \end{bmatrix} B^3_{111} & = 
\begin{bmatrix}-\frac43 & \frac92 \end{bmatrix}
B^3_{111}\left(
\begin{bmatrix}
\vl_{23} & \vl_{24} & \vl_{25} & \vl_7\\
\vl_{13} & \vl_{14} & \vl_{15} & \vl_{16}
\end{bmatrix}\right)\\
& =
\begin{bmatrix}-\frac43 & \frac92 \end{bmatrix}
\begin{bmatrix}
\frac{3}{16} & \frac{3}{16} & \frac{3}{16} & 0\\
\frac19 & \frac19 & \frac19 & \frac29
\end{bmatrix} = 
\begin{bmatrix}
\frac14 & \frac14 & \frac14 & 1
\end{bmatrix},
\end{align*}
consistent with \eqref{eq:reproductionB111}. Similarly one establishes reproduction of $B^3_{300}$ and $B^3_{210}$; additional details are shown in the Jupyter notebook.
\end{proof}

The quasi-interpolant $Q_d$ is bounded independently of the geometry of $\PS$, since, using that $\vs_d$ forms a partition of unity,
\[ \|Q_d(F)\|_{L_\infty(\PSsmall)}
\leq \max_j |l_{j,d}(F)|
\leq C_d \|F\|_{L_\infty(\PSsmall)},\qquad C_d = \sum_{m=1}^d \frac{m^d}{d!}{d\choose m}.
\]
In particular, $[C_1, C_2, C_3] = [1, 3, 9]$. Therefore, by a standard argument, $Q_d$ is a quasi-interpolant that approximates locally with order 4 smooth functions whose first four derivatives are in $L_\infty(\PS)$. 

\begin{remark}
In \cite[Lemma 6.1]{Cohen.Lyche.Riesenfeld13} it was shown, for $d = 1,2$, that the functionals $l_{j,d}$ form the dual basis to $\vs_d$, i.e., $l_{j,d}(S_{i,d}) = \delta_{ij}$. This is equivalent to the statement that $Q_d$ reproduces all splines in $\SS_d(\PSB)$.

However, for $d = 3$ this is not the case. For instance, repeatedly using the recurrence relation \eqref{eq:Qrec2} with respect to the triangle $\Delta' := [\bfp_1, \bfp_4, \bfp_6]$, 
\[ S_{1,3} = \frac14 \SimS{400101} = \frac14 \beta_1^{1,4,6} \SimS{300101} = \cdots = \frac14 (\beta_1^{1,4,6})^3 \SimS{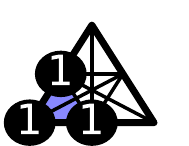} = (\beta_1^{1,4,6})^3 \bfone_{\Delta'}. \]
As immediately seen from the support of $S_{1,3}$ and Figure \ref{fig:ControlNet-3}, $S_{1,3}(\vl_i) \neq 0$ for $i\neq 1,2,12,13,17,22$. Hence $l_{16,3}(S_{1,3}) = \frac16 S_{1,3}(\vl_1) = \frac16 \neq 0$.
\end{remark}

\begin{remark}
Although $Q_d$ involves $n_d$ dual functionals each involving a sum with $\sum_{m = 1}^d {m\choose d}$ terms, many of the dual point averages (and hence the point evaluations) coincide. In particular the quasi-interpolant for the linear basis $\vs_1$ involves 10 point evaluations at the domain points. The quadratic bases $\vs_2, \tilde{\vs}_2$ (respectively cubic bases $\vs_3, \tilde{\vs}_3$) involve 16 (respectively 25) point evaluations, whose carriers are shown in Figure \ref{fig:domain_mesh}.
\end{remark}

\subsection{$L_\infty$ stability and distance to the control points}
Each basis $\vs$ in \eqref{eq:SS-bases} is stable in the $L_\infty$ norm with a condition number bounded independent of the geometry of $\PS$.

\begin{theorem}
Let $F = \vs^\rmT \bfc$ with $\vs$ as in \eqref{eq:SS-bases}. There is a constant $\kappa>0$ independent of the geometry of $\PS$, such that
\begin{equation}\label{eq:stability}
\kappa^{-1}\|\bfc\|_\infty \leq \|F\|_{L_\infty(\PSsmall)} \le \|\bfc\|_\infty.
\end{equation}
\end{theorem}
\begin{proof}
For $\vs = \vs_0, \vs_1, \vs_2$ this was shown in \cite[Theorem 6.2]{Cohen.Lyche.Riesenfeld13}, with the best possible constants. It remains to show this for $\vs = [S_1, \ldots, S_{n_d}] = \tilde{\vs}_2, \vs_3, \tilde{\vs}_3$. Let $\vlambda = [\lambda_1,\ldots,\lambda_{n_d}]^\rmT$ be the point evaluations at the corresponding domain points $\bfxi = [\bfxi_1,\ldots,\bfxi_{n_d}]$.

Applying these functionals to $F = \vs^\rmT \bfc$ yields a system $\bff := \vlambda F = \mM \bfc$, with \emph{collocation matrix} $\mM := \vlambda \vs^\rmT = [S_j(\xi_i)]_{i,j=1}^{n_d}$. A computation (see the Jupyter notebook) shows that each collocation matrix is nonsingular, and its elements are rational numbers independent of the geometry of $\PS$. Note that $\bfc = \mM^{-1} \bff$ are the coefficients of the Lagrange interpolant of the function values $\bff$ at the domain points $\bfxi$. Hence, since $\vs$ forms a partition of unity and therefore $\|\mM\|_\infty = 1$, \eqref{eq:stability} holds with
$\kappa := \|\mM\|_\infty \cdot \|\mM^{-1}\|_\infty = \|\mM^{-1}\|_\infty$.
\end{proof}

An exact computation in the Jupyter notebook shows that the $L_\infty$ condition numbers for the collocation matrices $\mM_1, \mM_2,\tilde{\mM}_2, \mM_3, \tilde{\mM}_3$ corresponding to the bases $\vs_1, \vs_2,\tilde{\vs}_2, \vs_3, \tilde{\vs}_3$ and domain points $\bfxi_1, \bfxi_2,\tilde{\bfxi}_2, \bfxi_3, \tilde{\bfxi}_3$ are
\begin{equation}
\begin{aligned}
                  \kappa_1 & = 1, \qquad
&                 \kappa_2 & = \frac{28}{9}  \approx  3.111, \qquad &         \kappa_3 & = \frac{415}{8}      =    51.875,\\
&\qquad & \tilde{\kappa}_2 & = \frac{295}{9} \approx 32.778, \qquad & \tilde{\kappa}_3 & = \frac{1297}{17} \approx 76.294.
\end{aligned} 
\end{equation}

Using a standard argument (presented in \cite[Corollary 1]{Lyche.Muntingh16}), one obtains an $O(h^2)$ bound for the distance between values of a spline function $F = \vs^\rmT \bfc$, with $\vs = \vs_2, \tilde{\vs}_2, \vs_3, \tilde{\vs}_3$, and its control points.

\begin{corollary}
Let $h$ be the longest edge in $\PS$. With $\vs = \vs_2, \tilde{\vs}_2, \vs_3, \tilde{\vs}_3$ and corresponding domain points $\bfxi_1,\ldots,\bfxi_{n_d}$ and condition number $\kappa$, let $F = \vs^\rmT \bfc$ with Hessian matrix (polynomial) $\mH$ and values $\bff = [F(\bfxi_1),\ldots,F(\bfxi_{n_d})]^\rmT$. Then
\[ \|\bff - \bfc\|_\infty \leq 2 \kappa h^2 \max_{\bfx \in \PSsmall} \| \mH(\bfx) \|_\infty. \]
\end{corollary}

\section{Smooth surface joins}\label{sec:SmoothnessConditions}
Let $\PS := [\bfp_1,\bfp_2,\bfp_3]$ and $\hat{\PS} := [\bfp_1, \bfp_2, \hat{\bfp}_3]$ be triangles sharing the edge $e := [\bfp_1, \bfp_2]$ and with 12-splits $\PSB$ and $\hat{\PSB}$. On these triangles we consider the spline spaces $\SS_3(\PSB)$ and $\SS_3(\hat{\PSB})$ with bases $\vs_3 = [S_{i,3}]_i$ and $\hat{\vs}_3 = [\hat{S}_{i,3}]_i$. Here $\hat{\vs}_3$ is the pull-back of $\vs_3$ under the affine map $\bfA: \hat{\PS}\longrightarrow \PS$ that maps $\hat{\bfp}_3$ to $\bfp_3$ and leaves $\bfp_1, \bfp_2$ invariant, i.e., $\hat{\vs}_3 = \vs_3 \circ \bfA$. In this section we derive conditions for smooth joins of 
\begin{equation}\label{eq:neighboringsimplexsplines}
F(\bfx) := \sum_{i=1}^{16} c_i S_{\sigma_i,3} (\bfx), \ \bfx\in \PS,\qquad
\hat{F}(\bfx) := \sum_{i=1}^{16} \hat{c}_i \hat{S}_{\sigma_i,3} (\bfx), \ \bfx\in \hat{\PS},
\end{equation}
where we used the reordering $i\mapsto \sigma_i$ (cf. Figure \ref{fig:basis-order}) defined by
\begin{equation}\label{eq:reordering}
\sigma = [\sigma_i]_{i=1}^{16} = [1,2,3,4,5,12,13,14,6,11,16,7,15,10,8,9].
\end{equation}

\begin{figure}
\includegraphics[scale=0.7]{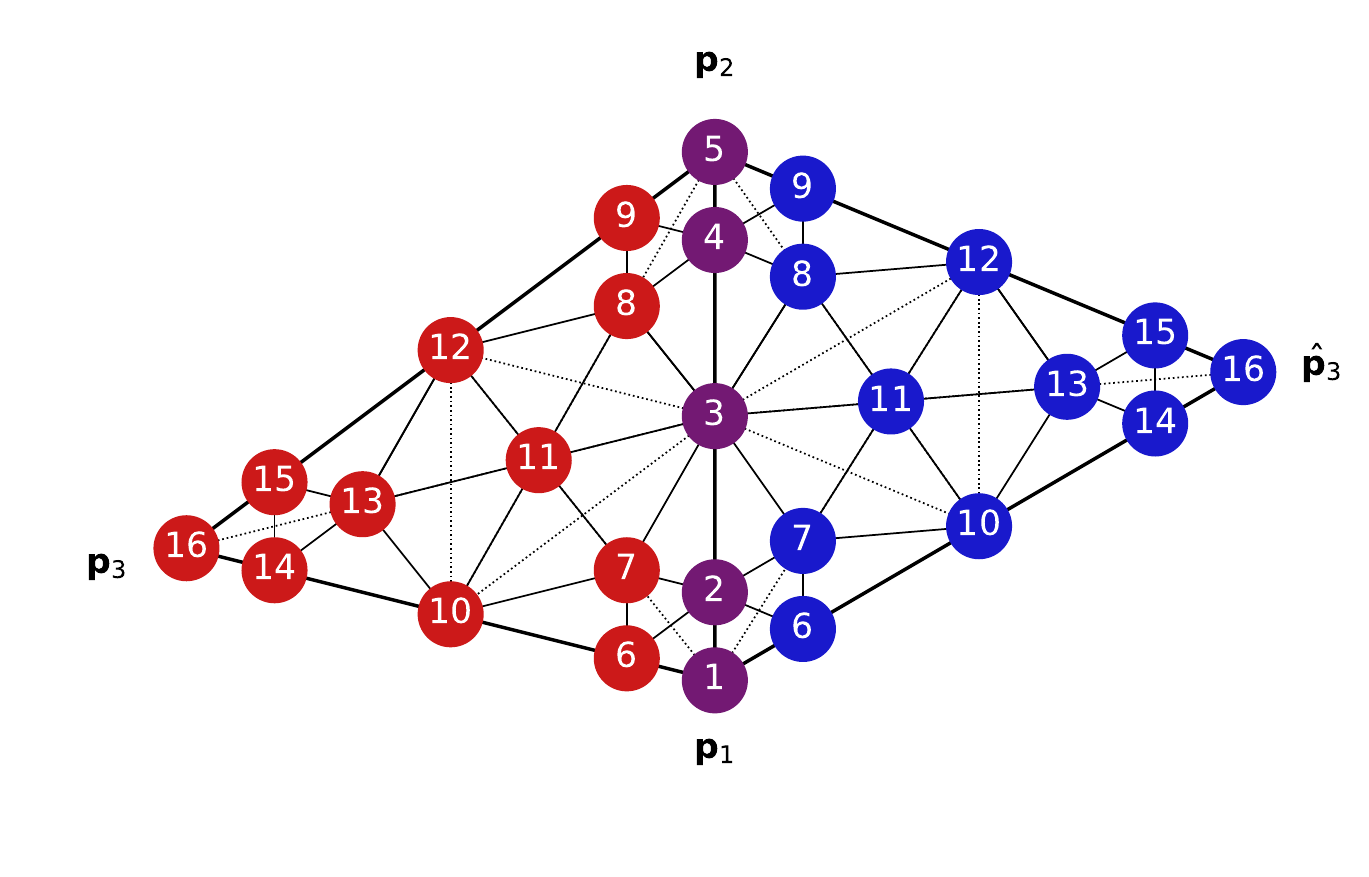}
\caption{Domain points of the reordered bases $\vs^\sigma_3 = [S_{\sigma_i, 3}]$ (left) and $\hat{\vs}^\sigma_3 = [\hat{S}_{\sigma_i, 3}]$ (right).}\label{fig:basis-order}
\end{figure}

\begin{remark}
The reordering \eqref{eq:reordering} is chosen such that the splines $S_{\sigma_i, 3}$ have an increasing number of knots outside of $e$. In particular, there is 1 such knot for $S_{\sigma_1, 3}, \ldots, S_{\sigma_5, 3}$, 2 knots for $S_{\sigma_6, 3}, \ldots, S_{\sigma_9, 3}$, 3 knots for $S_{\sigma_{10}, 3}, S_{\sigma_{11}, 3}, S_{\sigma_{12}, 3}$, and more than 3 knots for $S_{\sigma_{13}, 3}, \ldots, S_{\sigma_{16}, 3}$. By \eqref{eq:Qdiff}, this implies that after this reordering only the first 5 (resp. 5 + 4, resp. 5 + 4 + 3) splines in $\vs^\sigma_3$ are involved in the $C^0$ (resp. $C^1$, resp. $C^2$) conditions, as only these (resp. their derivatives, resp. their 2nd order derivatives) are not identically zero on $e$.
\end{remark}

Imposing a smooth join of $F$ and $\hat{F}$ along $e$ translates into B\'ezier-like linear relations among the ordinates $c_i$ and $\hat{c}_i$ (Property P8).

\begin{theorem}
Let $\beta_1,\beta_2,\beta_3$ be the barycentric coordinates of $\hat{\bfp}_3$ with respect to the triangle $\PS$. Then $F$ and $\hat{F}$ meet with
\noindent$C^0$-smoothness if and only if
\[\hat{c}_1 = c_1,\qquad \hat{c}_2 = c_2,\qquad \hat{c}_3 = c_3,\qquad \hat{c}_4 = c_4,\qquad \hat{c}_5 = c_5;\]
\noindent$C^1$-smoothness if and only if in addition
\[ \hat{c}_6 = \beta_1 c_1 + \beta_2 c_2 + \beta_3 c_6, \qquad
   \hat{c}_7 = \beta_1 c_2 + \beta_2 \frac{c_2 + c_3}{2} + \beta_3 c_7,\]
\[ \hat{c}_9 = \beta_1 c_4 + \beta_2 c_5 + \beta_3 c_9, \qquad
   \hat{c}_8 = \beta_2 c_4 + \beta_1 \frac{c_3 + c_4}{2} + \beta_3 c_8;\]
\noindent$C^2$-smoothness if and only if in addition
\begin{align*}
\hat{c}_{10} & = 2\beta_1\beta_2 \frac{3c_2 - c_1}{2} + 2\beta_1\beta_3 \frac{3c_6 - c_1}{2} + 2\beta_2\beta_3 \frac{4c_7 - c_2 - c_6}{2} + \beta_1^2 c_1 + \beta_2^2 c_3 + \beta_3^2 c_{10},\\
\hat{c}_{12} & = 2\beta_1\beta_2 \frac{3c_4 - c_5}{2} + 2\beta_2\beta_3 \frac{3c_9 - c_5}{2} + 2\beta_1\beta_3 \frac{4c_8 - c_4 - c_9}{2} + \beta_1^2 c_3 + \beta_2^2 c_5 + \beta_3^2 c_{12},\\
\hat{c}_{11} & = + 2\beta_1\beta_2\frac{c_1 - 2c_2 + 4c_3 - 2c_4 + c_5}{2} + \beta_3^2 c_{11}\\
               & \quad + 2\beta_1\beta_3 \frac{ c_1 - 2c_2 + c_3 - 3c_6 + 6 c_7 - 2c_8 + c_9}{2} + \beta_2^2(2c_4 - c_5) \\
               & \quad + 2\beta_2\beta_3 \frac{ c_3 - 2c_4 + c_5 - 3c_9 + 6 c_8 - 2c_7 + c_6}{2} +  \beta_1^2(2c_2 - c_1).
\end{align*}
\end{theorem}

\begin{proof}
By the barycentric nature of the statement, we can change coordinates by the linear affine map that sends $\bfp_1\longmapsto (0,0), \bfp_2\longmapsto (1,0)$, and $\bfp_3\longmapsto (0,1)$. In these coordinates, $\hat{\bfp}_3 = \beta_1(0,0) + \beta_2 (1,0) + \beta_3 (0,1) = (\beta_2, \beta_3)$. Let $\bfu := \hat{\bfp}_3 - \bfp_1 = (\beta_2, \beta_3)$. For $r = 0,1,2$, the splines $F$ and $\hat{F}$ meet with $C^r$-smoothness along $e$ if and only if $D^k_\bfu F(\cdot, 0) = D^k_\bfu \hat{F}(\cdot, 0)$ for $k = 0,\ldots,r$. Substituting \eqref{eq:neighboringsimplexsplines} this is equivalent to
\begin{equation}
\sum_{i=1}^{16} c_i D^k_\bfu S_{\sigma_i,3} (\cdot, 0) =
\sum_{i=1}^{16} \hat{c}_i D^k_\bfu \hat{S}_{\sigma_i,3} (\cdot, 0),\qquad k = 0,\ldots, r,
\end{equation}
which using Table \ref{tab:derivative_restrictions} reduces to a sparse system
\[ \sum_{j = 1}^{5-k} r_{kj} B_j^{3-k} = 0,\qquad k = 0,\ldots, r, \]
where $r_{kj}$ is a linear combination of the $c_i, \hat{c}_i$ with $i = 1,\ldots, n_{k+1}$, with $n_1 = 5$, $n_2 = 5+4$, and $n_3 = 5 + 4 + 3$. This system holds identically if and only if $r_{kj} = 0$ for $j = 1,\ldots, 5-k$ and $k=0,\ldots, r$. Let $n_0 = 0$. For $k = 0,1,2,3$ one solves for $\hat{c}_{n_k + 1}, \ldots, \hat{c}_{n_{k + 1}}$, each time eliminating the ordinates $\hat{c}_i$ that were previously obtained, resulting in the smoothness relations of the Theorem; see the Jupyter notebook for details.
\end{proof}

\begin{remark}
Since we forced our S-bases to restrict to B-spline bases on the boundary, the presented local bases can trivially be extended to global bases with $C^0$-smoothness across the macrotriangles.
\end{remark}

\begin{remark}
To obtain global $C^2$-smoothness for $d=3$, maximal sharing of the degrees of freedom is obtained by specifying values, first-order and second-order derivatives at the vertices of the coarse triangulation. This would amount to 18 degrees of freedom on a single macro triangle, exceeding the 16 available degrees of freedom. In particular, the degrees of freedom along the edges will be overdetermined. Hence global $C^2$-smoothness cannot be obtained in this way.
\end{remark}

\begin{remark}
For $d = 3$, it is not clear how to extend these local bases to bases with $C^1$-smoothness on triangulations. One strategy is to attempt to construct a dual basis that determines the value of the spline, as well as the value of its cross-boundary derivative, at the edges of the macro triangles, thus forcing $C^0$- and $C^1$-smoothness across these edges. On a single edge, this would require fixing 5 degrees of freedom for obtaining $C^0$-smoothness and an additional 4 degrees of freedom for attaining $C^1$-smoothness. Again maximal sharing is obtained by locating degrees of freedom (i.e., values and first-order derivatives) at the vertices of the macro triangles. For each edge, one would additionally need 1 degree of freedom for $C^0$ and 2 degrees of freedom for $C^1$, exceeding the remaining 7 of the available 16 degrees of freedom. Hence global $C^1$-smoothness on a general triangulation cannot be obtained in this way. Whether this is possible on specific triangulations, such as cells, is an open problem.
\end{remark}



\bibliographystyle{plain}

\bibliography{biblio}

\end{document}